\documentclass[11pt]{article}

\usepackage{mathtools}
\usepackage{color}
\usepackage{microtype}
\usepackage[usenames,dvipsnames,svgnames,table]{xcolor}
\usepackage{amsfonts,amsmath, amssymb,amsthm,amscd,mathrsfs}
\usepackage[height=9in,width=6.5in]{geometry}
\usepackage{verbatim}
\usepackage{tikz}
\usetikzlibrary{decorations.markings}
\tikzset{->-/.style={decoration={
			markings,
			mark=at position .6 with {\arrow{>}}},postaction={decorate}}}
\usepackage{tkz-euclide}
\usepackage[margin=10pt,font=small,labelfont=bf]{caption}
\usepackage{latexsym}
\usepackage[pdfauthor={ },bookmarksnumbered,hyperfigures,colorlinks=true,citecolor=BrickRed,linkcolor=BrickRed,urlcolor=BrickRed,pdfstartview=FitH]{hyperref}

\usepackage{graphicx}
\usepackage{enumitem}

\usepackage[inline,nolabel]{showlabels}

\usepackage[capitalize]{cleveref}
\Crefname{fact}{Fact}{Facts}
\crefname{theorem}{Theorem}{Theorems}
\crefname{thm}{Theorem}{Theorems}
\crefname{lemma}{Lemma}{Lemmas}
\crefname{claim}{Claim}{Claims}
\crefname{lem}{Lemma}{Lemmas}
\crefname{remark}{Remark}{Remarks}
\crefname{prop}{Proposition}{Propositions}
\crefname{defn}{Definition}{Definitions}
\crefname{corollary}{Corollary}{Corollaries}
\crefname{conjecture}{Conjecture}{Conjectures}
\crefname{question}{Question}{Questions}
\crefname{chapter}{Chapter}{Chapters}
\crefname{section}{Section}{Sections}
\crefname{part}{Part}{Parts}
\crefname{figure}{Figure}{Figures}

\newtheorem{theorem}{Theorem}[section]
\newtheorem{corollary}[theorem]{Corollary}
\newtheorem{lemma}[theorem]{Lemma}
\newtheorem{proposition}[theorem]{Proposition}
\newtheorem{question}[theorem]{Question}
\newtheorem{conj}[theorem]{Conjecture}
\newtheorem{claim}[theorem]{Claim}
\newtheorem{example}[theorem]{Example}

\newtheorem{definition}[theorem]{Definition}
\newtheorem{remark}[theorem]{Remark}

\numberwithin{equation}{section}

\newtheorem*{theorem*}{Theorem}
\newtheorem*{claim*}{Claim}
\newtheorem*{conj*}{Conjecture}
\newtheorem*{corollary*}{Corollary}
\newtheorem*{definition*}{Definition}
\newtheorem*{example*}{Example}
\newtheorem*{exercise*}{Exercise}
\newtheorem*{lemma*}{Lemma}
\newtheorem*{observation*}{Observation}
\newtheorem*{proposition*}{Proposition}
\newtheorem*{question*}{Question}
\newtheorem*{remark*}{Remark}

\newcommand{\bbC}{{\ensuremath{\mathbb C}} }

\newcommand{\bbE}{{\ensuremath{\mathbb E}} }

\newcommand{\bbP}{{\ensuremath{\mathbb P}} }

\newcommand{\bbR}{{\ensuremath{\mathbb R}} }

\newcommand{\kcc}{\chi_k}
\newcommand{\khcc}{\hat{\chi}_k}
\newcommand{\kticc}{\tilde{\chi}_k}

\newcommand{\kacc}{\chi_{k-1}}

\newcommand{\katicc}{\tilde{\chi}_{k-1}}
\newcommand{\kbcc}{\chi_{k-2}}

\newcommand{\kicc}{\chi_{j}}

\newcommand{\kiticc}{\tilde{\chi}_{j}}

\newcommand{\kjcc}{\chi_{j}}

\newcommand{\kjticc}{\tilde{\chi}_{j}}

\newcommand{\kajcc}{\chi_{k-1-j}}
\newcommand{\kajccr}{\chi_{r-1-j}}

\newcommand{\kajticc}{\tilde{\chi}_{k-1-j}}
\newcommand{\kajticcr}{\tilde{\chi}_{r-1-j}}

\newcommand{\Rtwotwosum}{\Xi}
\newcommand{\bC}{{\ensuremath{\mathbf C}} }

\newcommand{\bE}{{\ensuremath{\mathbf E}} }

\newcommand{\bP}{{\ensuremath{\mathbf P}} }

\newcommand{\cF}{{\ensuremath{\mathcal F}} }

\newcommand{\cJ}{{\ensuremath{\mathcal J}} }

\newcommand{\cN}{{\ensuremath{\mathcal N}} }

\newcommand{\sC}{{\ensuremath{\mathscr C}} }
\newcommand{\sD}{{\ensuremath{\mathscr D}} }

\newcommand{\sF}{{\ensuremath{\mathscr F}} }

\newcommand{\sT}{{\ensuremath{\mathscr T}} }

\DeclareMathSymbol{\leqslant}{\mathalpha}{AMSa}{"36} % nicer `smaller or equal'
\DeclareMathSymbol{\geqslant}{\mathalpha}{AMSa}{"3E} % nicer `larger or equal'
\DeclareMathSymbol{\eset}{\mathalpha}{AMSb}{"3F}     % nicer `emptyset'
\renewcommand{\le}{\;\leqslant\;}                   % redef. of < or =
\renewcommand{\ge}{\;\geqslant\;}                   % redef. of > or =
\renewcommand{\leq}{\;\leqslant\;}                   % redef. of < or =
\renewcommand{\geq}{\;\geqslant\;}                   % redef. of > or =

\newcommand{\be}{\begin{equation}}
	\newcommand{\ee}{\end{equation}}
\newcommand\ba{\begin{align}}
	\newcommand\ea{\end{align}}

\usepackage{nicematrix}
\usepackage{indentfirst}
\usepackage{float}
\usepackage{etoolbox}

\usepackage{bibunits}
\defaultbibliographystyle{plain}

\newcommand{\hO}{\mathbf{O}}

\newcommand{\UF}{\mathbb{P}_{\mathrm{UF}}}
\newcommand{\kUF}{\mathbb{P}_{\mathrm{kUF}}}
\newcommand{\UC}{\mathbb{P}_{\mathrm{UC}}}
\newcommand{\kUC}{\mathbb{P}_{\mathrm{kUC}}}
\newcommand{\twoUF}{\mathbb{P}_{\mathrm{2UF}}}
\newcommand{\kF}{\sF^{(k)}}
\newcommand{\oneF}{\sF^{(1)}}
\newcommand{\twoF}{\sF^{(2)}}
\newcommand{\threeF}{\sF^{(3)}}

\newcommand{\kC}{\sC^{(k)}}
\newcommand{\card}[1]{\left|#1\right|}
\newcommand{\ncc}[1]{\cN\left(#1\right)}

\newcommand{\resetpartcounters}{
    \setcounter{section}{0}
    \setcounter{subsection}{0}
    \setcounter{equation}{0}
    \setcounter{figure}{0}
    \setcounter{table}{0}
}

\newcommand{\papertitle}[1]{
    \begin{center}
        \vspace{1em}
        {\LARGE\bfseries #1}
        \vspace{1em}
    \end{center}
}

\newenvironment{paperabstract}
{
    \begin{center}
        \begin{minipage}{0.9\textwidth}
        \small
        \textbf{Abstract.}
}
{
        \end{minipage}
    \end{center}
    \vspace{0.5em}
}

\newcommand{\startsupplement}{
    \clearpage

    \setcounter{section}{0}
    \setcounter{subsection}{0}
    \setcounter{equation}{0}
    \setcounter{figure}{0}
    \setcounter{table}{0}

    \renewcommand{\thesection}{S\arabic{section}}
    \renewcommand{\thesubsection}{\thesection.\arabic{subsection}}
    \renewcommand{\theequation}{\thesection.\arabic{equation}}
    \renewcommand{\thefigure}{\thesection.\arabic{figure}}
    \renewcommand{\thetable}{\thesection.\arabic{table}}

}

\begin{document}

\author{
    Pengfei Tang
    \thanks{
        Center for Applied Mathematics and KL-AAGDM,
        Tianjin University, Tianjin, 300072, China.
        Email: \textsf{pengfei\_tang@tju.edu.cn}.
        Supported by the National Natural Science Foundation of China
        No.~12571151.
    }
    \qquad
    Zibo Zhang
    \thanks{
        School of Mathematics, Tianjin University,
        Tianjin, 300350, China.
        Email: \textsf{19932788533@163.com}.
    }
}

\date{\today}

\title{
    Pairwise Negative Correlation for Uniform Spanning\\
    Subgraphs of the Complete Graph
}

\maketitle

%%%%%%%%%%%%%%%%%%%%%%%%%%%%%%%%%%%%%%%%%%%%%%%%%%%%%%%%%%%%
%% General abstract
%%%%%%%%%%%%%%%%%%%%%%%%%%%%%%%%%%%%%%%%%%%%%%%%%%%%%%%%%%%%

\begin{abstract}
We study pairwise negative correlation for three families of
uniform spanning-subgraph measures on the complete graph.
In Part~I, we consider the uniform probability measure on connected
spanning subgraphs and prove pairwise negative correlation for all
sufficiently large complete graphs.
In Part~II, we study the uniform measure on spanning forests with a
prescribed number of connected components and prove pairwise negative
correlation for every fixed number of components when the number of
vertices is sufficiently large.
In Part~III, we consider connected spanning subgraphs with prescribed
excess and establish the analogous result for every fixed excess.
The three parts are self-contained and are intended as separate
manuscripts.
\end{abstract}

%%%%%%%%%%%%%%%%%%%%%%%%%%%%%%%%%%%%%%%%%%%%%%%%%%%%%%%%%%%%
%% Note on this version
%%%%%%%%%%%%%%%%%%%%%%%%%%%%%%%%%%%%%%%%%%%%%%%%%%%%%%%%%%%%

\section*{Note on this version}

This version reorganizes the original manuscript into three
self-contained parts, each of which is intended to be submitted
separately. The three parts concern, respectively, the uniform
connected spanning-subgraph measure, the uniform $k$-forest measure,
and the uniform $k$-excess connected spanning-subgraph measure on the
complete graph.

The mathematical arguments are essentially those of the original
version, with some corrections, reorganization, and additional
exposition. Each part contains its own introduction and bibliography
and may be read independently. The supplementary material at the end
of the present document contains the technical computations deferred
from Part~III.

\bigskip

\noindent
\textbf{Numbering convention.}
The numbering of sections, theorems, equations, figures, and tables
restarts in each of Parts~I--III, in order to keep the numbering
consistent with the corresponding stand-alone manuscripts.
The supplementary material uses the prefix~\textup{S}.

%%%%%%%%%%%%%%%%%%%%%%%%%%%%%%%%%%%%%%%%%%%%%%%%%%%%%%%%%%%%
%% Table of contents
%%%%%%%%%%%%%%%%%%%%%%%%%%%%%%%%%%%%%%%%%%%%%%%%%%%%%%%%%%%%

\tableofcontents

\clearpage

%%%%%%%%%%%%%%%%%%%%%%%%%%%%%%%%%%%%%%%%%%%%%%%%%%%%%%%%%%%%
%%%%%%%%%%%%%%%%%%%%%%%%%%%%%%%%%%%%%%%%%%%%%%%%%%%%%%%%%%%%
%% PART I: UC
%%%%%%%%%%%%%%%%%%%%%%%%%%%%%%%%%%%%%%%%%%%%%%%%%%%%%%%%%%%%
%%%%%%%%%%%%%%%%%%%%%%%%%%%%%%%%%%%%%%%%%%%%%%%%%%%%%%%%%%%%

\part{Uniform connected spanning subgraphs}

\resetpartcounters

\papertitle{
Negative Correlation for the Uniform Connected
Spanning Subgraphs of the Complete Graph
}

\begin{bibunit}[plain]

\begin{paperabstract}
Motivated by the conjectured negative dependence property of the random-cluster model with $q<1$, we investigate the pairwise negative correlation (p-NC) property for uniform probability measure on the  set of connected spanning subgraphs of the complete graph $K_n$. We prove that this measure satisfies the p-NC property for all $n$ sufficiently large. The main ingredient is an observation linking the p-NC property to  estimates of connection probabilities in random graph theory.
\end{paperabstract}

%% Insert the body of the UC paper here.
%% The file should start with \section{Introduction}
%% and end immediately before its bibliography if you want to
%% control the bibliography here.
%%
%% Recommended:

%%%%%%%%%%%%%%%%%%%%%%%%%%%%%%%%%%%%%%%%%%%%%%%%%%%%%%%%%%%%%%%%%%%%%%%%%%%%%%%

%\tableofcontents

%\begin{abstract}
%
%
% Motivated by the conjectured negative dependence property of the random-cluster model with $q<1$, we investigate the pairwise negative correlation (p-NC) property for uniform probability measure on the  set of connected spanning subgraphs of the complete graph $K_n$. We prove that this measure satisfies the p-NC property for all $n$ sufficiently large. The main ingredient is an observation linking the p-NC property to  estimates of connection probabilities in random graph theory.
%\end{abstract}

%%%%%%%%%%%%%%%%%%%%%%%%%%%%%%%%%%%%%%%%%%%%%%%%%%%%%%%%%%%%%%%%%%%%%%%%
%%%%%%%%%%%%%%%%%%%%  Introduction and main results %%%%%%%%%%%%%%%%%%%%
%%%%%%%%%%%%%%%%%%%%%%%%%%%%%%%%%%%%%%%%%%%%%%%%%%%%%%%%%%%%%%%%%%%%%%%%
\section{Introduction}

%%%%%%%%%    motivation  %%%%%%%%%
%\subsection{Motivation}

Given a finite graph $G=(V,E)$, for a configuration $\omega\in\Omega\coloneq \{0,1\}^E$, we denote by $\eta(\omega)=\{e\in E\colon \omega(e)=1\}$  the set of \textit{open} edges in $\omega$ and by $\ncc{\omega}$  the number of connected components of the subgraph $(V,\eta(\omega))$ of $G$. The famous random-cluster model, introduced by Fortuin and Kasteleyn, is defined on the state space $\Omega$ as follows. Let $\cF$ be the $\sigma$-field on $\Omega$ consisting of all  subsets of $\Omega$. The random cluster measure $\phi_{p,q}$ depends on two parameters $p\in[0,1]$ and $q\in(0,\infty)$, and is defined on $(\Omega,\cF)$ by
\[
\phi_{p,q}(\omega)=\frac{1}{Z}\left\{ \prod_{e\in E}p^{\omega(e)}(1-p)^{1-\omega(e)} \right\}q^{\ncc{\omega}}\,,
\]
where $Z=\sum_{\omega\in\Omega}\big\{ \prod_{e\in E}p^{\omega(e)}(1-p)^{1-\omega(e)} \big\}q^{\ncc{\omega}}$  is the normalizing constant. A classical reference for the random-cluster model is \cite{Grimmett2006}.  

A well-known fact is that for $q\ge1$, the  random-cluster measures $\phi=\phi_{p,q}$ satisfy the  positive lattice condition (PLC): for all $\omega_1,\omega_2\in\{0,1\}^E$,
\be\label{eq: PLC}
\phi(\omega_1\vee \omega_2) \phi(\omega_1\wedge \omega_2) \geq \phi(\omega_1)\phi(\omega_2) \,, \tag{PLC}
\ee
where $\omega_1\vee \omega_2$ and $\omega_1\wedge \omega_2$ denote the configurations given by $(\omega_1\vee \omega_2)(e)=\max\{  \omega_1(e),\omega_2(e)\}$ and $(\omega_1\wedge \omega_2)(e)=\min\{\omega_1(e),\omega_2(e)  \}$, respectively. 
Consequently, these measures satisfy the celebrated Fortuin--Kasteleyn--Ginibre (FKG) inequality \cite{Fortuin_Kasteleyn_Ginibre1971}:
\begin{theorem}[FKG]\label{thm: FKG}
	Let $\mu$ be a strictly positive probability measure on $\Omega\coloneq \{0,1\}^E$ satisfying \eqref{eq: PLC}. Then 
	\[
	\bbE_\mu(XY)\geq \bbE_\mu(X)\bbE_\mu(Y)
	\] 	
	for all increasing functions $X,Y: \Omega\to\bbR\,.$
\end{theorem}

On the other hand, for $q\in(0,1)$, the random-cluster measures $\phi_{p,q}$ satisfy the negative lattice condition (NLC): 
for all $\omega_1,\omega_2\in\{0,1\}^E$
\be\label{eq: NLC}
\phi(\omega_1\vee \omega_2) \phi(\omega_1\wedge \omega_2)  \leq   \phi(\omega_1)\phi(\omega_2) \,. \tag{NLC}
\ee
It is expected that such measures exhibit some form of negative dependence (see \cite[page~1374]{Pemantle2000} or \cite[Section 4]{Grimmett_Winkler2004}). One  weak version of negative dependence is the property of  \emph{pairwise negative correlation}  (p-NC). Here for a probability measure $\mu$ on $(\{0,1\}^E,\cF)$, this property requires that
\be\label{eq: def p-NC}
\mu\big[  \omega(e)=\omega(f)=1 \big]\leq \mu\big[\omega(e)=1\big]\cdot \mu\big[\omega(f)=1\big] %\tag*{p-NC}
\ee
for all distinct edges $e,f\in E$. However, condition \eqref{eq: NLC} is neither necessary nor sufficient for \eqref{eq: def p-NC}  (see \cite[Examples 2.1 and 2.2]{Borcea_Branden_Liggett2009}).
\begin{conj}[\cite{Kahn2000,Grimmett_Winkler2004}]\label{conj: p-NC for random cluster model}
	The random-cluster measures $\phi_{p,q}=\phi_{p,q}(G)$ on a finite graph $G=(V,E)$ with parameter $q\in(0,1)$ satisfy the p-NC property. 
\end{conj}

The  present work focuses on the p-NC property  for the uniform probability measure on the set of connected spanning  subgraphs of complete graphs. Our motivation  stems from the  conjectured validity of Conjecture~\ref{conj: p-NC for random cluster model}.  Indeed, uniform probability measures on spanning forests, connected spanning subgraphs, and spanning trees arise as limits of $\phi_{p,q}$ in various regimes as $q\to 0$ (see \cite[Section 1.5, Theorem 1.23]{Grimmett2006}), and are therefore natural candidates to inherit the conjectured p-NC property. In particular, the  p-NC property is well known for the uniform spanning tree measure \cite[Chapter 4]{LP2016}.

Progress towards \cref{conj: p-NC for random cluster model} has been rather limited. The following \cref{conj: NC for UF} can be viewed as a special case of \cref{conj: p-NC for random cluster model}.
\begin{conj}\label{conj: NC for UF}
	If $G$ is a finite connected graph and $\UF$ is the uniform probability measure on the set of spanning forests on $G$, then  $\UF$  satisfies the p-NC property. 
\end{conj}
Grimmett and Winkler \cite{Grimmett_Winkler2004} verified \cref{conj: NC for UF} numerically for graphs having eight or fewer vertices, or having nine vertices and no more than $18$ edges. 
Stark \cite{Stark2011} proved that \cref{conj: NC for UF} holds for complete graphs $K_n$ when $n$ is sufficiently large. Br\"and\'en and Huh \cite{Branden_Huh2020} proved an interesting related inequality valid for all finite connected graphs:
\[
\UF\big[  \omega(e)=\omega(f)=1 \big]\leq 2\cdot \UF\big[\omega(e)=1\big]\cdot \UF\big[\omega(f)=1\big] \,.
\]

 Little is known about the p-NC property of the limiting measure $\UC$ obtained from the random-cluster measure $\phi_{p,q}$ as $q\to0$ with $p=\frac{1}{2}$. 
Our main result provides, for the first time, a family of graphs for which the associated uniform measure $\UC$   satisfies the p-NC property. 
\begin{theorem}\label{thm: p-NC for UC}
	There exists a constant $N>0$ such that for all $n\ge N$, the uniform probability measure $\UC$ on the set of connected spanning subgraphs of the complete graph $K_n$ satisfies the p-NC property. 	
\end{theorem}

 The rest of the paper is organized as follows.
 
In Section~\ref{sec:observation} we present a key observation that reformulates the problem in terms of a percolation problem about connection probabilities (Lemma~\ref{lem: reformulation of p-NC for UC}).  Then in Section~\ref{sec:estimates} we  apply elementary results from random graph theory to establish the estimates of the connection probabilities. Specifically, we rely on the fact that the Bernoulli$(1/2)$ bond percolation configuration on the complete graph $K_n$ is connected with high probability, and that the probability of disconnection is dominated by the event that an isolated vertex exists. Finally we present the proof of Theorem~\ref{thm: p-NC for UC} in Section~\ref{sec: proof of thm}.

 For the reader's convenience, we list below the notation frequently used throughout the paper.

 \begin{enumerate}
 	\item $K_n$: the complete graph on $n$ vertices.
 	\item $G=(V,E)$: a finite connected graph.
 	\item $\omega\in\{0,1\}^E$: a configuration (assignment of open/closed edges).
 	\item $\eta(\omega)=\{e\in E: \omega(e)=1\}$: the set of open edges in $\omega$.
 	\item $\ncc{\omega}$: the number of connected components of the subgraph $(V,\eta(\omega))$.

 	\item $\sC=\sC(G)$: the set of connected spanning subgraphs of $G$.

 	\item $\UC$: the uniform probability measure on $\sC$.
 	
 	\item $\sC^{(e)} = \{\omega\in\sC: \omega(e)=1\}$ and $\sC^{(e,f)} = \{\omega\in\sC: \omega(e)=\omega(f)=1\}$.
 	\item $x_n = (1/2)^{n-1}$.
 	\item $E^{(k)}$, $E_S$, $A_S$: events defined in Definition~\ref{def: E^k, E_S and A_S}.

 	\item $f(n)\asymp g(n)$: there exist constants $c_1,c_2>0$ such that $c_1g(n)\leq f(n)\leq c_2g(n)$ for all $n$.
 	\item $f(n)=O(g(n))$: there exists a constant $c>0$ such that $f(n)\leq cg(n)$ for all $n$.
 	\item $f(n)=o(g(n))$: $\displaystyle\lim_{n\to\infty}\frac{f(n)}{g(n)}=0$.
 \end{enumerate}

%%%%%%%%%%%%%%%%%%%%  Introduction and main results %%%%%%%%%%%%%%%%%%%%
%%%%%%%%%%%%%%%%%%%%%%%%%%%%%%%%%%%%%%%%%%%%%%%%%%%%%%%%%%%%%%%%%%%%%%%%

%%%%%%%%%%%%%%%%%%%%%%%%%%%%%%%%%%%%%%%%%%%%%%%%%%%%%%%%%%%%%%%%%%%%%%%%%%%%%%%%%
%%%%%%%%%%%%%%%%%%%%%    Uniform connected subgraphs       %%%%%%%%%%%%%%%%%%%%%
%%%%%%%%%%%%%%%%%%%%%%%%%%%%%%%%%%%%%%%%%%%%%%%%%%%%%%%%%%%%%%%%%%%%%%%%%%%%%%%%% 
 
%\section{Uniform connected subgraphs of complete graphs} \label{sec: p-NC for UC}

\section{A key observation}\label{sec:observation}
 To prove Theorem~\ref{thm: p-NC for UC}, we begin with an observation that reformulates the problem in terms of Bernoulli percolation. Recall that Bernoulli$(p)$ bond percolation on a finite graph $G=(V,E)$ is the product measure on $\{0,1\}^E$ with Bernoulli$(p)$ marginals, i.e, $\bbP_p[\omega(e)=1]=p$ for every edge $e\in E$. Note also that the Bernoulli$(p)$ bond percolation coincides with the random-cluster measure $\phi_{p,q}$ measure when $q=1$.

\begin{lemma}\label{lem: reformulation of p-NC for UC}
	Let $\bbP_{1/2}$ be the Bernoulli$(1/2)$ bond percolation measure on a finite connected graph $G=(V,E)$. Then the p-NC property for $\UC$ on $G$ is equivalent to the following inequality:
	\begin{align}\label{eq: 2.1}
		&\bbP_{1/2}\big[ \omega(e)=\omega(f)=1,\omega\in\sC  \big] \cdot
		\bbP_{1/2}\big[ \omega\in\sC\big] \nonumber\\
		\leq &  \bbP_{1/2}\big[ \omega(e)=1, \omega\in\sC  \big]\cdot \bbP_{1/2}\big[ \omega(f)=1, \omega\in\sC  \big]\, \quad \forall \, e\neq f\,,
	\end{align}
	where $\{\omega\in\sC\}$ denotes the event that the subgraph $(V,\eta(\omega))$ is connected. 
\end{lemma}
\begin{proof}

	For a pair of distinct edges \(e, f \in E\), let us denote by
	\[
	\sC^{(e)} = \{ \omega \in \sC : \omega(e) = 1 \},\quad
	\sC^{(f)} = \{ \omega \in \sC : \omega(f) = 1 \},\quad
	\sC^{(e,f)} = \{ \omega \in \sC : \omega(e) = \omega(f) = 1 \}
	\]
	the subsets of \(\sC\) consisting of configurations that contain \(e\), contain \(f\), and contain both \(e\) and \(f\), respectively. Then the  inequality \eqref{eq: def p-NC} for the pair \((e,f)\) becomes simply
	\begin{equation}\label{eq: 2.1a}
	\frac{|\sC^{(e,f)}|}{|\sC|} \le \frac{|\sC^{(e)}|}{|\sC|} \cdot \frac{|\sC^{(f)}|}{|\sC|}.
	\end{equation}
	
	On the other hand, under the Bernoulli\((1/2)\) measure \(\mathbb{P}_{1/2}\), each configuration has equal probability, i.e.,
	\[
	\mathbb{P}_{1/2}(\omega) = \frac{1}{2^{|E|}} \qquad \text{for all } \omega \in \{0,1\}^E.
	\]
	Hence inequality \eqref{eq: 2.1} is equivalent to
	\[
	\frac{|\sC^{(e,f)}|}{2^{|E|}} \cdot \frac{|\sC|}{2^{|E|}} \le \frac{|\sC^{(e)}|}{2^{|E|}} \cdot \frac{|\sC^{(f)}|}{2^{|E|}},
	\]
	which clearly reduces to \eqref{eq: 2.1a}.
\end{proof}

Another key ingredient used in the proof of \cref{thm: p-NC for UC} is  that the Bernoulli$(1/2)$ bond percolation on the complete graph $K_n$ is connected with high probability. In fact, Bernoulli$(p)$ bond percolation on the complete graph $K_n$ is precisely the Erd\H{o}s--R\'enyi  random graph $G(n,p)$, and it is well known that  the threshold for connectedness  is $p=\frac{\log n}{n}$ \cite{ER1959}. For our case $p=\frac{1}{2}$, the random subgraph  is therefore connected with  probability tending to $1$. This allow us  to estimate the quantities appearing in \eqref{eq: 2.1} with sufficient accuracy and to  establish the inequality \eqref{eq: 2.1} for large $n$. 

 There is an extensive literature on the connectedness of  Erd\H{o}s--R\'enyi  graphs;  we refer the interested reader to \cite[Chapter 7]{Bollobas2001_2nd_edition} for background and references. The key idea we need, however, is quite simple---for $p=\frac{1}{2}$, the probability that the Erd\H{o}s--R\'enyi  graph $G(n,p)$  is \textit{disconnected} is approximately the probability that there exists an isolated vertex.

\section{Estimates of the disconnection probabilities}\label{sec:estimates}

For the remainder of the present paper, for simplicity we write $\sC$ to denote $\sC(K_n)$.
We first introduce several events that will be used repeatedly throughout this section.
\begin{definition}\label{def: E^k, E_S and A_S}
	\begin{enumerate}
		\item[(a)]	For an integer $k\in[1,n]$, let $E^{(k)}$ denote the event that the smallest connected component of the random graph $\big(V(K_n),\eta(\omega)\big)$ contains exactly $k$ vertices. 
		
		\item[(b)] For a subset $S\subset V(K_n)$, let $E_S$  denote the event that  $S$ is the vertex set of some connected component of the random graph $\big(V(K_n),\eta(\omega)\big)$. 
		
		\item[(c)]
		 For a subset $S\subset V(K_n)$, let  $A_S$ denote the event that there are no isolated vertices in $V(K_n)\setminus S$ for the random graph $\big(V(K_n),\eta(\omega)\big)$.  
	\end{enumerate}

\end{definition}
Note that 
\be\label{eq: disc as union}
\big\{ \omega\notin\sC  \big\}=\bigcup_{k=1}^{\lfloor \frac{n}{2} \rfloor }E^{(k)}\,
\ee
and by definition the events $E^{(k)}$ are pairwise disjoint for $k\in \big\{1,\ldots,\lfloor \frac{n}{2} \rfloor \big\}$.

The following lemma is a straightforward consequence of Bonferroni's inequalities.
\begin{lemma}\label{lem: Bonferroni}
	Let $m\ge2$ and let $B_1,B_2,\ldots,B_m$ be $m$ events. Set $S_1\coloneq \sum_{i=1}^{m}\bbP\big[B_i\big]$ and for $k\in\{2,\ldots,m\}$, define 
	\[
	S_k\coloneq \sum_{1\leq i_1<\cdots<i_k\leq m}\bbP\big[B_{i_1}\cap \cdots\cap B_{i_k} \big]\,.
	\]
	Then for every $k\in\{1,\ldots,m\}$, 
	\be\label{eq: Bonferroni}
	\bigg| \sum_{j=k}^{m} (-1)^{j-1}S_j \bigg|\leq S_k\,.
	\ee
\end{lemma}
\begin{proof}
	Bonferroni's inequalities state the following: 
	\begin{itemize}
		\item[(i)] if $k\in\{1,\ldots,m\}$ is odd, then  
		\[
		\bbP\big[B_1\cup \cdots \cup B_m\big]\leq \sum_{j=1}^{k} (-1)^{j-1}S_j\,. 
		\]
		
		\item[(ii)]  if $k\in\{1,\ldots,m\}$ is even, then  
		\[
		\bbP\big[B_1\cup \cdots \cup B_m\big]\geq \sum_{j=1}^{k} (-1)^{j-1}S_j\,. 
		\]
	\end{itemize} 
Since 
\[
\bbP\big[B_1\cup \cdots \cup B_m\big]=\sum_{j=1}^{m} (-1)^{j-1}S_j\,, 
\]
for $k\in\{1,\ldots,m-1\}$ we obtain
\[
\left\{
\begin{array}{ccc}
\sum_{j=k+1}^{m} (-1)^{j-1}S_j & \leq 0 & \text{ if } k \text{ is odd,} \\
&&\\
\sum_{j=k+1}^{m} (-1)^{j-1}S_j & \geq 0 & \text{ if } k \text{ is even.} 
\end{array}
\right.
\]
Shifting the index, this can be rewritten  for $k\in\{2,\ldots,m\}$ as
\be\label{eq: sign of alternating sum}
\left\{
\begin{array}{ccc}
\sum_{j=k}^{m} (-1)^{j-1}S_j & \leq 0 & \text{ if } k \text{ is even,} \\
&&\\
\sum_{j=k}^{m} (-1)^{j-1}S_j & \geq 0 & \text{ if } k \text{ is odd.} 
\end{array}
\right.
\ee
The case  $k=1$ of \eqref{eq: Bonferroni} follows directly from the $k=1$ case of  Bonferroni's inequality. Next  assume $k\geq2$. 
If $k$ is even, then we have
\begin{align*}
\bigg|\sum_{j=k}^{m} (-1)^{j-1}S_j \bigg|
&\stackrel{\eqref{eq: sign of alternating sum}}{=}-\sum_{j=k}^{m} (-1)^{j-1}S_j\\
&=S_k-\sum_{j=k+1}^{m}(-1)^{j-1}S_j\stackrel{\eqref{eq: sign of alternating sum}}{\leq }S_k\,.
\end{align*}
The case where $k$ is odd can be proved similarly.
\end{proof}

\begin{lemma}\label{lem: disconnected one}
	 Set $x_n=\big(\frac{1}{2}\big)^{n-1}$ and consider Bernoulli$(1/2)$ bond percolation on the complete graph $K_n$. Then  the following estimate holds:  
	\be\label{eq: disconnected 1}
	\bbP_{1/2}\big[ \omega\notin\sC\big] 
	=nx_n+o(1)\,x_n^2\,,
	\ee
	where the term $o(1)=o_n(1)$ denotes a quantity that tends to $0$ as $n\to \infty$.

\end{lemma}
\begin{proof}
By \eqref{eq: disc as union} and the disjointness of the events $E^{(k)}$, we have $\bbP_{1/2}\big[ \omega\notin\sC \big] =\sum_{k=1}^{\lfloor \frac{n}{2}\rfloor}\bbP_{1/2}\big[E^{(k)} \big]$. We shall establish the following three estimates:
	\be\label{eq: dis esti 1}
	\bbP_{1/2}\big[E^{(1)}\big]=nx_n-n(n-1)x_n^2+o(1)x_n^2\,,
	\ee
	\be\label{eq: dis esti 2}
	\bbP_{1/2}\big[E^{(2)}\big]=n(n-1)x_n^2+o(1)x_n^2\,,
	\ee
	and 
	\be\label{eq: dis esti 3}
	\sum_{k=3}^{\lfloor \frac{n}{2}\rfloor}\bbP_{1/2}\big[E^{(k)}\big]=o(1)x_n^2\,,
	\ee
	 The desired estimate \eqref{eq: disconnected 1} then follows immediately from \eqref{eq: dis esti 1}, \eqref{eq: dis esti 2} and \eqref{eq: dis esti 3}. 
	
	\noindent\textbf{Step 1: Establishing \eqref{eq: dis esti 1}.}\\
	List $V(K_n)=\{v_1,v_2,\ldots,v_n\}$. Note that $E^{(1)}=\bigcup_{j=1}^{n}E_{v_j}$. Using the inclusion-exclusion formula and the symmetry of $K_n$ we have 
	\begin{align*}
	\bbP_{1/2}\big[E^{(1)}\big]&=\bbP_{1/2}\bigg[ \bigcup_{j=1}^{n}E_{v_j} \bigg]
	=\sum_{m=1}^{n}\binom{n}{m}(-1)^{m-1}\bbP_{1/2}\bigg[ \bigcap_{j=1}^{m}E_{v_j} \bigg]\\
	&=\sum_{m=1}^{n}\binom{n}{m}(-1)^{m-1}\bigg(\frac{1}{2}\bigg)^{\binom{m}{2}+m(n-m)}\,.
	\end{align*}
	Applying Lemma~\ref{lem: Bonferroni} with $k=3$ yields the estimate
	\begin{align*}
	\bigg| \sum_{m=3}^{n}\binom{n}{m}(-1)^{m-1}\bigg(\frac{1}{2}\bigg)^{\binom{m}{2}+m(n-m)}\bigg|
	&\leq \binom{n}{3}\bigg(\frac{1}{2}\bigg)^{3+3(n-3)}\\
	&=\binom{n}{3}\bigg(\frac{1}{2}\bigg)^{n-4}\, x_n^2=o(1)x_n^2\,.
	\end{align*}

Plugging this estimate into the  inclusion-exclusion expansion  gives \eqref{eq: dis esti 1}:
	\begin{align*}
	\bbP_{1/2}\big[E^{(1)}\big]&=\sum_{m=1}^{n}\binom{n}{m}(-1)^{m-1}\bigg(\frac{1}{2}\bigg)^{\binom{m}{2}+m(n-m)}\\&=n\bigg(\frac{1}{2}\bigg)^{0+n-1}-\binom{n}{2}\bigg(\frac{1}{2}\bigg)^{1+2(n-2)}+o(1)x_n^2\\
	&=nx_n-n(n-1)x_n^2+o(1)x_n^2\,.
	\end{align*}

	\noindent\textbf{Step 2: Establishing \eqref{eq: dis esti 2}.}\\
	Recall the events $E_S$ and $A_S$ from Definition~\ref{def: E^k, E_S and A_S}. It is straightforward to verify that
	\[
	E^{(2)}=\bigcup_{\substack{ S\subset V(K_n), \\ \card{S}=2 }} \big(E_S\cap A_S\big)\,.
	\]
	Notice that if $S\neq S'$ and $S\cap S'\neq \emptyset$, then $E_S\cap E_{S'}=\emptyset$ . 
	Consequently, by the inclusion-exclusion formula we have
	\begin{align}\label{eq: dis esti 4}
	\bbP_{1/2}\big[E^{(2)}\big]
	=\sum_{S\subset V(K_n),\card{S}=2} \bbP_{1/2}\big[E_S\cap A_S\big]
	+\sum_{m=2}^{\lfloor \frac{n}{2}\rfloor} (-1)^{m-1} \sum_{\substack{S_1,S_2,\ldots,S_m\\ S_i\cap S_j=\emptyset, \card{S_i}=2}  } \bbP_{1/2}\Big[ \bigcap_{j=1}^{m}\big(E_{S_j}\cap A_{S_j}\big) \Big]\,,
	\end{align}
	where the  inner sum runs over unordered collections of pairwise-disjoint two-element subsets $S_1,\ldots,S_m$ of $V(K_n)$. The number of ways to choose such $m$ disjoint two-element sets $S_1,\ldots,S_m$ ($m\leq \lfloor \frac{n}{2}\rfloor$) from $V(K_n)$ is 
	\[
	\frac{1}{m!}\, \binom{n}{2}\, \binom{n-2}{2}\cdots \binom{n-2(m-1)}{2}
	=\frac{1}{m!}\frac{(2m)!}{2^m}\binom{n}{2m}\,. 
	\]
	Moreover for a fixed collection of pairwise-disjoint two-element sets $S_1,S_2,\ldots,S_m$, we have 
	\[
	\bbP_{1/2}\bigg[ \bigcap_{j=1}^{m}E_{S_j} \bigg]=\bigg(\frac{1}{2}\bigg)^{m}\times \bigg(\frac{1}{2}\bigg)^{4\binom{m}{2}}\times \bigg(\frac{1}{2}\bigg)^{2m(n-2m)}\,.
	\]
	The three factors correspond respectively to: the edges inside each $S_j$ being open; the edges between different $S_i$ and $S_j$ being closed; and the edges between the vertices in $\bigcup_{j=1}^{m} S_j$ and its complement being closed. 
	
We now bound the second summation in \eqref{eq: dis esti 4} as follows:
	\begin{align*}
		&\bigg|\sum_{m=2}^{\lfloor \frac{n}{2}\rfloor} (-1)^{m-1} \sum_{\substack{S_1,S_2,\ldots,S_m\\ S_i\cap S_j=\emptyset, \card{S_i}=2}  } \bbP_{1/2}\Big[ \bigcap_{j=1}^{m}\big(E_{S_j}\cap A_{S_j}\big) \Big]\bigg|\\
		&\leq \sum_{S_1\cap S_2=\emptyset, \card{S_1}=\card{S_2}=2}\bbP_{1/2}\Big[ (E_{S_1}\cap A_{S_1}) \cap (E_{S_2}\cap A_{S_2}) \Big] \tag*{ by Lemma~\ref{lem: Bonferroni}}\\
		&\leq \sum_{S_1\cap S_2=\emptyset, \card{S_1}=\card{S_2}=2}\bbP_{1/2}\Big[ E_{S_1}\cap E_{S_2} \Big]\\
		&=\frac{1}{2!}\frac{4!}{2^2}\binom{n}{4}\cdot \bigg(\frac{1}{2}\bigg)^{2}\times \bigg(\frac{1}{2}\bigg)^{4}\times \bigg(\frac{1}{2}\bigg)^{4(n-4)}=o(1)x_n^2\,.
	\end{align*}

	Plugging this bound into \eqref{eq: dis esti 4} we obtain that 
	\begin{align*}
	\bbP_{1/2}\big[E^{(2)}\big]
	&=\binom{n}{2}\bbP_{1/2}\big[E_{\{v_1,v_2\}}\cap A_{\{v_1,v_2\}}\big]+o(1)x_n^2\\
	&=\binom{n}{2}\bigg(\frac{1}{2}\bigg)^{1+2(n-2)}\times \bbP_{1/2}\big[A_{\{v_1,v_2\}} \mid E_{\{v_1,v_2\}}\big]+o(1)x_n^2\\
	&=n(n-1)x_n^2\times \bbP_{1/2}\big[A_{\{v_1,v_2\}} \mid E_{\{v_1,v_2\}}\big]+o(1)x_n^2\,.
	\end{align*}
	Observe that, conditioned on $E_{\{v_1,v_2\}}$, the complement of  $A_{\{v_1,v_2\}}$  has the same probability as the event $E^{(1)}$ in the complete graph induced by the remaining $n-2$ vertices. Hence applying estimate \eqref{eq: dis esti 1} to the case of $K_{n-2}$  we have
	\begin{align}\label{eq: A_e conditioned on E_e}
	\bbP_{1/2}\big[A_{\{v_1,v_2\}} \mid E_{\{v_1,v_2\}}\big]
	&=1-\big[(n-2)x_{n-2}-(n-2)(n-3)x_{n-2}^2+o(1)x_{n-2}^2\big] \nonumber\\
	&=1-O(nx_n)\,.
	\end{align}
Inserting this into the previous expression we obtain the desired estimate \eqref{eq: dis esti 2}:
	\begin{align*}
		\bbP_{1/2}\big[E^{(2)}\big]&=n(n-1)x_n^2\times \big[ 1-O(nx_n) \big]+o(1)x_n^2\\
		&=n(n-1)x_n^2+o(1)x_n^2\,.
	\end{align*}

\noindent\textbf{Step 3: Establishing \eqref{eq: dis esti 3}.}\\	
	We adapt an argument from the bottom of page~164 of \cite{Bollobas2001_2nd_edition}.  If a connected component of size $k$ exists, it contains a spanning tree on those $k$ vertices, and all edges joining this tree to the rest of the graph must be closed. Therefore for each $k\ge3$,
	\[
	\bbP_{1/2}\big[E^{(k)}\big]\leq\binom{n}{k}k^{k-2}\bigg(\frac{1}{2}\bigg)^{k-1}\times \bigg(\frac{1}{2}\bigg)^{k(n-k)}\,,
	\]
	where we use Cayley's formula which states that there are $k^{k-2}$ trees on $k$ labelled vertices. Consequently, 
	\[
	\sum_{k=3}^{\lfloor \frac{n}{2}\rfloor}\bbP_{1/2}\big[E^{(k)}\big]\leq \sum_{k=3}^{\lfloor \frac{n}{2}\rfloor} \binom{n}{k}k^{k-2}\bigg(\frac{1}{2}\bigg)^{k-1}\times \bigg(\frac{1}{2}\bigg)^{k(n-k)}\,.
	\]
	There exists a constant $C\in\bbR$ such that for  every integer $k$ with $4\leq k\leq \frac{n}{2}$,
	\[
	k\log k \leq \log 2 \cdot \big[k(n-k)-3(n-1)\big]+C\,.
	\] 
	 From this inequality we deduce,
	for  $4\leq k\leq \frac{n}{2}$,
	\[
	k^{k-2}\bigg(\frac{1}{2}\bigg)^{k(n-k)} \leq 	k^{k}\bigg(\frac{1}{2}\bigg)^{k(n-k)} 
	=\exp\bigg[ k\log k-\log 2\cdot k(n-k)  \bigg] \leq e^C\bigg(\frac{1}{2}\bigg)^{3(n-1)}\,.
	\]
	We now split the sum into the term $k=3$ and the rest. For $k=3$,
	\[
	\binom{n}{3}3^{3-2}\bigg(\frac{1}{2}\bigg)^{3-1}\times \bigg(\frac{1}{2}\bigg)^{3(n-3)}=\binom{n}{3}\frac{3}{4}\bigg(\frac{1}{2}\bigg)^{3(n-3)}\,.
	\]
	For $k\ge4$, using the bound above we have 
	\[
	\binom{n}{k}k^{k-2}\bigg(\frac{1}{2}\bigg)^{k-1}\times \bigg(\frac{1}{2}\bigg)^{k(n-k)}
	\leq \binom{n}{k}e^C\bigg(\frac{1}{2}\bigg)^{k-1+3(n-1)}\,.
	\]
	Then we can obtain \eqref{eq: dis esti 3}:
	\begin{align*}
		\sum_{k=3}^{\lfloor \frac{n}{2}\rfloor}\bbP_{1/2}\big[E^{(k)}\big]
		& \leq \binom{n}{3}\frac{3}{4}\times \bigg(\frac{1}{2}\bigg)^{3(n-3)}+2e^C \bigg(\frac{1}{2}\bigg)^{3(n-1)}\sum_{k=4}^{\lfloor \frac{n}{2}\rfloor} \binom{n}{k}\bigg(\frac{1}{2}\bigg)^{k}\\
		&\leq  \binom{n}{3}\frac{3}{4}\times \bigg(\frac{1}{2}\bigg)^{3(n-3)}+2e^C \bigg(\frac{1}{2}\bigg)^{3(n-1)}\bigg(1+\frac{1}{2}\bigg)^n\\
		&=o(1)x_n^2\,.  \qedhere
	\end{align*}

\end{proof}

\begin{lemma}\label{lem: disconnected two}
	Recall that $x_n=\big(\frac{1}{2}\big)^{n-1}$, and let  $e$ be an edge of $K_n$. Then we have the following estimate:  
	\be\label{eq: disconnected 5}
	\bbP_{1/2}\big[ \omega \notin \sC \mid \omega(e)=1\big] 
	=(n-2)x_n+4x_n^2+o(1)x_n^2
	\ee
\end{lemma}
\begin{proof}
	The proof of this lemma is pretty similar to Lemma~\ref{lem: disconnected one}. 
 Without loss of generality we assume $e=(v_{n-1},v_n)$. 
	
	\noindent\textbf{Step 1: Establishing the analogue of \eqref{eq: dis esti 1}.}\\
	Conditioned on $\omega(e)=1$, we have $\bbP_{1/2}\big[E_{v_{n-1}} \mid \omega(e)=1\big]=\bbP_{1/2}\big[E_{v_n} \mid \omega(e)=1\big]=0$. 
	Using  \eqref{eq: disc as union} together with the inclusion--exclusion formula, we obtain
	\begin{align*}
	\bbP_{1/2}\big[E^{(1)} \mid \omega(e)=1\big]
	&=\bbP_{1/2}\bigg[ \bigcup_{j=1}^{n-2}E_{v_j}\mid \omega(e)=1 \bigg]\\
	&=\sum_{m=1}^{n-2}\binom{n-2}{m}(-1)^{m-1}\bbP_{1/2}\bigg[\bigcap_{j=1}^{m}E_{v_j}  \mid \omega(e)=1\bigg]\\
	&=\sum_{m=1}^{n-2}\binom{n-2}{m}(-1)^{m-1}\bigg(\frac{1}{2}\bigg)^{\binom{m}{2}+m(n-m)}\,.
	\end{align*}
By Lemma~\ref{lem: Bonferroni}, the contribution of the terms with $m\ge 3$ is negligible:
\[
\bigg|\sum_{m=3}^{n-2}\binom{n-2}{m}(-1)^{m-1}\bigg(\frac{1}{2}\bigg)^{\binom{m}{2}+m(n-m)}  \bigg|
\leq \binom{n-2}{3}\bigg(\frac{1}{2}\bigg)^{3+3(n-3)}=o(1)x_n^2\,.
\]
Consequently, we obtain the desired analogue of \eqref{eq: dis esti 1}:
\begin{align}\label{eq:ana-dist esti 1}
\bbP_{1/2}\big[E^{(1)} \mid \omega(e)=1\big]
&=(n-2)\bigg(\frac{1}{2}\bigg)^{0+n-1}-\binom{n-2}{2}\bigg(\frac{1}{2}\bigg)^{1+2(n-2)}+o(1)x_n^2\nonumber\\
&=(n-2)x_n-(n-2)(n-3)x_n^2+o(1)x_n^2\,.
\end{align}

\noindent\textbf{Step 2: Establishing the analogue of \eqref{eq: dis esti 2}.}\\
Recall from the proof of Lemma~\ref{lem: disconnected one} that
\[
E^{(2)}=\bigcup_{\substack{ S\subset V(K_n), \\ \card{S}=2 }}\big(E_S\cap A_S\big)\,,
\]
and for two distinct $2$-element subsets $S$ and $S'$ with $S\cap S'\neq\emptyset$, we have 
$
E_S\cap E_{S'}=\emptyset\,.
$
Moreover for any $S$ such that $\card{S\cap\{v_{n-1},v_n\}}=1$, we have $\bbP_{1/2}\big[E_S\mid \omega(e)=1\big]=0$. To simplify notation, write $E_{e}$ and $A_e$ for  $E_{\{v_{n-1},v_n\}}$ and $A_{\{v_{n-1},v_n\}}$ respectively.

 Using the inclusion-exclusion formula and the observation above,  we obtain 
\begin{align}\label{eq: step 2 analogue two}
\bbP_{1/2}\big[E^{(2)} \mid \omega(e)=1\big]
&=\bbP_{1/2}[E_e\cap A_e\mid \omega(e)=1]+\sum_{ \substack{  S\subset V(K_n)\setminus\{v_{n-1} , v_n \}\\\card{S}=2  }}\bbP_{1/2}\big[E_S\cap A_S\mid \omega(e)=1\big]  \nonumber\\
&+\sum_{m=2}^{\lfloor \frac{n}{2}\rfloor} (-1)^{m-1} \sum_{\substack{S_1,S_2,\ldots,S_m\\ S_i\cap S_j=\emptyset, \card{S_i}=2}  } \bbP_{1/2}\Big[ \bigcap_{j=1}^{m}\big(E_{S_j}\cap A_{S_j}\big) \mid \omega(e)=1 \Big]\,.
\end{align}  
Similar to the last summation in \eqref{eq: dis esti 4},  the last summation in \eqref{eq: step 2 analogue two} is negligible and its proof proceeds as follows. First applying Lemma~\ref{lem: Bonferroni} to bound it, 
\begin{align*}
& \bigg| \sum_{m=2}^{\lfloor \frac{n}{2}\rfloor} (-1)^{m-1} \sum_{\substack{S_1,S_2,\ldots,S_m\\ S_i\cap S_j=\emptyset, \card{S_i}=2}  } \bbP_{1/2}\Big[ \bigcap_{j=1}^{m}\big(E_{S_j}\cap A_{S_j}\big) \mid \omega(e)=1 \Big]\bigg|\\ 
&\leq \sum_{S_1\cap S_2=\emptyset, \card{S_1}=\card{S_2}=2}\bbP_{1/2}\Big[ (E_{S_1}\cap A_{S_1}) \cap (E_{S_2}\cap A_{S_2}) \mid \omega(e)=1 \Big] \\
&\leq \sum_{S_1\cap S_2=\emptyset, \card{S_1}=\card{S_2}=2}\bbP_{1/2}\Big[ E_{S_1}\cap E_{S_2} \mid \omega(e)=1 \Big]\,.
\end{align*}
For disjoint $2$-element subsets $S_1, S_2$, we distinguish three cases:
\begin{itemize}
	\item[(i)] If $S_1,S_2\subset V(K_n)\setminus\{v_{n-1},v_n\}$, then 
	\[
	\bbP_{1/2}\Big[  E_{S_1}\cap E_{S_2} \mid \omega(e)=1 \Big]
	=\bbP_{1/2}\Big[  E_{S_1}\cap E_{S_2} \Big]\,.
	\]
	\item[(ii)] If $\card{S_j\cap \{v_{n-1},v_n\} }=1$ for some  $j\in\{1,2\}$, then $\bbP_{1/2}\Big[  E_{S_1}\cap E_{S_2} \mid \omega(e)=1 \Big]=0$.
	\item[(iii)] If $S_j=\{v_{n-1},v_n\}$ for some  $j\in\{1,2\}$, then   conditioning on $\omega(e)=1$ doubles the probability:
	\[
	\bbP_{1/2}\Big[  E_{S_1}\cap E_{S_2} \mid \omega(e)=1 \Big]
	=2\bbP_{1/2}\Big[  E_{S_1}\cap E_{S_2} \Big]\,.
	\]
\end{itemize}
Consequently, 
\[
\sum_{S_1\cap S_2=\emptyset, \card{S_1}=\card{S_2}=2}\bbP_{1/2}\Big[ E_{S_1}\cap E_{S_2} \mid \omega(e)=1 \Big] \\
\leq 2\sum_{S_1\cap S_2=\emptyset, \card{S_1}=\card{S_2}=2}\bbP_{1/2}\Big[ E_{S_1}\cap E_{S_2} \Big]\\
=o(1)x_n^2\,,
\]
where the last estimate comes from  the proof of Lemma~\ref{lem: disconnected one}.
Therefore we obtain a good upper bound of the last summation in \eqref{eq: step 2 analogue two}:
\begin{align*}
& \bigg| \sum_{m=2}^{\lfloor \frac{n}{2}\rfloor} (-1)^{m-1} \sum_{\substack{S_1,S_2,\ldots,S_m\\ S_i\cap S_j=\emptyset, \card{S_i}=2}  } \bbP_{1/2}\Big[ \bigcap_{j=1}^{m}\big(E_{S_j}\cap A_{S_j}\big) \mid \omega(e)=1 \Big]\bigg|\\ 
&\leq \sum_{S_1\cap S_2=\emptyset, \card{S_1}=\card{S_2}=2}\bbP_{1/2}\Big[ E_{S_1}\cap E_{S_2} \mid \omega(e)=1 \Big] \\
&=o(1)x_n^2\,.
\end{align*}

Now evaluate the first term in \eqref{eq: step 2 analogue two}.  Since $E_e\subset \{ \omega(e)=1 \}$, we have
\begin{align*}
\bbP_{1/2}[E_e\cap A_e\mid \omega(e)=1]
&=\bbP_{1/2}\big[E_e\mid \omega(e)=1 \big]\cdot \bbP_{1/2} \big[A_e\mid E_e\big]\\
&=\bigg(\frac{1}{2}\bigg)^{2(n-2)}\times \bbP_{1/2}\big[A_e\mid E_e\big]\stackrel{\eqref{eq: A_e conditioned on E_e}}{=}4x_n^2\big[1-O(nx_n)\big]\,.
\end{align*}
Finally we turn to the second term on the right-hand side of \eqref{eq: step 2 analogue two}. Fix an arbitrary subset  $S\subset V(K_n)\setminus\{ v_{n-1},v_n \}$ with $\card{S}=2$. Observe that $A_S$ and $\{\omega(e)=1\}$ are increasing events that depend  only on the status of  edges not incident to $S$. Applying the FKG inequality to these two events under  the conditional measure $\bbP_{1/2}\big[\cdot \mid E_S\big]$ yields 
\[
\bbP_{1/2}\big[A_S,\omega(e)=1\mid E_S\big]\geq \bbP_{1/2}\big[A_S\mid E_S\big]\cdot 
\bbP_{1/2}\big[\omega(e)=1\mid E_S\big]\,.
\]
Consequently,
\begin{align*}
1&\geq \bbP_{1/2}\big[A_S\mid E_S,\omega(e)=1\big]=\frac{\bbP_{1/2}\big[ A_S,\omega(e)=1\mid E_S \big]}{ \bbP_{1/2}\big[\omega(e)=1\mid E_S\big] }\\
&\geq \bbP_{1/2}\big[A_S\mid E_S\big]\stackrel{\eqref{eq: A_e conditioned on E_e}}{=}1-O(nx_n)\,.
\end{align*}
Therefore
\begin{align*}
\bbP_{1/2}\big[E_S\cap A_S\mid \omega(e)=1\big]
&=\bbP_{1/2}\big[E_S\mid \omega(e)=1\big]\cdot \bbP_{1/2}\big[A_S\mid E_S,\omega(e)=1\big]\\
&=\bbP_{1/2}\big[E_S\big]\cdot \big[1-O(nx_n)\big]\\
&=\bigg(\frac{1}{2}\bigg)^{1+2(n-2)}\big[1-O(nx_n)\big]=2x_n^2\big[1-O(nx_n)\big]\,,
\end{align*}
where in the second equality we use the fact that $E_S$ and $\{\omega(e)=1\}$ are independent.

Inserting these estimates into \eqref{eq: step 2 analogue two} yields the desired analogue of \eqref{eq: dis esti 2}:
\be\label{eq:ana-dis esti 2}
\bbP_{1/2}\big[E^{(2)} \mid \omega(e)=1\big]=4x_n^2+\binom{n-2}{2}2x_n^2+o(1)x_n^2=\big[(n-2)(n-3)+4\big]x_n^2+o(1)x_n^2\,.
\ee

\noindent\textbf{Step 3: Establishing the analogue of \eqref{eq: dis esti 3}.}\\
Conditioned on $\omega(e)=1$, a disconnected component of size $k$ (with $3 \le k \le \frac{n}{2}$) can arise in two mutually exclusive ways: either the two endpoints of $e$ belong to the same component of size $k$, or there exists some component of size $k$ that does not contain either endpoint of $e$. Paralleling the derivation of \eqref{eq: dis esti 3}, we obtain the bound
\[
\sum_{k=3}^{\lfloor \frac{n}{2} \rfloor }\bbP_{1/2}\big[E^{(k)} \mid \omega(e)=1\big]
\leq \sum_{k=3}^{\lfloor \frac{n}{2} \rfloor } \left[\binom{n-2}{k-2}k^{k-2}\bigg(\frac{1}{2}\bigg)^{k-2+k(n-k)}+\binom{n-2}{k}k^{k-2}\bigg(\frac{1}{2}\bigg)^{k-1+k(n-k)}\right]\,.
\]

Repeating the asymptotic analysis that led to \eqref{eq: dis esti 3} shows that each term in the sum is of order $o(1)x_n^2$, and therefore
\be\label{eq:ana-dis esti 3}
\sum_{k=3}^{\lfloor \frac{n}{2} \rfloor }\bbP_{1/2}\big[E^{(k)} \mid \omega(e)=1\big]
=o(1)x_n^2\,.
\ee
Combining the three analogues \eqref{eq:ana-dist esti 1}, \eqref{eq:ana-dis esti 2} and \eqref{eq:ana-dis esti 3} we obtain \eqref{eq: disconnected 5}. 
\end{proof}

Our goal now is to estimate the probability
\[
\bbP_{1/2}\big[ \omega \notin \sC \mid \omega(e)=\omega(f)=1\big]\,.
\]
We first treat separately the case where the two edges $e$ and $f$ do not share a vertex.

\begin{lemma}\label{lem: disconnected three}
	Recall that $x_n=\big(\frac{1}{2}\big)^{n-1}$, and let  $e,f$ be two non-adjacent edges  of $K_n$ (i.e., they share no common vertex). Then we have the following estimate:  
	\be\label{eq: disconnected 6}
	\bbP_{1/2}\big[ \omega \notin \sC \mid \omega(e)=\omega(f)=1\big] 
	=(n-4)x_n+8x_n^2+o(1)x_n^2
	\ee
\end{lemma}
\begin{proof}

	The argument follows the same lines as the proofs of Lemmas~\ref{lem: disconnected one} and \ref{lem: disconnected two}. Without loss of generality we assume $n\ge 4$ and set $e=(v_{n-1},v_n)$ and $f=(v_{n-3},v_{n-2})$; these two edges are disjoint. We establish the analogues of \eqref{eq: dis esti 1}, \eqref{eq: dis esti 2} and \eqref{eq: dis esti 3} under the conditioning $\omega(e)=\omega(f)=1$. As the estimates closely follow those in the proofs of the two lemmas, we omit the routine details.
	
	\noindent\textbf{Step 1: Isolated vertices (analogue of \eqref{eq: dis esti 1}).}\\
	Conditioned on $\omega(e)=\omega(f)=1$, a vertex can be isolated only if it belongs to ${v_1,\dots,v_{n-4}}$. Hence
	\begin{align*}
	\bbP_{1/2}\big[E^{(1)} \mid \omega(e)=\omega(f)=1\big]
	&=\bbP_{1/2}\bigg[ \bigcup_{j=1}^{n-4}E_{v_j}\mid \omega(e)=\omega(f)=1 \bigg]\\
	&=\sum_{m=1}^{n-4}\binom{n-4}{m}(-1)^{m-1}\bbP_{1/2}\bigg[\bigcap_{j=1}^{m}E_{v_j}  \mid \omega(e)=1,\omega(f)=1\bigg]\\
	&=\sum_{m=1}^{n-4}\binom{n-4}{m}(-1)^{m-1}\bigg(\frac{1}{2}\bigg)^{\binom{m}{2}+m(n-m)}\\
	&=(n-4)x_n-(n-4)(n-5)x_n^2+o(1)x_n^2\,.
	\end{align*}
	
	\noindent\textbf{Step 2: Components of size $2$ (analogue of \eqref{eq: dis esti 2}).}\\
	Write $E_e,A_e$ for $E_{{v_{n-1},v_n}},A_{{v_{n-1},v_n}}$ and $E_f,A_f$ for $E_{{v_{n-3},v_{n-2}}},A_{{v_{n-3},v_{n-2}}}$. Analogously to \eqref{eq: step 2 analogue two} we have
	\begin{align}\label{eq: step 2 analogue three}
	\bbP_{1/2}\big[E^{(2)} \mid \omega(e)=\omega(f)=1\big]
	&=\bbP_{1/2}[E_{e}\cap A_e\mid \omega(e)=\omega(f)=1]+\bbP_{1/2}[E_{f}\cap A_f\mid \omega(e)=\omega(f)=1] \nonumber\\
	&\quad +\sum_{ \substack{ S\subset V(K_n),\card{S}=2 \\ v_j\notin S,j=n-3,\ldots,n}}\bbP_{1/2}\big[E_S\cap A_S\mid \omega(e)=\omega(f)=1\big] \nonumber\\
	&\quad +\sum_{m=2}^{\lfloor \frac{n}{2}\rfloor} (-1)^{m-1} \sum_{\substack{S_1,S_2,\ldots,S_m\\ S_i\cap S_j=\emptyset, \card{S_i}=2}  } \bbP_{1/2}\Big[ \bigcap_{j=1}^{m}\big(E_{S_j}\cap A_{S_j}\big) \mid \omega(e)=\omega(f)=1 \Big]\,.
	\end{align}  
	The same analysis as in Lemma~\ref{lem: disconnected two} yields the following three facts:
		\begin{itemize}
		\item[(i)] The last sum in  \eqref{eq: step 2 analogue three} is  $o(1)x_n^2$:

		\item[(ii)] 
	
		$
		\bbP_{1/2}[E_{e}\cap A_e\mid \omega(e)=\omega(f)=1]=\bbP_{1/2}[E_{f}\cap A_f\mid \omega(e)=\omega(f)=1]=4x_n^2\big[1-O(nx_n)\big]\,.
		$
		
		\item[(iii)] 
	For $S\subset V(K_n)\setminus \{v_{n-3},v_{n-2},v_{n-1},v_n  \}$ with $\card{S}=2$, we have
		\[
		\bbP_{1/2}\big[E_S\cap A_S\mid \omega(e)=\omega(f)=1\big]=2x_n^2\big[1-O(nx_n)\big]\,.
		\]
	\end{itemize}
Consequently we obtain the analogue of \eqref{eq: dis esti 2}:
	\[
\bbP_{1/2}\big[E^{(2)} \mid \omega(e)=\omega(f)=1\big]= \big[(n-4)(n-5)+8\big]x_n^2+o(1)x_n^2\,.
\]

\noindent\textbf{Step 3: Components of size $k\ge 3$ (analogue of \eqref{eq: dis esti 3}).}\\
If $E^{(k)}$ occurs, there exists a set $S$ with $|S|=k$ for which $E_S$ holds. By the union bound,
\[
\sum_{k=3}^{\lfloor \frac{n}{2} \rfloor }\bbP_{1/2}\big[E^{(k)} \mid \omega(e)=\omega(f)=1\big]
\leq \sum_{k=3}^{\lfloor \frac{n}{2} \rfloor }\,\sum_{\substack{  S\subset V(K_n)\\ \card{S}=k  } }\bbP_{1/2}\big[E_S\mid \omega(e)=\omega(f)=1\big]\,.
\]
The subsequent case-by-case analysis enables us to bound the terms in the above summation.
\begin{itemize}
	\item If $\card{S\cap \{v_{n-1},v_n  \}}=1$ or $\card{S\cap \{v_{n-3},v_{n-2}  \}}=1$, then  $\bbP_{1/2}\big[E_S\mid \omega(e)=\omega(f)=1\big]=0$;
	
	\item If $\card{S\cap \{v_{n-1},v_n  \}}\in\{0,2\}$ and $\card{S\cap \{v_{n-3},v_{n-2}  \}}\in\{0,2\}$, then 
		\[
	\bbP_{1/2}\big[E_S\mid \omega(e)=\omega(f)=1\big]\leq 4\bbP_{1/2}\big[E_S\big]\leq 4k^{k-2}\bigg(\frac{1}{2}\bigg)^{k-3+k(n-k)}\,,
	\]
	because a spanning tree of the component on $S$ can be chosen in at  most $k^{k-2}$ ways, at least $k-3$ of its edges are distinct from $e$ and $f$ and  must be open, and all  edges connecting $S$ to its complement must be closed.

\end{itemize}

	Therefore we have the corresponding analogue of \eqref{eq: dis esti 3}:
\[
\sum_{k=3}^{\lfloor \frac{n}{2} \rfloor }\bbP_{1/2}\big[E^{(k)} \mid \omega(e)=\omega(f)=1\big]
\leq 4\sum_{k=3}^{\lfloor \frac{n}{2} \rfloor } \binom{n}{k}k^{k-2}\bigg(\frac{1}{2}\bigg)^{k-3+k(n-k)}=o(1)x_n^2\,,
\]
where the last equality follows from the estimates used in the proof of \eqref{eq: dis esti 3}.

Combining these three analogues  yields precisely \eqref{eq: disconnected 6}.
\end{proof}

The situation where $e$ and $f$ share a vertex is simpler, as we do not need to determine the $x_n^2$ coefficient with the precision required in Lemma~\ref{lem: disconnected three}.

\begin{lemma}\label{lem: disconnected four}
	Recall that $x_n=\big(\frac{1}{2}\big)^{n-1}$ and let $e,f$ be two adjacent edges  of $K_n$. Then we have the following estimates:  
	\be\label{eq: disconnected 7}
	\bbP_{1/2}\big[ \omega \notin \sC\mid \omega(e)=\omega(f)=1\big] 
	=(n-3)x_n+o(1)x_n\,.
	\ee
\end{lemma}
\begin{proof}
	The proof is similar to that of Lemma~\ref{lem: disconnected one} and \ref{lem: disconnected two} , but simpler because we only require an estimate up to order $x_n$.
	Without loss of generality, assume $e=(v_{n-1},v_n)$ and $f=(v_{n-2},v_n)$; these two edges share the common vertex $v_n$. 
	
	\noindent\textbf{Step 1: Isolated vertices.}\\
	Conditioned on $\omega(e)=\omega(f)=1$, an isolated vertex can only lie in ${v_1,\dots,v_{n-3}}$. Hence
	\begin{align*}
	\bbP_{1/2}\big[E^{(1)} \mid \omega(e)=\omega(f)=1 \big]&=\bbP_{1/2}\bigg[\bigcup_{j=1}^{n-3}E_{v_j} \mid \omega(e)=\omega(f)=1\bigg]\\
	&=\sum_{m=1}^{n-3}\binom{n-3}{m}(-1)^{m-1}\bigg(\frac{1}{2}\bigg)^{\binom{m}{2}+m(n-m)}\\
	&=(n-3)x_n+o(1)x_n\,.
	\end{align*}
	
	\noindent\textbf{Step 2: Components of size $2$.}\\
	Conditioned on $\omega(e)=\omega(f)=1$, if $E^{(2)}$ occurs, then there must exist a $2$-element subset $S\subset V(K_n)\setminus\{v_{n-2},v_{n-1},v_n\}$ such that $E_S$ holds. (The two vertices of $S$ form an isolated component of size $2$.) Applying the union bound,
	\[
	\bbP_{1/2}\big[E^{(2)} \mid \omega(e)=\omega(f)=1 \big]\leq \binom{n-3}{2}\bigg(\frac{1}{2}\bigg)^{1+2(n-2)}=o(1)x_n\,.
	\] 
	
	\noindent\textbf{Step 3: Components of size $k\ge 3$.}
	Exactly the same reasoning as in the proof of Lemma~\ref{lem: disconnected three} shows that
	\[
	\sum_{k=3}^{\lfloor \frac{n}{2} \rfloor }\bbP_{1/2}\big[E^{(k)} \mid \omega(e)=\omega(f)=1 \big]=o(1)x_n\,.
	\]
	
	Combining the above estimates we have \eqref{eq: disconnected 7}. 
\end{proof}
\begin{remark*}
	Of course, one could refine the argument as in Lemma~\ref{lem: disconnected three} to obtain the more precise estimate
	\[
		\bbP_{1/2}\big[ \omega \notin \sC \mid \omega(e)=\omega(f)=1\big] 
	=(n-3)x_n+o(1)x_n^2\,.
	\]
However, the coarser estimate in \eqref{eq: disconnected 7} is already sufficient for the proof of Theorem~\ref{thm: p-NC for UC}.
\end{remark*}

\section{Proof of Theorem~\ref{thm: p-NC for UC}}\label{sec: proof of thm}

\begin{proof}[Proof of Theorem~\ref{thm: p-NC for UC}]
	
 First we have $\bbP_{1/2}\big[ \omega(e)=1,\omega\in\sC  \big]=\bbP_{1/2}\big[ \omega(f)=1,\omega\in\sC  \big] $ by the edge-transitivity of $K_n$. 
In view of Lemma~\ref{lem: reformulation of p-NC for UC}, it suffices to estimate the quantities $\bbP_{1/2}\big[ \omega(e)=\omega(f)=1,\omega\in\sC  \big]$, $\bbP_{1/2}\big[ \omega(e)=1,\omega\in\sC  \big]$ and $\bbP_{1/2}\big[ \omega\in\sC \big]$. For  $\bbP_{1/2}\big[ \omega(e)=\omega(f)=1,\omega\in\sC  \big]$, we use the following evaluation procedure:
\begin{align*}
	\bbP_{1/2}\big[ \omega(e)=\omega(f)=1,\omega\in\sC  \big]
	&=\bbP_{1/2}\big[ \omega(e)=\omega(f)=1\big]\bbP_{1/2}\big[\omega\in\sC \mid \omega(e)=\omega(f)=1  \big]\\
	&=\frac{1}{4}\left(1- \bbP_{1/2}\big[\omega\notin\sC \mid \omega(e)=\omega(f)=1  \big]\right)\,.
\end{align*}
Similarly, 
\[
\bbP_{1/2}\big[ \omega(e)=1,\omega\in\sC  \big]
=\frac{1}{2}\left(1- \bbP_{1/2}\big[\omega\notin\sC \mid \omega(e)=1  \big]\right)
\]
and
\[
\bbP_{1/2}\big[ \omega\in\sC  \big] =1-\bbP_{1/2}\big[ \omega\notin\sC  \big]\,.
\]

\noindent
\textbf{The adjacent case:}\\
We first consider the case $e$ and $f$ being adjacent. 
As  noted earlier, the dominant contribution to   $\bbP_{1/2}\big[ \omega\notin\sC \big]$ comes from the event that  an isolated vertex exists. Specifically, by Lemma~\ref{lem: disconnected one} we have
\[
\bbP_{1/2}\big[ \omega\notin\sC  \big]=nx_n+o(1)\cdot x_n\,,
\]
where $x_n=\big(\frac{1}{2}\big)^{n-1}$. 
Similarly, conditioning on $\omega(e)=1$ leaves only the remaining $n-2$ vertices as possible isolated vertices, leading to (see Lemma~\ref{lem: disconnected two})
\[
\bbP_{1/2}\big[ \omega\notin\sC \mid \omega(e)=1  \big]=(n-2)x_n+o(1)\cdot x_n\,.
\]
For adjacent $e$ and $f$, by Lemma~\ref{lem: disconnected four} we have 
\[
\bbP_{1/2}\big[ \omega\notin\sC \mid \omega(e)=\omega(f)=1  \big]=(n-3)x_n+o(1)\cdot x_n\,.
\]
These estimates together with Lemma~\ref{lem: reformulation of p-NC for UC} yield the desired p-NC property for adjacent  edges $e$ and $f$, provided $n$ is sufficiently large:
\begin{align*}
	&\bbP_{1/2}\big[ \omega(e)=\omega(f)=1,\omega\in\sC  \big] \cdot
	\bbP_{1/2}\big[ \omega\in\sC\big]\\
	&=\frac{1}{4}\Big(1- (n-3)x_n+o(1)x_n\Big)\Big(1- nx_n+o(1)x_n\Big)=\frac{1}{4}\Big(1- (2n-3)x_n+o(1)x_n \Big)\\
	&\leq \frac{1}{4}\Big(1- (2n-4)x_n+o(1)x_n \Big) \tag*{ for $n$ sufficiently large}\\
	&=\bigg( \frac{1}{2}\Big[ 1- (n-2)x_n+o(1)x_n  \Big] \bigg)^2 = \bbP_{1/2}\big[ \omega(e)=1, \omega\in\sC  \big]\cdot \bbP_{1/2}\big[ \omega(f)=1, \omega\in\sC  \big]\,.
\end{align*}

%However, the same estimates do not suffice for non-adjacent edges; a more refined analysis of the probabilities $\bbP_{1/2}\big[ \omega\notin\sC  \big]$, $\bbP_{1/2}\big[ \omega\notin\sC \mid \omega(e)=1  \big]$ and $\bbP_{1/2}\big[ \omega\notin\sC \mid \omega(e)=\omega(f)=1  \big]$ is required. These estimates will be our focus of the next subsection. 
\noindent
\textbf{The non-adjacent case:}\\
For non-adjacent $e$ and $f$, 
by Lemma~\ref{lem: disconnected one} and \ref{lem: disconnected three}  we have
\begin{align*}
&\bbP_{1/2}\big[ \omega(e)=\omega(f)=1,\omega \in \sC \big] \cdot
\bbP_{1/2}\big[ \omega \in \sC \big]\\
&=\frac{1}{4}\left( 1-  \bbP_{1/2}\big[ \omega \notin \sC  \mid \omega(e)=\omega(f)=1\big] \right)\cdot \left(1-\bbP_{1/2}\big[ \omega \notin \sC \big] \right)\\
&=\frac{1}{4}\big( 1- (n-4)x_n-8x_n^2+o(1)x_n^2 \big)\cdot \big(1-nx_n+o(1)x_n^2  \big)\\
&=\frac{1}{4}\Big( 1-(2n-4)x_n+(n^2-4n-8)x_n^2+o(1)x_n^2 \Big)\,.	
\end{align*}
On the other hand,  Lemma~\ref{lem: disconnected two} yields
\begin{align*}
&\bbP_{1/2}\big[ \omega(e)=1,\omega \in \sC  \big] \cdot
\bbP_{1/2}\big[\omega(f)=1,\omega \in \sC \big]\\
&=\frac{1}{4}\bigg(1 -(n-2)x_n-4x_n^2+o(1)x_n^2 \bigg)^2\\
&=\frac{1}{4}\bigg(1 -(2n-4)x_n+(n^2-4n-4)x_n^2+o(1)x_n^2 \bigg)\,.
\end{align*}
Comparing the two expansions, we see that for all sufficiently large $n$,
\begin{align*}
&\bbP_{1/2}\big[ \omega(e)=\omega(f)=1,\omega \in \sC   \big] \cdot
\bbP_{1/2}\big[ \omega \in \sC \big]\\
&\leq \bbP_{1/2}\big[ \omega(e)=1,\omega \in \sC   \big] \cdot
\bbP_{1/2}\big[\omega(f)=1, \omega \in \sC \big]\,.	
\end{align*}
By Lemma~\ref{lem: reformulation of p-NC for UC}, this inequality is exactly the required p-NC property for non-adjacent edges $e,f$.
\end{proof}

\begin{remark*}
	From the  proof of Theorem~\ref{thm: p-NC for UC}, we  see that the adjacent-edges case is simpler: it requires estimates only up to the order of $o(1)x_n$, whereas  the non-adjacent case demands estimates up to $o(1)x_n^2$.
\end{remark*} 

\begin{remark}
	Fix $\beta\in(0,\infty)$ and let $\bbP_{\beta \mathrm{UC}}$  be the probability measure on $\sC(G)$ such that 
	\[
	\bbP_{\beta \mathrm{UC}}(\omega)\propto \beta^{\#\eta(\omega)}\,,\quad \forall\, \omega\in \sC(G)\,.
	\]
	This measure is in fact the limit of the random cluster measures $\varphi_{p,q}$ as $ q\to0$ with $p=\frac{\beta}{1+\beta}$ held fixed. Hence, we also expect it to exhibit the p-NC property. The method developed in  the present work applies to such measures on $K_n$ as well. It suffices to replace the measure $\bbP_{1/2}$ in Lemma~\ref{lem: reformulation of p-NC for UC} with $\bbP_{p}$ where $p=\frac{\beta}{1+\beta}$, yielding  an analogue of Lemma~\ref{lem: reformulation of p-NC for UC} for the p-NC property of $\bbP_{\beta \mathrm{UC}}$. Then, by establishing similar estimates as in Section~\ref{sec:estimates} under $\bbP_{p}$, one can derive the corresponding p-NC property for $\bbP_{\beta \mathrm{UC}}$ on $K_n$ when $n$ is sufficiently large. We choose to present the $\beta=1$ case for simplicity. 
\end{remark}

\begin{remark*}
	We expect that the approach used for Theorem~\ref{thm: p-NC for UC} can be adapted to other families of graphs with sufficient symmetry and known connectivity thresholds. In particular, the hypercube $\{0,1\}^n$ is a promising candidate, as its critical probability for connectedness is exactly $1/2$ \cite{ES1979}. We plan to investigate this direction in future work.
\end{remark*}

%%%%%%%%%%%%%%%%%%%%%    Uniform connected subgraphs       %%%%%%%%%%%%%%%%%%%%%
%%%%%%%%%%%%%%%%%%%%%%%%%%%%%%%%%%%%%%%%%%%%%%%%%%%%%%%%%%%%%%%%%%%%%%%%%%%%%%%%% 

%% If UC.tex does NOT already contain its bibliography, use:
%%
%\bibliographystyle{plain}
%\bibliography{UC_ref}

\putbib[NCUS_ref]

\end{bibunit}

\clearpage

%%%%%%%%%%%%%%%%%%%%%%%%%%%%%%%%%%%%%%%%%%%%%%%%%%%%%%%%%%%%
%%%%%%%%%%%%%%%%%%%%%%%%%%%%%%%%%%%%%%%%%%%%%%%%%%%%%%%%%%%%
%% PART II: kUF
%%%%%%%%%%%%%%%%%%%%%%%%%%%%%%%%%%%%%%%%%%%%%%%%%%%%%%%%%%%%
%%%%%%%%%%%%%%%%%%%%%%%%%%%%%%%%%%%%%%%%%%%%%%%%%%%%%%%%%%%%

\part{Uniform $k$-component spanning forests}

\resetpartcounters

\papertitle{
Negative Correlation for Uniform $k$-Forests
of the Complete Graph
}

\begin{bibunit}[plain]

\begin{paperabstract}
We investigate the pairwise negative correlation property for the uniform probability measure on spanning forests of the complete graph $K_n$ with exactly $k$ connected components. We prove that, for each fixed $k$, this measure has the pairwise negative correlation property for all sufficiently large $n$. Our result gives a fixed-component counterpart to the negative-correlation problem for the unconditioned uniform forest measure and identifies moment relations that also arise in the recent resolution of the latter problem on complete graphs by Gupta~\cite{Gupta2026UF}.
\end{paperabstract}

\section{Introduction}

\subsection{Motivation}

The enumeration of spanning trees and forests is a classical topic in combinatorics, beginning with Kirchhoff's celebrated matrix-tree theorem \cite{Kirchhoff1847}; see also  Moon \cite{Moon1970}. Liu and Chow \cite{Liu_Chow1981} extended this line of work by giving an explicit formula for the number of spanning forests with a prescribed number of connected components. Beyond such enumeration problems,  one may ask about the probabilistic structure of a uniformly chosen forest. In particular, given two distinct edges, does the occurrence of one make the occurrence of the other less likely?

This question is captured by the pairwise negative correlation (p-NC) property. Let $G=(V,E)$ be a finite graph and let $\mu$ be a probability measure on $\{0,1\}^E$. We say that $\mu$ satisfies the p-NC property, if for every pair of distinct edges $e,f\in E$,
\be\label{eq: def p-NC}
\mu\big[ \omega(e)=\omega(f)=1 \big]\leq \mu\big[\omega(e)=1\big]\cdot\mu\big[\omega(f)=1\big] \,.
\ee
For a uniform probability measure on a family of subgraphs, the inequality \eqref{eq: def p-NC} is equivalently a  counting inequality comparing the number of subgraphs containing both $e$ and $f$ with the number containing $e$ (or $f$) alone, normalized by the total number of subgraphs in the family.

Negative dependence of this type arises naturally in the study of random-cluster models. The p-NC property is conjectured for random-cluster measures $\phi_{p,q}$ with $q<1$ \cite{Grimmett_Winkler2004}. In the appropriate $q\to0$ limit, the random-cluster model gives rise to the uniform forest measure \cite[Section 1.5, Theorem 1.23]{Grimmett2006}. This motivates the study of negative correlation for uniform forests and, more specifically, for their conditional distributions obtained by prescribing the number of connected components.

Let us therefore fix the number of components. A spanning forest on $n$ vertices with $k$ connected components has exactly $n-k$ edges. Thus the uniform measure on $k$-component spanning forests is the $(n-k)$-truncation of the uniform forest measure, where truncation means conditioning on the number of open edges. Such conditioning is nontrivial from the viewpoint of negative dependence: p-NC is not, in general, preserved under truncation. In fact, there are finite connected graphs for which the uniform two-component forest measure fails p-NC; see Example~\ref{example:HSW2022} below.

The complete graph $K_n$ is therefore a natural setting in which to investigate the fixed-component problem. Its symmetry reduces the joint law of two distinct edges to two cases, according to whether the edges are adjacent or disjoint, and also makes exact enumeration available. We ask whether, for a prescribed number $k$ of components, the uniform probability measure on $k$-component spanning forests of $K_n$ satisfies pairwise negative correlation.

%The passage from the unconditioned forest measure to a fixed-component stratum is not innocuous from the viewpoint of negative dependence. In general, truncation does not preserve p-NC: there are finite connected graphs for which the uniform measure on two-component spanning forests fails this property, as illustrated later in Example~\ref{example:HSW2022}. Thus, even if negative correlation is known or expected for an unconditioned model, it does not automatically follow for its fixed-component counterparts.
%
%This makes the complete graph \(K_n\) a particularly natural setting in which to investigate the problem. Its high degree of symmetry substantially restricts the possible edge-pair configurations and makes exact enumeration accessible, while the failure of p-NC under truncation on general graphs shows that a positive result cannot be taken for granted. We are therefore led to the following question: for a prescribed number \(k\) of connected components, does the uniform measure on \(k\)-component spanning forests of \(K_n\) satisfy pairwise negative correlation?

\subsection{Main results}

Our main result gives a partial affirmative answer to the question posed above: for every fixed number of components, pairwise negative correlation holds on $K_n$ once $n$ is sufficiently large. Before stating the result precisely, we introduce the necessary notation.

Let $G=(V,E)$ be a finite connected graph. For $\omega\in\{0,1\}^{E}$, write
\[
\eta(\omega):=\{e\in E:\omega(e)=1\},
\]
and let $\ncc{\omega}$ denote the number of connected components of the spanning subgraph $\big(V,\eta(\omega)\big)$. Write
\[
\sF=\sF(G)
:=\left\{\omega\in\{0,1\}^{E} \colon \big(V,\eta(\omega)\big)\text{ is a forest} \right\}
\]
for the set of spanning forests of $G$, and denote by $\UF$ the uniform probability measure on $\mathcal F$. For $1\le k\le |V|$, let
\[
\kF=\kF(G)
:=\left\{\omega\in\mathcal F\colon \ncc{\omega}=k\right\},
\]
and let $\kUF$ be the uniform probability measure on $\mathcal F^{(k)}$, whenever this set is nonempty. Thus $\kF$ is the set of spanning forests with exactly $k$ connected components. In particular, $\oneF=\sT$, where $\sT=\sT(G)$ denotes the set of spanning trees, and we write $\mathbb P_{\mathrm{UST}}$ for the uniform spanning tree measure. It is well known that $\mathbb P_{\mathrm{UST}}$ is negatively associated, and hence in particular satisfies p-NC; see \cite{FM1992} and also \cite[Chapter 4]{LP2016}. Thus the case $k=1$ is classical, and we henceforth focus on $k\ge2$.

We can now state our main theorem.

\begin{theorem}\label{thm: p-NC for kUF}
	For every fixed integer $k\ge 2$, there exists $N(k)<\infty$ such that, for all $n\ge N(k)$, the uniform measure $\kUF$ on $\kF(K_n)$ satisfies pairwise negative correlation.
\end{theorem}
The restriction to sufficiently large \(n\) arises from the asymptotic nature of the proof rather than from a known obstruction. For several small values of \(k\), including \(k=2,3,4\), direct calculations show that \(\mathbb P_{k\mathrm{UF}}\) satisfies p-NC throughout the admissible range; see Remark~\ref{rem: 3.10}. It remains open, however, what happens for general pairs \((n,k)\) outside the asymptotic regime. We return to this question in Section~\ref{sec: remarks and ques}.

\subsection{Discussion and related work}
Negative dependence under conditioning has been studied in several forms;  see, for example \cite{Dev_Proschan1983, Pemantle2000, Borcea_Branden_Liggett2009}. Stronger notions such as the strong Rayleigh property are preserved under truncation  \cite[Corollary 4.18]{Borcea_Branden_Liggett2009}, whereas p-NC alone need not be. In particular, even within the class of spanning forests, conditioning on the number of components may destroy pairwise negative correlation; see Example~\ref{example:HSW2022}. Thus the fixed-component problem considered here is not an immediate consequence of negative-correlation results for the unconditioned forest measure.

The proof of Theorem~\ref{thm: p-NC for kUF} combines the forest-enumeration formula of Liu and Chow \cite{Liu_Chow1981} with first- and second-moment identities arising from the symmetry of $K_n$. Since a $k$-component forest has deterministically $n-k$ edges, the two orbits of pairs of distinct edges---adjacent and nonadjacent pairs---can be related through the degree-square statistic
\[
\sum_{x\in V(K_n)}\deg(x;\omega)^2.
\]
It is then enough to estimate the probability that a fixed adjacent pair is present. This probability is obtained by contracting the two edges and applying the Liu--Chow formula.

After the arXiv version of this work appeared \cite{TZ2026}, Gupta \cite{Gupta2026UF} proved that the unconditioned uniform forest measure on $K_n$
satisfies p-NC for every $n\ge2$. His result and Theorem~\ref{thm: p-NC for kUF} concern different measures: here the number of components is fixed, whereas in the unconditioned model it is random. The corresponding moment calculations are closely related. For each fixed number of components, Gupta's identities reduce to those used here; averaging over different component counts introduces an additional variance term associated with the random number of edges. Thus the fixed-component calculation isolates the part of the moment mechanism in which the edge count is deterministic.

The earlier manuscript \cite{TZ2026} also considers connected spanning subgraphs of $K_n$ with fixed excess $k$, and proves that, for every fixed $k$, the corresponding uniform measure satisfies p-NC for all sufficiently large $n$. Thus the fixed-component forest and fixed-excess connected models provide two complementary truncations in which the number of edges is deterministic. The present paper contains the fixed-component forest part of the earlier manuscript \cite{TZ2026}.

%The present paper is a shortened version of part of the work in the arXiv manuscript \cite{TZ2026}; here we focus exclusively on the fixed-component forest theorem. 

%The question of whether truncations preserve forms of negative dependence has been discussed extensively in the literature \cite{Dev_Proschan1983, Pemantle2000, Borcea_Branden_Liggett2009}. While strong forms such as the strong Rayleigh property are preserved under truncations \cite[Corollary 4.18]{Borcea_Branden_Liggett2009}, there is no general implication between the p-NC property of a measure and that of its truncations. As noted above, our positive results for complete graphs do not extend to arbitrary finite connected graphs.
%
%Our proofs rely on two distinct classical techniques. For $k$-forests, we use the explicit counting formula of Liu and Chow \cite{Liu_Chow1981}. For connected subgraphs with excess $k$, we employ the singular analysis of exponential generating functions developed in \cite{FO1990}, following the approach used by Stark \cite{Stark2011} for the uniform forest measure.
%
%The present work represents  part of the work in the arXiv version \cite{TZ2026}. 

\subsection{Outline}
%The rest of the paper is organized as follows.

Section~\ref{sec: p-NC for k-forests} is devoted to the proof of Theorem~\ref{thm: p-NC for kUF}. We first derive moment identities that reduce p-NC to bounds on the probability that a fixed pair of adjacent edges is present. We then use the Liu--Chow enumeration formula to compute this probability and obtain the required asymptotic bounds. Section~\ref{sec: remarks and ques} discusses the role of truncation and records two open problems suggested by the theorem.

%
%The key ingredient is the counting formula for the number of $k$-component forests due to Liu and Chow \cite{Liu_Chow1981}. This formula, together with the symmetry of the complete graph, allows us to compute the joint probabilities $\kUF[\omega(e) = \omega(f) = 1]$. By comparing these with the square of the marginal probability $\kUF[\omega(e) = 1]^2$---which is easily obtained via the first moment method---we establish the desired negative correlation.

%In Section~\ref{sec: p-NC for kUC} we prove Theorem~\ref{thm: p-NC for kUC}. The proof relies on the singular analysis technique applied to the associated exponential generating functions. Our argument requires a delicate analysis of these functions, which involves rather lengthy computations; some of these calculations are therefore deferred to the supplementary material \footnote{the supplementary material is the appendix in the arxiv version \cite{TZ2026}.}.

%Finally, in Section~\ref{sec: remarks and ques} we conclude with some examples, remarks, and open questions.

	\paragraph{General asymptotic notation.}
We will use the standard asymptotic notation: 
$f(n)\asymp g(n)$ means there exist constants $c_1,c_2>0$ such that $c_1g(n)\leq f(n)\leq c_2g(n)$ for all sufficiently large $n$; 
$f(n)=O(g(n))$ means there exists $c>0$ with $f(n)\leq cg(n)$; and 
$f(n)=o(g(n))$ means $\lim_{n\to\infty} f(n)/g(n)=0$.

\section{Uniform $k$-component forests of complete graphs}\label{sec: p-NC for k-forests}
%Because $K_n$ has exactly two orbits on unordered pairs of distinct edges, p-NC amounts to two inequalities. The deterministic edge count gives a linear relation between the two joint probabilities. Consequently, it suffices to estimate only the adjacent-pair probability.
%
%To prove \cref{thm: p-NC for kUF}, we compare $\kUF\big[\omega(e)=\omega(f)=1\big]$ with $\kUF\big[\omega(e)=1\big]^2$. A straightforward first-moment argument yields 
%\[
%\kUF\big[\omega(e)=1\big]=\frac{n-k}{\binom{n}{2}}\,,
%\]
%while a second-moment analysis relates the quantity $\kUF\big[\omega(e)=\omega(f)=1\big]$ for adjacent edges $e,f$ to that for a non-adjacent pair. We first present these two observations about the first- and second-moment methods. 

Because $K_n$ has exactly two orbits on unordered pairs of distinct edges, p-NC reduces to two inequalities, corresponding to adjacent and non-adjacent pairs. The deterministic number of edges in a spanning $k$-forest yields a linear relation between the corresponding joint probabilities, so it suffices to estimate the adjacent-pair probability.
A straightforward first-moment argument yields 
\[
\kUF\big[\omega(e)=1\big]=\frac{n-k}{\binom{n}{2}}\,,
\]
while the second-moment identity relates the joint probabilities for adjacent and non-adjacent pairs. We begin by recording these identities.

We then recall the enumeration formula of  Liu and Chow \cite{Liu_Chow1981} and apply it to count the number of $k$-forests of $K_n$, as well as those containing  a fixed pair of adjacent edges $e,f$. This yields an explicit expression for $\kUF\big[\omega(e)=\omega(f)=1\big]$, when $e$ and $f$ are adjacent.

Finally, combining this expression with the first- and second-moment identities, we compare the joint probabilities for adjacent and non-adjacent pairs with
$ \displaystyle \kUF\big[\omega(e)=1\big]^2. $
For all sufficiently large $n$, these comparisons yield the two inequalities required for p-NC and hence complete the proof of \cref{thm: p-NC for kUF}.

\subsection{A moment reduction for pairwise negative correlation}

\begin{lemma}\label{lem: first moment}
	For the $\kUF$ measure on the complete graph $K_n$ and an edge $e\in E(K_n)$, we have 
	\[
	\kUF\big[\omega(e)=1\big]=\frac{n-k}{\binom{n}{2}}=\frac{2(n-k)}{n(n-1)}\,.
	\]
\end{lemma}
\begin{proof}
	Notice that for each $\omega\in \kF(K_n)$, there are exactly $n-k$ edges in $\eta(\omega)$, i.e., 
	\[
	\sum_{f \in E(K_n)} \mathbf{1}_{\{\omega(f)=1\}}=n-k\,.
	\]
	Taking the expectation and using the edge-transitivity property of $K_n$, we obtain 
	\[
	\binom{n}{2}\kUF\big[\omega(e)=1\big]=n-k\,,
	\]
	whence
	\[
	\kUF\big[\omega(e)=1\big]=\frac{2(n-k)}{n(n-1)}\,.  \qedhere
	\]
\end{proof}

\begin{lemma}\label{lem: second moment}
	Let $\mathrm{Aut}(K_n)$ denote the automorphism group of the complete graph $K_n$, and recall that $\kF(K_n)$ denotes the set of $k$-forests of $K_n$. Suppose $\bP$ is an $\mathrm{Aut}(K_n)$-invariant probability measure such  that $\bP\big[(V,\eta(\omega))\in \kF(K_n)\big]=1$, where $\omega$ is a random element of $\{0,1\}^{E(K_n)}$ with law $\bP$. For $n\ge3$ define $p_1\coloneq \bP\big[\omega(e)=\omega(f)=1\big]$, where $e$ and $f$ are a pair of adjacent edges in $K_n$. For $n\ge 4$, define $p_2\coloneq \bP\big[ \omega(e)=\omega(e')=1\big]$, where $e$ and $e'$ are a pair of non-adjacent edges in $K_n$. Let $\deg(x;\omega)$ denote the degree of the vertex $x$ in the random subgraph $(V,\eta(\omega))$. Then we have
	\be\label{eq: p_1}
	p_1=\frac{\bE\Big[\sum_{x\in V(K_n)}{\deg(x;\omega)^2}\Big]-2(n-k)}{n(n-1)(n-2)}\quad \mbox{ for }n\ge3\,,
	\ee
	and
	\be\label{eq: p_2}
	p_2=\frac{4\bigg((n-k)(n-k+1)-\bE\Big[\sum_{x\in V(K_n)}{\deg(x;\omega)^2}\Big]\bigg)}{n(n-1)(n-2)(n-3)}\quad \mbox{ for }n\ge4\,.
	\ee
	In particular, for $n\ge4$,
	\be\label{eq: p_1 and p_2}
	p_2=\frac{4\big[(n-k)(n-k-1)-n(n-1)(n-2)p_1\big]}{n(n-1)(n-2)(n-3)}\,.
	\ee
\end{lemma}
\begin{proof}
	First, consider ordered pairs $\langle e, f\rangle$ such that $e,f$ have a common vertex (denoted $e\sim f$). Observe that the number of such ordered pairs is 
	\[
	\bigg|\Big\{ \langle e, f\rangle\in E(K_n)\times E(K_n)\colon e\sim f \Big\}\bigg|=n(n-1)(n-2)\,.
	\]
	Next, for each $\omega\in \kF(K_n)$, we have 
	\[
	\sum_{x\in V(K_n)}\deg(x;\omega)=2\big|\{e\in E(K_n)\colon \omega(e)=1 \}\big|=2(n-k)\,,
	\]
	since the sum of degrees in a graph equals twice the number of edges, and $\omega$ corresponds to a $k$-forest with $n-k$ edges.
	
	Now, note that
	\begin{align*}
	\sum_{\substack{\langle e, f\rangle\in E(K_n)\times E(K_n) \\ e\sim f}}\mathbf{1}_{\{ \omega(e)=\omega(f)=1 \}}&=\sum_{x\in V(K_n)}\deg(x;\omega)\cdot \big[\deg(x;\omega)-1\big]\\
	&=\sum_{x\in V(K_n)}\deg(x;\omega)^2-2(n-k)\,.
	\end{align*}
	Taking the expectation on both sides and using the $\mathrm{Aut}(K_n)$-invariance of $\bP$ (which implies all adjacent edge pairs have the same probability $p_1$), we obtain \eqref{eq: p_1}:
	\[
	n(n-1)(n-2)p_1=\bE\Big[\sum_{x\in V(K_n)}{\deg(x;\omega)^2}\Big]-2(n-k)\,.
	\]

	We now turn to non-adjacent edges. 
	Observe that the number of ordered pairs  $\langle e,e'\rangle$ in $K_n$  such that $e\not\sim e'$ (i.e., $e$ and $e'$ are disjoint) is $6\binom{n}{4}$. Moreover for an edge $e$ with $\omega(e)=1$, let $e^-,e^+$ denote its two endpoints of $e$. The number of edges $e'\in\eta(\omega)$ (i.e., $\omega(e')=1$) that are disjoint from $e$ is
	\[
	(n-k)-\big[\deg(e^-;\omega)+\deg(e^+;\omega)-1\big]\,.
	\] 
	This follows because there are $n-k$ edges total in $\eta(\omega)$, and we subtract the edges incident to $e^-$ or $e^+$ (counting $e$ itself twice, so we add $1$ to correct for over-counting.)
	
	Summing over edges $e\in\eta(\omega)$, we have
	\begin{align*}
	\sum_{\substack{\langle e,e'\rangle\in E(K_n)\times E(K_n)\\ e\not\sim e'}}\mathbf{1}_{\{\omega(e)=\omega(e')=1\}}&=\sum_{e\in\eta(\omega)}\Big( (n-k)-\big[\deg(e^-;\omega)+\deg(e^+;\omega)-1\big] \Big)\\
	&=(n-k+1)\card{\eta(\omega)}-\sum_{e\in\eta(\omega)}\big[\deg(e^-;\omega)+\deg(e^+;\omega)\big]\\
	&=(n-k)(n-k+1)-\sum_{x\in V(K_n)}\deg(x;\omega)^2\,,
	\end{align*}
	where the last equality uses $\card{\eta(\omega)}=n-k$ and the fact that $\sum_{e\in\eta(\omega)}\big[\deg(e^-;\omega)+\deg(e^+;\omega)\big]=\sum_{x\in V(K_n)}\deg(x;\omega)^2$ (each vertex's degree is counted once for each edge incident to it).
	Taking the expectation of both sides and using the $\mathrm{Aut}(K_n)$-invariance of $\bP$ (which implies all non-adjacent edge pairs have the same probability  $p_2$), we obtain
	\[
	6\binom{n}{4}p_2=(n-k)(n-k+1)-\bE\Big[\sum_{x\in V(K_n)}\deg(x;\omega)^2\Big]\,.
	\]
	Since $6\binom{n}{4}=\frac{n(n-1)(n-2)(n-3)}{4}$, substituting this into the equation and rearranging gives \eqref{eq: p_2}.
	
	Finally, substituting the expression for $\bE\Big[\sum_{x\in V(K_n)}\deg(x;\omega)^2\Big]$  from \eqref{eq: p_1} into \eqref{eq: p_2} and simplifying yields \eqref{eq: p_1 and p_2}. 
\end{proof}

\begin{corollary}\label{cor: 3.3}
	With notation as in \cref{lem: second moment} and $p_0\coloneq \bP\big[\omega(e)=1\big]=\frac{2(n-k)}{n(n-1)}$ (see \cref{lem: first moment}), we have
	\[
	p_1\leq p_0^2\quad \mbox{ if and only if }\quad p_1\leq \frac{4(n-k)^2}{n^2(n-1)^2}\,,
	\]
	and
	\[
	p_2\leq p_0^2\quad \mbox{ if and only if }\quad p_1\geq \frac{(n-k)\big[3n^2-(5+4k)n+6k\big]}{n^2(n-1)^2(n-2)}\,.
	\]
\end{corollary}

\subsection{The number of $k$-forests of complete graphs}

\subsubsection{Review of the $k$-forest formula due to Liu and Chow   }

Our main tool for proving \cref{thm: p-NC for kUF} is a counting scheme for the number of $k$-component forests due to Liu and Chow \cite{Liu_Chow1981}; see also Myrvold \cite{Myrvold1992} for a simplified proof of their result. We now briefly recall the $k$-forest formula from \cite{Liu_Chow1981}.

Suppose $G=(V,E)$ is a finite, connected graph, and we allow multiple parallel edges between pairs of vertices. The graphs are viewed as labelled graphs, so that  parallel edges can be distinguished if they exist. Let  $n=|V|$ and $V=\{v_1,\ldots,v_n\}$. The \textbf{adjacency matrix} $A=A(G)$ of $G$ is an $n\times n$ matrix whose entry $a_{ij}$ is the number of edges between vertex $v_i$ and $v_j$.  The \textbf{Kirchhoff matrix} is defined as $M=M(G)\coloneq D(G)-A(G)$, where $D(G)$ is the $n\times n$ diagonal matrix such that  $D_{ii}$ equals the degree of vertices $v_i$ for each $i\in\{1,\ldots,n\}$. Thus $M_{ii}=\sum_{j\neq i}a_{ij}$. In particular, loops in $G$ may contribute to $A(G)$ but not to $M(G)$. Following \cite{Liu_Chow1981}, we let $M(i)$ be the $i$-th principal submatrix of $M$, that is, the submatrix obtained from $M$ by deleting both the $i$-th row and $i$-th column. Similarly, we let $M(i_1,\ldots,i_m)$ denote the principal submatrix obtained from $M$ by deleting both the $i_j$-th row and the $i_j$-th column for each $j\in\{1,\ldots,m\}$. We adopt the same convention as in \cite{Liu_Chow1981}:
\be\label{eq: convention 1 for Kirchhoff matrix}
M(i_1,\ldots,i_n)=1\,,\qquad \mbox{ if all }i_j \mbox{ are distinct}
\ee
and 
\be\label{eq: convention 2 for Kirchhoff matrix}
M(i_1,\ldots,i_m)=0\,,\qquad \mbox{ if some }i_j=i_k \mbox{ for some } j\neq k,
\ee
where $0$ and $1$ are simply scalars (real numbers).

Fix an arbitrary vertex, say $v_*\in V(G)$, and let $n^*=V(G)\setminus \{v_*\}$ denote the set of  the remaining vertices. For a subset $S=\{v_{i_1},\ldots,v_{i_m}\}\subset n^*$, write $M(S)=M(i_1,\ldots,i_m)$ for the principal submatrix of $M$ obtained by deleting both the rows and columns corresponding to the vertices in $S$. Denote by $\nu_r(S)$ the total number of matchings (in $G$) each consisting of $r$ edges whose endpoints all lie in $S$.  Moreover we adopt the convention that $\nu_0(S)=1$. 

Recall that $\kF=\kF(G)$ denotes the set of $k$-forests of $G$, and we write $\card{\kF}$ for its cardinality. In particular, $\card{\oneF(G)}$ is just the number of spanning trees of $G$, and Kirchhoff's  famous matrix-tree theorem  \cite{Kirchhoff1847} states that
\be\label{eq: Kirchhoff matrix-tree thm}
\card{\oneF(G)}=\det M(1)=\cdots=\det M(n)\,.
\ee
Liu and Chow \cite[Equation (20) on p.~662]{Liu_Chow1981} provided the following formula for counting the number of $k$-forests of $G$: for $k\leq n$,
\be\label{eq: Liu and Chow's formula for k-forests}
\card{\kF(G)}=\sum_{r=0}^{k-1}(-1)^r\sum_{\substack{ S\subset n^*\\
		|S|=k+r-1}}\nu_r(S)\card{\oneF(G_S)}\,,
\ee
where $G_S$ is the graph obtained from $G$ by identifying all the vertices in $S\cup \{v_*\}$. 
In particular, applying \eqref{eq: Liu and Chow's formula for k-forests} to complete graphs gives \footnote{This is Equation (22) in \cite{Liu_Chow1981}, but with a minor correction: the upper limit in the sum should reflect that if $r+k > n$, then there is no subset $S \subset n^*$ with $|S| = k+r-1$, and thus we take $\nu_r(S)=0$ for such $r$.} 
\be\label{eq: number of k-forest of K_n}
\card{\kF(K_n)}=n^{n-k-1}\cdot (n-1)!\sum_{r=0}^{(k-1)\wedge(n-k)}\Big(-\frac{1}{2n}\Big)^r\frac{k+r}{r!(k-r-1)!(n-k-r)!}\,.
\ee
For example, 
\be\label{eq: number of 2-forest of K_n}
\card{\twoF(K_n)}=\frac{1}{2}(n-1)(n+6)n^{n-4}  \,,\quad \card{\threeF(K_n)}=\frac{n^{n-6}(n-1)(n-2)(n^2+13n+60)}{8}\,.
\ee

\subsubsection{Some preliminary lemmas}

To count the number of $k$-forests of $G$ that contain some fixed edges $e,f$, we instead count the number of $k$-forests in a new graph obtained from $G$ by contracting the edges $e$ and $f$. We can then apply Liu and Chow's formula in this new graph to obtain the desired result. First, we define what we mean by ``contraction".

\begin{definition}\label{def: contraction along an edge}
	Suppose $G=(V,E)$ is  a finite graph and $e\in E$ is an edge. The \textbf{contraction} $G/e$ is the graph obtained from $G$ by identifying  the endpoints of $e$ (and removing the resulting self-loops). Similarly, for two edges $e,f\in E(G)$, we write $G/\{e,f\}$ for the graph obtained by contracting both $e$ and $f$. See Fig.~\ref{fig: K_n contraction along an edge} for examples of $K_n/e$ and $K_n/\{e,f\}$ for a pair of adjacent edges. 
\end{definition}

%%%%%%%%%%%%%%%%%%%%%%%%%%%%%%%%%%%%%%%%%%%%%%%%%%%%%%%%%%%%
%%%%%%%%%%%%%%%%%%%%%  Figure 1 %%%%%%%%%%%%%%%%%%%%%%%%%%%
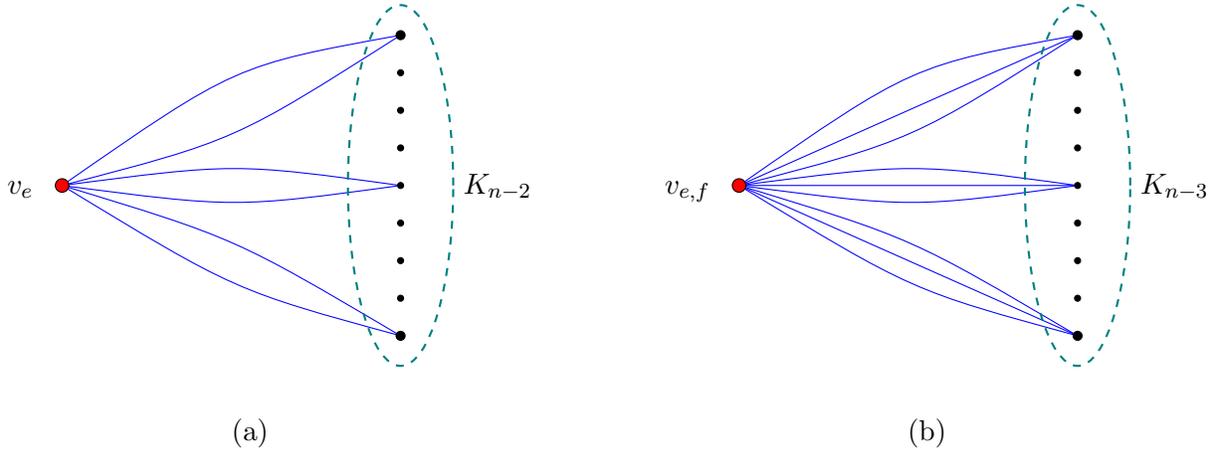
\begin{figure}[H]
	\centering
	\begin{tikzpicture}[scale=1, text height=1.5ex,text depth=.25ex] 
	%\draw [help lines] (0,0) grid (8,6);

	%%%% connect $v_e$ and the rest %%%%
	\draw[color=blue,thin] (0.5,3)..controls (2.8,3.6)..(5,5);
	\draw[color=blue,thin] (0.5,3)..controls (2.8,4.6)..(5,5);
	\draw[color=blue,thin] (0.5,3)..controls (2.8,3.3)..(5,3);
	\draw[color=blue,thin] (0.5,3)..controls (2.8,2.7)..(5,3);
	\draw[color=blue,thin] (0.5,3)..controls (2.8,2.3)..(5,1);
	\draw[color=blue,thin] (0.5,3)..controls (2.8,1.6)..(5,1);
	
	%%%%  draw $v_e$ %%%%%
	\draw[fill=red] (0.5,3) circle [radius=0.09];
	\node[left] at (0.25,3) {$v_e$};

	%%%%% draw $K_{n-2}$ %%%%%%%
	\draw[fill=black] (5,1) circle [radius=0.06];
	
	\draw[fill=black] (5,3.5) circle [radius=0.04];
	\draw[fill=black] (5,4) circle [radius=0.04];
	\draw[fill=black] (5,4.5) circle [radius=0.04];
	
	\draw[fill=black] (5,3) circle [radius=0.04];
	
	\draw[fill=black] (5,1.5) circle [radius=0.04];
	\draw[fill=black] (5,2) circle [radius=0.04];
	\draw[fill=black] (5,2.5) circle [radius=0.04];
	
	\draw[fill=black] (5,5) circle [radius=0.06];

	\draw [color=teal,thick,dashed] (5,3) ellipse (0.7 and 2.4);
	\node[right]at (5.7,3) {$K_{n-2}$};

	\node[below] at (3,0) {(a)};

	\begin{scope}[shift={(9,0)}]
	%%%% connect $v_e$ and the rest %%%%
	\draw[color=blue,thin] (0.5,3)..controls (2.8,3.6)..(5,5);
	\draw[color=blue,thin] (0.5,3)..controls (2.8,4.6)..(5,5);
	\draw[color=blue,thin] (0.5,3)..controls (2.8,3.3)..(5,3);
	\draw[color=blue,thin] (0.5,3)..controls (2.8,2.7)..(5,3);
	\draw[color=blue,thin] (0.5,3)..controls (2.8,2.3)..(5,1);
	\draw[color=blue,thin] (0.5,3)..controls (2.8,1.6)..(5,1);

	\draw[color=blue,thin](0.5,3)--(5,5);

	\draw[color=blue,thin](0.5,3)--(5,3);

	\draw[color=blue,thin](0.5,3)--(5,1);
	
	%%%%  draw $v_e$ %%%%%
	\draw[fill=red] (0.5,3) circle [radius=0.09];
	\node[left] at (0.25,3) {$v_{e,f}$};

	%%%%% draw $K_{n-2}$ %%%%%%%
	\draw[fill=black] (5,1) circle [radius=0.06];
	
	\draw[fill=black] (5,3.5) circle [radius=0.04];
	\draw[fill=black] (5,4) circle [radius=0.04];
	\draw[fill=black] (5,4.5) circle [radius=0.04];
	
	\draw[fill=black] (5,3) circle [radius=0.04];
	
	\draw[fill=black] (5,1.5) circle [radius=0.04];
	\draw[fill=black] (5,2) circle [radius=0.04];
	\draw[fill=black] (5,2.5) circle [radius=0.04];
	
	\draw[fill=black] (5,5) circle [radius=0.06];

	\draw [color=teal,thick,dashed] (5,3) ellipse (0.7 and 2.4);
	\node[right]at (5.7,3) {$K_{n-3}$};

	\node[below] at (3,0) {(b)};
	
	\end{scope}
	\end{tikzpicture}
	\caption{(a) The graph $G=K_n/e$; (b) The graph $G=K_n/\{e,f\}$ for a pair of adjacent edges $e,f$}
	\label{fig: K_n contraction along an edge}
\end{figure}
%%%%%%%%%%%%%%%%%%%%%  Figure 1 %%%%%%%%%%%%%%%%%%%%%%%%%%%
%%%%%%%%%%%%%%%%%%%%%%%%%%%%%%%%%%%%%%%%%%%%%%%%%%%%%%%%%%%%

%\begin{lemma}\label{lem: bijection between forest on G containing e and G e contraction}
%	Suppose $G=(V,E)$ is a finite connected graph and $e,f\in E$ are not self-loops. For $k\leq |V|$, we have 
%	\be\label{eq: contract one edge}
%	\card{\big\{ F\in\kF(G)\colon e\in F  \big\}} =\card{\kF(G/e)}\,,
%	\ee 
%	and 
%	\be\label{eq: contract two edges}
%	\card{\big\{ F\in\kF(G)\colon e,f\in F  \big\}} =\card{\kF(G/\{e,f\})}\,.
%	\ee
%\end{lemma}
%\begin{proof}
%	Equality \eqref{eq: contract one edge} follows from the observation that  contraction along $e$ induces a bijection between the set $\big\{ F\in\kF(G)\colon e\in F  \big\}$ and the set $\kF(G/e)$. Indeed, in the case $k=1$, this corresponds to the spatial Markov property of the uniform spanning tree (UST)  conditioned on the event that $e$ is contained in the UST (see, for instance, \cite[Section 4.2, page~106]{LP2016}).
%	
%	Equality \eqref{eq: contract two edges} can be established either via a similar bijection or by applying \eqref{eq: contract one edge} twice. 
%\end{proof}

\begin{lemma}\label{lem: bijection between forest on G containing e and G e contraction}
	 Let $G=(V,E)$ be a finite connected graph, and let $A\subseteq E$ be a forest. Then, for every $k\leq |V|$, contraction of the edges in $A$ induces a bijection
	  \[ 
	  \{F\in\kF(G): A\subseteq F\big\} \longleftrightarrow \kF(G/A). 
	  \] 
	  In particular,
	   \begin{equation}\label{eq:contraction-bij}
	   \card{\big\{F\in\kF(G): A\subseteq F\big\}} = \card{\kF(G/A)}. 
	   \end{equation}
 \end{lemma}
 \begin{proof} For $F\in\kF(G)$ with $A\subseteq F$, let $F/A$ denote the image of $F\setminus A$ in $G/A$. Since $A\subseteq F$ and $F$ is a forest, contracting the edges of $A$ preserves the number of connected components and creates no cycle. Hence $F/A\in\kF(G/A)$. Conversely, if $\widetilde F\in\kF(G/A)$, then adjoining the edges of $A$ to the lift of $\widetilde F$ to $G$ yields a spanning $k$-forest of $G$ containing $A$. These two operations are inverse to each other, and therefore give the asserted bijection. 
 \end{proof}

The following matrices $A_m$ and $B_m$ arise frequently in applications of  Liu and Chow's formula; we record an observation about their determinants  for later use. 
\begin{lemma}\label{lem: two elementary determinants}
	Let $1\leq m\leq n-2$. Let $A_m$ denote a $m\times m$ matrix and $B_m$  a $(m+1)\times (m+1)$ matrix with the following  forms:
	\[
	A_{m}=
	\begin{pNiceMatrix}[nullify-dots,xdots/line-style=dotted]
	n-1 &  -1 & \Cdots &\Cdots & -1 \\
	-1 & \Ddots &  \Ddots & & \Vdots \\
	\Vdots &  \Ddots & & \\
	\Vdots &  & & & -1 \\
	-1 & \Cdots &  & -1 & n-1
	\end{pNiceMatrix}\,,\,
	B_m=
	\begin{pNiceMatrix}[margin,hvlines]
	\Block{3-3}<\LARGE>{A_{m}} & & & -(n-m) \\
	& \hspace*{1cm} & & \Vdots \\
	& & & -(n-m) \\
	-(n-m) & \Cdots& -(n-m) & m(n-m)
	\end{pNiceMatrix}\,.
	\]
	Then 
	\[
	\det A_m=n^{m-1}(n-m)\,,\quad \det B_m=0\,.
	\]
	
\end{lemma}

\subsubsection{The number of  $k$-forests containing a pair of adjacent edges}

First, we recall that Liu and Chow's formula  \eqref{eq: number of k-forest of K_n} already counts the number of  $k$-forests of $K_n$. To compute $\kUF\big[\omega(e)=\omega(f)=1\big]$, it suffices to count the number of $k$-forests that contain both $e$ and $f$. 

\begin{proposition}\label{prop: kFef}
	Suppose $k\geq 1$, $n\ge3$, and $(e,f)$ is  a pair of adjacent edges of $K_n$. 
	Then the number of  $k$-forests of $K_n$ that contain both $e$ and $f$ satisfies
	\be\label{eq: kFef}
	\card{\big\{ F\in\kF(K_n)\colon e,f\in F \big\}} =
	\sum_{r=0}^{\min\{k-1, n-k-2\}}\frac{(-1)^r(k+r+2) (n-3)!n^{n-k-r-3}}{2^rr!(k-r-1)! (n-k-r-2)!}  \,.
	\ee
\end{proposition}
\begin{proof}
	By \eqref{eq:contraction-bij}, it suffices to count the number of $k$-forests in $K_n/\{e,f\}$ (see part (b) in Fig.~\ref{fig: K_n contraction along an edge} for an illustration of $K_n/\{e,f\}$). Since $e,f$ are adjacent, the vertex $v_{e,f}$ (the vertex obtained by identifying the endpoints of $e$ and $f$) has three parallel edges to each of the remaining $n-3$ vertices, and these $n-3$ vertices themselves induce a complete graph $K_{n-3}$. 
	
	We  now apply Liu and Chow's formula \eqref{eq: Liu and Chow's formula for k-forests} to the graph $K_n/\{e,f\}$. First, we choose $v_*$ to be the special vertex $v_{e,f}$. Fix an index $r\in[0,\min\{k-1, n-k-2\}]$; there are $\binom{n-3}{k+r-1}$ choices for a subset  $S\subset V\big[K_n/\{e,f\}\big]\setminus \{v_*\}$ such that  $\card{S}=k+r-1$. For each such $S$,  the graph $G_S$ can be viewed as being obtained from $K_n$ by identifying all the vertices in $S$ and the endpoints of $e$ and $f$. Hence, it  has a Kirchhoff matrix of the form $B_m$ as in Lemma~\ref{lem: two elementary determinants}, where $m=n-(k+r+2)$. Deleting the row and column corresponding to the contracted vertex $v_*$ leaves $A_m$; hence the matrix-tree theorem gives $\card{\oneF(G_S)}=\det(A_m)=n^{m-1}(n-m)=n^{n-k-r-3}(k+r+2)$  using Lemma~\ref{lem: two elementary determinants}. 
	As for $\nu_r(S)$, we observe that the vertices in $S$ induce a complete graph $K_{k+r-1}$. Therefore, the number of $r$-matching $\nu_r(S)$ is given by
	\begin{align*}
	\nu_r(S)&=\frac{1}{r!}\binom{k+r-1}{2}\cdot \binom{k+r-1-2}{2}\cdots \binom{k+r-1-2(r-1)}{2}\\
	&=\frac{(k+r-1)!}{2^r \cdot  r! \cdot (k-r-1)!}\,.
	\end{align*}
	Substituting the above quantities into \eqref{eq: Liu and Chow's formula for k-forests}, we obtain the desired conclusion: 
	\begin{align*}
	& \quad \card{\big\{ F\in\kF(K_n)\colon e,f\in F \big\}} \\
	&=
	\sum_{r=0}^{\min\{k-1, n-k-2\}}(-1)^r \binom{n-3}{k+r-1}\cdot \frac{(k+r-1)!}{2^r   r!  (k-r-1)!} \cdot n^{n-k-r-3}(k+r+2)\\
	&=	\sum_{r=0}^{\min\{k-1, n-k-2\}}\frac{(-1)^r(k+r+2) (n-3)!n^{n-k-r-3}}{2^rr!(k-r-1)! (n-k-r-2)!}  \,. \qedhere
	\end{align*}
\end{proof}

\subsection{Proof of \cref{thm: p-NC for kUF}}

We are now ready to prove \cref{thm: p-NC for kUF}. We begin with some calculations involving the estimation of  $\card{\kF(K_n)}$ and $\card{\big\{ F\in\kF(K_n)\colon e,f\in F \big\}}$ for a pair of adjacent edges $(e,f)$.

\begin{lemma}\label{lem: leading coef}
	We have the following identities for $k\ge1$:
	\be\label{eq:lc1}
	\sum_{r=0}^{k-1}\frac{(-1)^r}{2^rr!(k-r-1)!}=\frac{1}{2^{k-1}(k-1)!}\,,
	\ee
	\be\label{eq:lc2}
	\sum_{r=0}^{k-1}\frac{(-1)^rr}{2^rr!(k-r-1)!}=-\frac{1}{2^{k-1}(k-2)!}\mathbf{1}_{\{k\geq 2\}}\,,
	\ee
	\be\label{eq:lc3}
	\sum_{r=0}^{k-1}\frac{(-1)^rr^2}{2^rr!(k-r-1)!}=\frac{k-3}{2^{k-1}(k-2)!}\mathbf{1}_{\{k\geq 2\}}\,,
	\ee
	and
	\be\label{eq:lc4}
	\sum_{r=0}^{k-1}\frac{(-1)^rr^3}{2^rr!(k-r-1)!}=\frac{-k^2+8k-13}{2^{k-1}(k-2)!}\mathbf{1}_{\{k\geq 2\}}.
	\ee
\end{lemma}
\begin{proof}
	The cases $k=1,2,3$ can be verified directly, so  we assume $k\ge 4$ in the following. 
	The first identity \eqref{eq:lc1}  is a simple application of the binomial expansion:
	\begin{align*}
	\sum_{r=0}^{k-1}\frac{(-1)^r}{2^rr!(k-r-1)!}
	&=\frac{1}{(k-1)!}\sum_{r=0}^{k-1}\binom{k-1}{r}\Big(-\frac{1}{2}\Big)^r\\
	&=\frac{1}{(k-1)!}\Big(1-\frac{1}{2}\Big)^{k-1}=\frac{1}{2^{k-1}(k-1)!}\,.
	\end{align*}

	For \eqref{eq:lc2}, note that $r\binom{k-1}{r}=0$ when $r=0$, and $r\binom{k-1}{r}=(k-1)\binom{k-2}{r-1}$ for $r\geq 1$. Hence, 
	\begin{align*}
	\sum_{r=0}^{k-1}\frac{(-1)^rr}{2^rr!(k-r-1)!}
	&=\frac{1}{(k-1)!}\sum_{r=1}^{k-1}(k-1)\binom{k-2}{r-1}\Big(-\frac{1}{2}\Big)^r\\
	&=\frac{1}{(k-2)!}\sum_{j=0}^{k-2}\binom{k-2}{j}\Big(-\frac{1}{2}\Big)^{j+1}\\
	&=-\frac{1}{2}\cdot\frac{1}{(k-2)!}\Big(1-\frac{1}{2}\Big)^{k-2}=-\frac{1}{2^{k-1}(k-2)!}\,.
	\end{align*}

	For \eqref{eq:lc3}, we use the fact that $r(r-1)\binom{k-1}{r}=0$ for $r=0,1$, and $r(r-1)\binom{k-1}{r}=(k-1)(k-2)\binom{k-3}{r-2}$ for $r\ge 2$. Thus we have
	\begin{align*}
	\sum_{r=0}^{k-1}\frac{(-1)^rr(r-1)}{2^rr!(k-r-1)!}
	&=\frac{1}{(k-1)!}\sum_{r=2}^{k-1}(k-1)(k-2)\binom{k-3}{r-2}\Big(-\frac{1}{2}\Big)^r\\
	&=\Big(-\frac{1}{2}\Big)^2\cdot\frac{(k-1)(k-2)}{(k-1)!}\cdot\Big(1-\frac{1}{2}\Big)^{k-3}\\
	&=\frac{1}{2^{k-1}(k-3)!}\,.
	\end{align*}
	Similarly we have 
	\[
	\sum_{r=0}^{k-1}\frac{(-1)^rr(r-1)(r-2)}{2^rr!(k-r-1)!}=-\frac{1}{2^{k-1}(k-4)!}\,.
	\]
	Combining these results, we obtain \eqref{eq:lc3}:
	\begin{align*}
	\sum_{r=0}^{k-1}\frac{(-1)^rr^2}{2^rr!(k-r-1)!}
	&=\sum_{r=0}^{k-1}\frac{(-1)^rr}{2^rr!(k-r-1)!}+\sum_{r=0}^{k-1}\frac{(-1)^rr(r-1)}{2^rr!(k-r-1)!}\\
	&=-\frac{1}{2^{k-1}(k-2)!}+\frac{1}{2^{k-1}(k-3)!}\\
	&=\frac{k-3}{2^{k-1}(k-2)!}\,.
	\end{align*}

	For \eqref{eq:lc4}, we use the decomposition
	\[
	r^3=r(r-1)(r-2)+3r^2-2r
	\]
	and the identities established above to derive the desired conclusion:
	\begin{align*}
	\sum_{r=0}^{k-1}\frac{(-1)^rr^3}{2^rr!(k-r-1)!}
	&=\sum_{r=0}^{k-1}\frac{(-1)^rr(r-1)(r-2)}{2^rr!(k-r-1)!}+3\sum_{r=0}^{k-1}\frac{(-1)^rr^2}{2^rr!(k-r-1)!}-2\sum_{r=0}^{k-1}\frac{(-1)^rr}{2^rr!(k-r-1)!}\\
	&=-\frac{1}{2^{k-1}(k-4)!}+3\cdot\frac{k-3}{2^{k-1}(k-2)!}+2\cdot\frac{1}{2^{k-1}(k-2)!}\\
	&=\frac{-k^2+8k-13}{2^{k-1}(k-2)!}\,. \quad \qedhere
	\end{align*}
	
\end{proof}

\begin{corollary}\label{cor: leading coef}
	For $k\ge2$, we have the following identities :
	\be\label{eq: lckr}
	\sum_{r=0}^{k-1}\frac{(-1)^r(k+r)}{2^rr!(k-r-1)!}=\frac{1}{2^{k-1}(k-1)!}\,,
	\ee
	\be\label{eq: lckr2}
	\sum_{r=0}^{k-1}\frac{(-1)^r(k+r)^2}{2^rr!(k-r-1)!}=\frac{-2k+3}{2^{k-1}(k-1)!}\,,
	\ee
	and
	\be\label{eq: lckr3}
	\sum_{r=0}^{k-1}\frac{(-1)^r(k+r)^3}{2^rr!(k-r-1)!}=\frac{-12k+13}{2^{k-1}(k-1)!} \,.
	\ee
\end{corollary}

\begin{proof}[Proof of \cref{thm: p-NC for kUF}]
	Fix $k\ge2$. We first derive the asymptotic expansions of the two quantities $\card{\kF(K_n)}$ (the number of $k$-forests of $K_n$) and $\card{\big\{ F\in\kF(K_n)\colon e,f\in F \big\}}$ (the number of $k$-forests of $K_n$ containing a fixed  pair of adjacent edges $(e,f)$).
	
	As $k$ is fixed and $n\to\infty$, for each $r\in[0,k-1]$, we have
	\begin{align*}
	\frac{(n-1)!}{(n-k-r)!}
	&=\prod_{j=1}^{k+r-1}\big(n-j\big)=n^{k+r-1}\prod_{j=1}^{k+r-1}\bigg(1-\frac{j}{n}\bigg)\\
	&=n^{k+r-1}\bigg[1-\frac{1+2+\cdots+(k+r-1)}{n}+O\bigg(\frac{1}{n^2}\bigg)\bigg]\\
	&=n^{k+r-1}\bigg[1-\bigg(\frac{(k+r)(k+r-1)}{2}\bigg)\frac{1}{n}+O\bigg(\frac{1}{n^2}\bigg)\bigg].
	\end{align*}
	Similarly,
	\begin{align*}
	\frac{(n-3)!}{(n-k-r-2)!}
	&=\prod_{j=3}^{k+r+1}\big(n-j\big)=n^{k+r-1}\prod_{j=3}^{k+r+1}\bigg(1-\frac{j}{n}\bigg)\\
	&=n^{k+r-1}\bigg[1-\bigg(\frac{(k+r+4)(k+r-1)}{2}\bigg)\frac{1}{n}+O\bigg(\frac{1}{n^2}\bigg)\bigg].
	\end{align*}
	Using these expansions, we first compute the expansion of $\card{\kF(K_n)}$:
	\begin{align*}
	\card{\kF(K_n)}
	&\stackrel{\eqref{eq: number of k-forest of K_n}}{=} \sum_{r=0}^{k-1}\frac{(-1)^r(k+r)}{2^rr!(k-r-1)!}\cdot \frac{n^{n-k-r-1}(n-1)!}{(n-k-r)!}\\
	&=\sum_{r=0}^{k-1}\frac{(-1)^r(k+r)}{2^rr!(k-r-1)!}\cdot n^{n-2}\bigg[1-\Big(\frac{(k+r)(k+r-1)}{2}\Big)\frac{1}{n}+O\bigg(\frac{1}{n^2}\bigg)\bigg]\\
	&=\frac{1}{2^{k-1}(k-1)!}n^{n-2}+\frac{5k-5}{2^{k-1}(k-1)!}n^{n-3}+O\big(n^{n-4}\big)\,,
	\end{align*}
	where the last equality follows from \cref{cor: leading coef}, which we use to compute the coefficients of $n^{n-2}$ and $n^{n-3}$.
	
	Next, we compute the expansion of $\card{\big\{ F\in\kF(K_n)\colon e,f\in F \big\}}$ for a fixed pair of adjacent edges $e,f$: 
	\begin{align*}
	\card{\big\{ F\in\kF(K_n)\colon e,f\in F \big\}}
	&\stackrel{\eqref{eq: kFef}}{=}	\sum_{r=0}^{k-1}\frac{(-1)^r(k+r+2) (n-3)!n^{n-k-r-3}}{2^rr!(k-r-1)! (n-k-r-2)!}  \\
	&=\sum_{r=0}^{k-1}\frac{(-1)^r(k+r+2) }{2^rr!(k-r-1)! } \cdot n^{n-4}\bigg[ 1-\bigg(\frac{(k+r+4)(k+r-1)}{2}\bigg)\frac{1}{n}+O\bigg(\frac{1}{n^2}\bigg) \bigg]\\
	&=\frac{3}{2^{k-1}(k-1)!}n^{n-4}+\frac{11k-11}{2^{k-1}(k-1)!}n^{n-5}+O\big(n^{n-6}\big)\,.
	\end{align*}
	
	We now compute $p_1$, the probability that both $e$ and $f$ are contained in a uniform 
	$k$-forest:
	\begin{align*}
	p_1& \coloneq \frac{\card{\big\{ F\in\kF(K_n)\colon e,f\in F \big\}}}{ \card{\kF(K_n)}}\\
	&=\frac{ \frac{3}{2^{k-1}(k-1)!}n^{n-4}+\frac{11k-11}{2^{k-1}(k-1)!}n^{n-5}+O\big(n^{n-6}\big) }{ \frac{1}{2^{k-1}(k-1)!}n^{n-2}+\frac{5k-5}{2^{k-1}(k-1)!}n^{n-3}+O\big(n^{n-4}\big)  }\\
	&=\frac{ 3n^{-2}+(11k-11)n^{-3}+O(n^{-4}) }{1+(5k-5)n^{-1}+O(n^{-2})} \\
	&=3n^{-2}-4(k-1)n^{-3}+ O\big(n^{-4}\big)\,.
	\end{align*}
	For fixed $k\ge 2$, we have 
	\[
	\frac{4(n-k)^2}{n^2(n-1)^2}=\frac{4}{n^2}-\frac{8(k-1)}{n^3}+O(n^{-4})
	\]
	and
	\[
	\frac{(n-k)\big[3n^2-(5+4k)n+6k\big]}{n^2(n-1)^2(n-2)}=3n^{-2}-7(k-1)n^{-3}+O(n^{-4})
	\]
	Hence,  when $n$ is sufficiently large, we have
	\[
	p_1\leq \frac{4(n-k)^2}{n^2(n-1)^2}
	\]
	and 
	\[
	p_1\geq \frac{(n-k)\big[3n^2-(5+4k)n+6k\big]}{n^2(n-1)^2(n-2)}\,.
	\]
	Combining these two inequalities with \cref{cor: 3.3} completes the proof of \cref{thm: p-NC for kUF}.
\end{proof}

\begin{remark}\label{rem: 3.10}
	In fact, for small values of $k$ such as $k=2,3,4$, one can verify that the p-NC property   holds for the uniform $k$-forest measure $\kUF$ on $\kF(K_n)$ for all $n\ge k$ (so that the measure $\kUF$ is well-defined). For instance, 
	substituting $k=2$ into \eqref{eq: kFef}, we obtain that for a pair of adjacent edges $(e,f)$, 
	\[
	\card{\big\{ F\in\twoF(K_n) \colon e,f\in F   \big\}} =\frac{(n-3)(3n+20)n^{n-6}}{2}\,.
	\]
	Recall from \eqref{eq: number of 2-forest of K_n} that 
	\[
	\card{\twoF(K_n)}=\frac{(n-1)(n+6)n^{n-4}}{2}\,.
	\]
	Hence, in this case  (assuming $n\ge 3$ so that $K_n$  contains a pair of adjacent edges), we have
	\[
	p_1 \coloneq \frac{\card{\big\{ F\in\kF(K_n)\colon e,f\in F \big\}}}{ \card{\kF(K_n)}}
	=\frac{(n-3)(3n+20)}{n^2(n-1)(n+6)}\,.
	\]
	For $k=2$, we have 
	\[
	\frac{4(n-k)^2}{n^2(n-1)^2}=\frac{4(n-2)^2}{n^2(n-1)^2} \quad \text{ and }\quad 
	\frac{(n-k)\big[3n^2-(5+4k)n+6k\big]}{n^2(n-1)^2(n-2)}=\frac{(n-3)(3n-4)}{n^2(n-1)^2}\,.
	\]
	It is straightforward to verify that the following two inequalities hold for all $n\ge3$:
	\[
	p_1\leq \frac{4(n-2)^2}{n^2(n-1)^2}
	\]
	and
	\[
	p_1\geq  \frac{(n-3)(3n-4)}{n^2(n-1)^2}\,.
	\]
	Thus, by \cref{cor: 3.3} the p-NC property  holds for $\twoUF$ on $K_n$ whenever $n\ge 3$. The same calculation can be carried out for $k=3,4$; we omit the resulting formulas. 
\end{remark}

\section{Further remarks and an open problem} \label{sec: remarks and ques}

%We conclude by discussing two issues suggested by Theorem~\ref{thm: p-NC for kUF}. The first concerns the behavior of pairwise negative correlation under truncation. The second concerns the dependence of fixed-component forest measures on the number of components. Together, these observations emphasize that the complete-graph setting is substantially more rigid than the corresponding problem on a general graph.

We conclude with several observations concerning the extent to which the preceding results depend on the symmetry of the complete graph. We first recall that pairwise negative correlation need not be preserved under truncation, even for spanning forests. We then record a related failure of monotonicity for one-edge marginals as the number of components varies. Both phenomena contrast with the behavior on highly symmetric graphs. Finally, we formulate an open problem concerning fixed-component forests of the complete graph.

\subsection{Truncation and negative correlation}

Let $S$ be a finite set and let $\mu$ be a probability measure on $\{0,1\}^S$. For every integer $m$ such that $\mu(|\omega|=m)>0$, define the $m$-truncation of $\mu$ by
\[
\mu_m(\,\cdot\,) \coloneq \mu(\,\cdot\mid |\omega|=m).
\]
Pairwise negative correlation is not, in general, preserved under this operation. In the forest setting, this failure already occurs for two-component spanning forests.

\begin{example}\label{example:HSW2022}
	Let $G$ be the graph depicted in Figure~\ref{fig: counterexample for p-NC for 2-forests}, taken from Example 2 of \cite{HSW2022}. For the pair of edges $e,f$ indicated in the figure, direct enumeration gives $\card{\twoF(G)}=384$, with $112, 272$ and $80$ two-component forests containing $e$,  $f$ and both $e,f$, respectively. Hence we have
	\[
	\twoUF[\omega(e)=\omega(f)=1] = \frac{80}{384} > \frac{112}{384}\cdot\frac{272}{384} = \twoUF[\omega(e)=1]\cdot\twoUF[\omega(f)=1],
	\]
	where $\twoUF$ denotes the uniform probability measure on the set of $2$-forests of $G$. Thus $\twoUF$ fails the p-NC property.
	
	\begin{figure}[H]
		\centering
		\begin{tikzpicture}[scale=0.8, text height=1.5ex,text depth=.25ex] 
		\draw[fill=black] (6,8) circle [radius=0.06];
		\node[above] at (6,8) {$A$};
		\draw[fill=black] (6,0) circle [radius=0.06];
		\node[below] at (6,0) {$B$};
		\draw[color=blue,thin] (6,0)--(6,8);
		\node[color=black,below] at (6.2,4.15) {$e$};
		\draw[color=teal,thin] (6,0)--(0,4)--(6,8);
		\draw[color=teal,thin] (6,0)--(2,4)--(6,8);
		\draw[color=teal,thin] (6,0)--(4,4)--(6,8);
		\draw[color=teal,thin] (6,0)--(8,4)--(6,8);
		\draw[color=teal,thin] (6,0)--(10,4)--(6,8);
		\draw[fill=black] (0,4) circle [radius=0.06];
		\node[color=black,below] at (0,4) {$P_1$};
		\draw[fill=black] (2,4) circle [radius=0.06];
		\node[color=black,below] at (2,4) {$P_2$};
		\draw[fill=black] (4,4) circle [radius=0.06];
		\node[color=black,below] at (4,4) {$P_{3}$};
		\draw[fill=black] (8,4) circle [radius=0.06];
		\node[color=black,below] at (8,4) {$P_{4}$};
		\draw[fill=black] (10,4) circle [radius=0.06];
		\node[color=black,below] at (10,4) {$P_{5}$};
		\draw[color=purple,thin] (12,4)--(6,8);
		\draw[fill=black] (12,4) circle [radius=0.06];
		\node[color=black,below] at (12,4) {$Q$};
		\node[color=black,below] at (9.7,6.5) {$f$};
		\end{tikzpicture}
		\caption{A graph for which the uniform $2$-forest measure fails the p-NC property (adapted from \cite{HSW2022}).}
		\label{fig: counterexample for p-NC for 2-forests}
	\end{figure}
\end{example}

There is also a planar-dual interpretation. If $H$ is planar with dual $H^\dagger$, then p-NC for a pair $e,f$ under the uniform two-component forest measure on $H$ is equivalent to p-NC for the dual pair $e^*,f^*$ under the uniform probability measure on connected unicyclic spanning subgraphs of $H^\dagger$. Thus the dual of the graph in Fig.~\ref{fig: counterexample for p-NC for 2-forests} also gives a counterexample for the latter model.

Example~\ref{example:HSW2022} shows, in particular, that the fixed-component problem cannot be deduced from negative-correlation properties of the unconditioned uniform forest measure by a general truncation principle. The positive result of Theorem~\ref{thm: p-NC for kUF} therefore reflects additional structure of the complete graph.

\subsection{Dependence on the number of components}

%
%
%For a finite connected graph $G=(V,E)$, set
%\[
%b_m := \bigl|\cF^{(|V|-m)}(G)\bigr|,
%\qquad 0\le m\le |V|-1,
%\]
%so that $b_m$ is the number of spanning forests of $G$ with exactly $m$ edges. These numbers form the initial segment of the independence sequence of the graphic matroid of $G$. The full independence sequence is ultra log-concave by the resolution of Mason's conjecture; see \cite{ALGV2018,BH2018}.

\begin{remark}
It is natural to ask whether, for a finite connected graph \(G=(V,E)\) and a fixed edge \(e\in E\), the marginal $\displaystyle \mathbb P_{k\mathrm{UF}}(\omega(e)=1) $
is nonincreasing in \(k\). This question was posed in an earlier arXiv version of this paper \cite{TZ2026}. Example~\ref{example:HSW2022} already shows that such a monotonicity statement fails in general. Indeed, for the edge \(e\) indicated in Fig.~\ref{fig: counterexample for p-NC for 2-forests}, the same graph has $112$ spanning trees, $32$ of which contain $e$. Hence
\[
\mathbb P_{2\mathrm{UF}}(\omega(e)=1)
=\frac{112}{384}
>
\frac{32}{112}
=\mathbb P_{\mathrm{UST}}(\omega(e)=1).
\]
Thus, passing from spanning trees to spanning forests with two components can increase the inclusion probability of a fixed edge.

The planar dual of the graph in Example~\ref{example:HSW2022} gives a corresponding counterexample to the analogous monotonicity question, posed in \cite{TZ2026}, comparing the uniform spanning tree measure with the uniform measure on connected unicyclic spanning subgraphs. Hence neither of these monotonicity properties holds for general finite graphs.
\end{remark}

\begin{remark}
In contrast, the monotonicity holds for every finite connected edge-transitive graph. Indeed, if \(G=(V,E)\) is edge-transitive, then invariance under graph automorphisms and the fact that every \(k\)-component spanning forest has \(|V|-k\) edges give, for every \(e\in E\),
\[
\mathbb P_{k\mathrm{UF}}(\omega(e)=1)
=\frac{|V|-k}{|E|}.
\]
Consequently, the marginal is strictly decreasing in \(k\) whenever the corresponding forest measures are defined.
\end{remark}

Our main theorem leaves open the behavior outside the asymptotic regime. Theorem~\ref{thm: p-NC for kUF} proves p-NC for every fixed $k$ when $n$ is sufficiently large, while the explicit calculations in Remark~\ref{rem: 3.10} establish the property throughout the admissible range for several small values of $k$. It remains unclear whether the restriction to sufficiently large $n$ is necessary. We therefore record the following question.

\begin{question}\label{ques:all-n-kUF}
	For which pairs $(n,k)$, with $2\le k\le n$, does the uniform probability measure on $k$-component spanning forests of $K_n$ satisfy the p-NC property? In particular, does p-NC hold for all such pairs $(n,k)$?
\end{question}

Gupta's recent result \cite{Gupta2026UF} shows that the unconditioned uniform forest measure on $K_n$ satisfies p-NC for every $n\ge2$. This does not settle Question~\ref{ques:all-n-kUF}: pairwise negative correlation need not be preserved under truncation, nor can it in general be recovered from the corresponding truncations. Thus the unconditioned and fixed-component problems concern distinct aspects of negative dependence for uniform forests on the complete graph.

%\bibliography{kUF_ref}
%\bibliographystyle{plain}
%
%\end{document}

%% If kUF.tex does NOT already contain its bibliography, use:
%%
%\bibliographystyle{plain}
%\bibliography{kUF_ref}

\putbib[NCUS_ref]

\end{bibunit}

\clearpage

%%%%%%%%%%%%%%%%%%%%%%%%%%%%%%%%%%%%%%%%%%%%%%%%%%%%%%%%%%%%
%%%%%%%%%%%%%%%%%%%%%%%%%%%%%%%%%%%%%%%%%%%%%%%%%%%%%%%%%%%%
%% PART III: kUC
%%%%%%%%%%%%%%%%%%%%%%%%%%%%%%%%%%%%%%%%%%%%%%%%%%%%%%%%%%%%
%%%%%%%%%%%%%%%%%%%%%%%%%%%%%%%%%%%%%%%%%%%%%%%%%%%%%%%%%%%%

\part{Uniform $k$-excess connected spanning subgraphs}

\resetpartcounters

\papertitle{
Negative Correlation for Uniform $k$-Excess Connected
Spanning Subgraphs of the Complete Graph
}

\begin{bibunit}[plain]

\begin{paperabstract}
We investigate the pairwise negative correlation property for the
uniform probability measure on connected spanning subgraphs of the
complete graph $K_n$ with prescribed excess $k$.
We prove that, for every fixed integer $k\ge0$, this measure satisfies
pairwise negative correlation for all sufficiently large $n$.
The proof combines first- and second-moment identities with singularity
analysis of exponential generating functions.
\end{paperabstract}

\section{Introduction}

%\subsection{Motivation}

The enumeration of connected labelled graphs with a prescribed number of
vertices and edges is a classical topic in combinatorics.  A convenient
parameter in the sparse regime is the \emph{excess}: a connected graph on
$n$ vertices and $n+k$ edges is said to have excess $k$.  The case
$k=-1$ corresponds to trees, while $k=0$ corresponds to unicyclic graphs.
R\'enyi \cite{Renyi1959b} determined the asymptotic number of unicyclic
graphs, and Wright~\cite{Wright1977one,Wright1978two,Wright1980three} developed a
systematic theory for connected labelled graphs of fixed excess.  In
particular, for each fixed $k$, precise asymptotic information is
available for the number of connected spanning subgraphs of $K_n$ with
excess $k$.

In the present paper we study a probabilistic question for these
classical ensembles.  Suppose that a connected spanning subgraph of
$K_n$ with excess $k$ is chosen uniformly at random.  Given two distinct
edges, are their inclusion events negatively correlated?

Let $G=(V,E)$ be a finite graph and let $\mu$ be a probability measure on
$\{0,1\}^E$.  For $\omega\in\{0,1\}^E$, write
\[
\eta(\omega):=\{e\in E:\omega(e)=1\}.
\]
We say that $\mu$ satisfies the \emph{pairwise negative correlation}
(p-NC) property if, for every pair of distinct edges $e,f\in E$,
\begin{equation}\label{eq: def p-NC}
\mu\bigl(\omega(e)=\omega(f)=1\bigr)
\le
\mu\bigl(\omega(e)=1\bigr)
\mu\bigl(\omega(f)=1\bigr).
\end{equation}

Questions of this type arise naturally in the study of negatively
dependent spanning-subgraph measures, including the conjectural
negative dependence of the random-cluster model for $q<1$; see, for
example,~\cite{Grimmett_Winkler2004,Grimmett2006,Pemantle2000}.

For a finite connected graph $G$, let
\[
\sC=\sC(G)
:=
\bigl\{
\omega\in\{0,1\}^E:
(V,\eta(\omega))\text{ is connected}
\bigr\},
\]
and let $\UC$ denote the uniform probability measure
on $\sC$.  For an integer $k\ge -1$, define
\[
\kC=\kC(G)
:=
\bigl\{
\omega\in\sC:
|\eta(\omega)|-|V|=k
\bigr\},
\]
and write $\kUC$ for the uniform probability measure
on $\kC$, whenever this set is nonempty.  Thus
$\kC$ consists of the connected spanning subgraphs with
excess $k$.  In particular, $\sC^{(-1)}$ is the set of spanning
trees, while $\sC^{(0)}$ is the set of connected unicyclic
spanning subgraphs.

The case $k=-1$ is classical: the uniform spanning tree measure is
negatively associated, and hence satisfies p-NC; see
Feder--Mihail~\cite{FM1992} and
Lyons--Peres~\cite{LP2016}.  Our main result gives the
corresponding asymptotic statement for every fixed nonnegative excess.

\begin{theorem}\label{thm: p-NC for kUC}
	For each fixed integer $k\ge0$, there exists a constant $N=N(k)>0$ such that, for all $n\geq N$, the uniform probability measure $\kUC$ on $\kC(K_n)$ satisfies the p-NC property.
\end{theorem}

There are two features of the complete graph that make this problem
particularly tractable.  First, every configuration in
$\kC(K_n)$ contains exactly $n+k$ edges.  Hence the one-edge
marginal is determined immediately by symmetry.  Second, there are only
two orbits of unordered pairs of distinct edges, according to whether
the two edges are adjacent or non-adjacent.  A second-moment identity
relates the corresponding two-edge probabilities through the
degree-square statistic
\[
\sum_{x\in V(K_n)}\deg(x;\omega)^2.
\]
Consequently, it is enough to obtain a sufficiently precise estimate
for the probability that a fixed adjacent pair is present.

The remaining input is enumerative.  Let
\[
C_{n,n+k}
:=
\bigl|\kC(K_n)\bigr|
\]
and, for a fixed adjacent pair of edges $e,f$, let
\[
C^{e,f}_{n,n+k}
:=
\bigl|
\{\omega\in\kC(K_n):
\omega(e)=\omega(f)=1\}
\bigr|.
\]
Using Wright's generating functions for connected graphs of fixed
excess and singularity analysis, we obtain
\[
\frac{C^{e,f}_{n,n+k}}{C_{n,n+k}}
=
\frac{3}{n^2}
+
\frac{9(k+1)}{n^3}
+
O(n^{-7/2})
\]
for every fixed $k\ge0$.  The second-order term is essential: the leading
term alone is enough for the adjacent-edge inequality, but the
non-adjacent case requires the finer expansion.

The model studied here is one of three closely related uniform
spanning-subgraph models considered in the earlier combined
manuscript~\cite{TZ2026}.  The unconditioned connected-subgraph
model chooses uniformly from all connected spanning subgraphs of
$K_n$, while the present model conditions on the number of edges.
On the forest side, one may instead prescribe the number of connected
components.  Thus the fixed-component forest and fixed-excess connected
models are two natural fixed-edge-count ensembles.  Their moment
reductions are closely related, but their enumerative inputs are quite
different: the forest problem is handled through explicit forest
enumeration and edge contraction, whereas the present problem relies
on the generating-function theory of connected graphs with fixed
excess.

It is important that Theorem~\ref{thm: p-NC for kUC} is not a formal consequence
of a negative-correlation statement for the unconditioned connected
measure.  Indeed,
\[
\kUC(\,\cdot\,)
=
\UC
\bigl(
\,\cdot\,
\,\big|\,
|\eta(\omega)|=n+k
\bigr),
\]
and pairwise negative correlation is not preserved under such
fixed-cardinality conditioning in general.  In Section~\ref{sec: kUC-remarks and ques}
we give a planar example showing that even the uniform connected
unicyclic measure may fail p-NC.

\medskip

\noindent
\textbf{Organization of the paper.}
Section~\ref{sec: p-NC for kUC} is devoted to the proof of
Theorem~\ref{thm: p-NC for kUC}.  We first derive the first- and second-moment
identities that reduce p-NC to an estimate for a fixed adjacent pair of
edges.  We then recall the generating-function framework for connected
labelled graphs of fixed excess and the singularity estimates needed
for coefficient extraction.  The unicyclic case $k=0$ is treated in
detail, and the cases $1\le k\le5$ are also worked out explicitly.
For arbitrary fixed $k\ge 6$, the same argument applies, but the coefficient
calculations become lengthy; these computations are collected in the
supplementary material.  Section~\ref{sec: kUC-remarks and ques} records a
counterexample for general graphs and discusses the remaining
fixed-excess problem on complete graphs.

\medskip

\noindent
\textbf{Asymptotic notation.}
We write $f(n)\asymp g(n)$ if there exist constants
$c_1,c_2>0$ such that
$c_1g(n)\le f(n)\le c_2g(n)$ for all sufficiently large $n$.
We use $O(\cdot)$ and $o(\cdot)$ with their standard meanings.

\section{Uniform $k$-excess connected spanning subgraphs of complete graphs} \label{sec: p-NC for kUC}

In this section we prove Theorem~\ref{thm: p-NC for kUC}. A first-moment argument readily yields that $\kUC\big[\omega(e)=1\big]=\frac{n+k}{\binom{n}{2}}$. A second-moment argument relates the probability $\kUC\big[\omega(e)=\omega(f)=1\big]$ for a pair of adjacent edges to that for a pair of non-adjacent edges. It therefore suffices to provide a suitable estimate of  $\kUC\big[\omega(e)=\omega(f)=1\big]$ for a pair of adjacent edges $e,f$.

\subsection{Proof of Theorem~\ref{thm: p-NC for kUC} using Proposition~\ref{prop: key prop for k-excess}}\label{subsec: proof of thm kUC}

For an integer $k\ge-1$ and a pair of adjacent edges $e,f$ in $K_n$, we write 
\[
C_{n,n+k}\coloneq \card{\sC^{(k)}(K_n)} \,\text{ and }\, C_{n,n+k}^{e,f}\coloneq \card{\{\omega\in \sC^{(k)}(K_n)\colon \omega(e)=\omega(f)=1 \}}\,
\]
for brevity. 
We use singular analysis techniques \cite{FO1990} to derive the asymptotics of $C_{n,n+k}$ and $ C_{n,n+k}^{e,f}$, and then establish the following key estimate for Theorem~\ref{thm: p-NC for kUC}.
In this section, we give a detailed proof of Proposition~\ref{prop: key prop for k-excess} in the base case $k=0$. The argument for $k\ge1$ is analogous. 
For $k\in\{1,\ldots,5\}$ we record the resulting explicit
asymptotic expansions and illustrate the computation in detail for
$k=1$. The remaining case $k\ge6$, for which the coefficient
computations are more involved, is treated in the supplementary
material; see also the appendix of \cite{TZ2026}.

\begin{proposition}\label{prop: key prop for k-excess}
	Suppose $k\ge0$ is a fixed integer,  edges $e$ and $f$ are a pair of adjacent edges of $K_n$, and $p_1\coloneq \kUC\big[\omega(e)=\omega(f)=1\big]$. Then there exists constants $\big\{\alpha_{k,i},\beta_{k,i},i=1,2,3  \}$  such that as $n\to\infty$ 
	\be\label{eq: cnnk asym}
	C_{n,n+k}=\sum_{i=1}^{3}\alpha_{k,i}n^{n+\frac{3k-i}{2}}+O\left( n^{n+\frac{3k-4}{2}} \right)\,,
	\ee
	and
	\be\label{eq: cnnkef asym}
	C_{n,n+k}^{e,f}
	=\sum_{i=1}^{3}\beta_{k,i}n^{n+\frac{3k-4-i}{2}}+O\left( n^{n+\frac{3k-8}{2}} \right)
	\ee
	Moreover, these constants $\big\{\alpha_{k,i},\beta_{k,i},i=1,2,3  \}$ satisfy  
	\be\label{eq: p1 kUC asym}
	p_1=\frac{	 C_{n,n+k}^{e,f}}{C_{n,n+k}}=\frac{3}{n^2}+\frac{9(k+1)}{n^3}+O(n^{-\frac{7}{2}})\,.
	\ee
\end{proposition}

We start with the first- and second-moment identities for the $\kUC$ measures on $K_n$; the proof is elementary and omitted. 
\begin{lemma}\label{lem: first and second moment for k cycles}
	Consider the $\kUC$ measure on the complete graph $K_n$ and assume that $n+k\leq \binom{n}{2}$ so that the measure $\kUC$ is well-defined. Let $e,f$ be adjacent edges in $K_n$ and $e,e'$ be non-adjacent edges in $K_n$. Define $p_1\coloneq \kUC\big[\omega(e)=\omega(f)=1\big]$ and $p_2\coloneq \kUC\big[\omega(e)=\omega(e')=1\big]$. 
	Then 
	\begin{enumerate}
		\item[(1)] $p_0\coloneq \kUC\big[\omega(e)=1\big]=\frac{n+k}{\binom{n}{2}}$;
		
		\item[(2)] for $n\ge3 $, the probability $p_1$ satisfies
		\[
		p_1=\frac{\bE\Big[\sum_{x\in V(K_n)}{\deg(x;\omega)^2}\Big]-2(n+k)}{n(n-1)(n-2)}\,,
		\]
		
		\item[(3)]  for $n\ge4$, the probability $p_2$ satisfies
		\[
		p_2=\frac{4\bigg((n+k)(n+k+1)-\bE\Big[\sum_{x\in V(K_n)}{\deg(x;\omega)^2}\Big]\bigg)}{n(n-1)(n-2)(n-3)}\,.
		\]
		In particular, 
		\[
		p_2=\frac{4\big[(n+k)(n+k-1)-p_1n(n-1)(n-2)\big]}{n(n-1)(n-2)(n-3)}\,.
		\]
	\end{enumerate}
	Accordingly,  the inequality \eqref{eq: def p-NC}  for a pair of adjacent edges under the measure $\kUC$ on $K_n$ is equivalent to 
	\[
	p_1\leq \frac{4(n+k)^2}{n^2(n-1)^2}=\frac{4}{n^2}+\frac{8(k+1)}{n^3}+O(n^{-4})\,,
	\]
	while the inequality \eqref{eq: def p-NC}  for a pair of non-adjacent edges is equivalent to
	\[
	p_1\geq \frac{(n+k)\big[2(k+1)(2n-3)+3(n-1)(n-2)\big]}{n^2(n-1)^2(n-2)}=\frac{3}{n^2}+\frac{7(k+1)}{n^3}+O(n^{-4})\,.
	\]
\end{lemma}

\begin{proof}[Proof of Theorem~\ref{thm: p-NC for kUC} assuming Proposition~\ref{prop: key prop for k-excess}]
	This follows directly from  Lemma~\ref{lem: first and second moment for k cycles} and the estimate \eqref{eq: p1 kUC asym}. 
\end{proof}

\subsection{Preliminaries}
The enumeration of labelled connected graphs by number of vertices and edges was initiated in a series of papers by Wright \cite{Wright1977one,Wright1978two,Wright1980three}. A key result of Wright shows that the exponential generating function for connected graphs with a fixed excess can be expressed as a rational function in the tree generating function $T(z)$. Here $T(z)\in\bbC[[z]]$ is defined by
\[
T(z)\coloneq \sum_{n\ge1} n^{n-1}\frac{z^n}{n!}\,,
\]
where $T_n=n^{n-1}$ is the number of rooted trees on $n$ labelled vertices, and $\bbC[[z]]$ denotes the ring of formal power series in  $z$ with complex coefficients.

For $k\ge-1$, recall that $C_{n,n+k}\coloneq \card{\sC^{(k)}(K_n)}$, i.e., $C_{n,n+k}$ counts the number of connected, spanning subgraphs of $K_n$ with excess $k$. R\'{e}nyi \cite{Renyi1959b}  already showed that the number  $C_{n,n}$---the number of unicyclic, spanning subgraphs of $K_n$---satisfies $C_{n,n}\sim\sqrt{\frac{\pi}{8}}n^{n-\frac{1}{2}}$. We require a more refined version of this result. 

Let $W_k(z)\in\bbC[[z]]$ denote the exponential generating function associated with $C_{n,n+k}$, defined by
\[
W_k(z)\coloneq \sum_{n\geq1}C_{n,n+k}\frac{z^n}{n!}\,.
\]
We state Theorem 1 of \cite{FSS2004} below for future reference;  parts (i), (ii) and (iii) of this theorem are attributed to Cayley (and his contemporaries), R\'{e}nyi, and Wright, respectively. 
\begin{theorem}\label{thm: FSS2004 thm1}
	\begin{enumerate}
		\item[(i)] The exponential generating function of unrooted trees is 
		\[
		W_{-1}(z)=T(z)-\frac{1}{2}T^2(z)\,.
		\] 
		
		\item[(ii)] The exponential generating function of connected graphs with excess $0$ (unicyclic graphs) is 
		\[
		W_0(z)=\frac{1}{2}\log\frac{1}{1-T(z)}-\frac{1}{2}T(z)-\frac{1}{4}T^2(z)\,.
		\]
		
		\item[(iii)] The exponential generating function of connected graphs with excess $k\geq1$ is a rational function of $T(z)$: there exist polynomials $A_k$ such that 
		\[
		W_k(z)=\frac{A_k\big(T(z)\big)}{\big[1-T(z)\big]^{3k}}\,.
		\] 
		For instance (see page~135 of \cite{FS2009}),
		\[
		W_1(z)=\frac{1}{24}\frac{T^4(z)\big(6-T(z)\big)}{\big[1-T(z)\big]^3}\,,\quad
		W_2(z)=\frac{1}{48}\frac{T^4(z)\big[2+28T(z)-23T^2(z)+9T^3(z)-T^4(z)\big]}{\big[1-T(z)\big]^6}
		\,.
		\]
	\end{enumerate}
\end{theorem}

The asymptotics of $C_{n,n}$ can be deduced from known results and will be presented in Section~\ref{sec: number of unicyclic subgraphs}. 
Our estimate for the asymptotics of  $C_{n,n}^{e,f}$ will be reduced to find the coefficients of $z^n$ in certain polynomials involving the derivatives of
$W_k(z)$ (see Section~\ref{subsubsec: 4.3.2} for details). Using the relations between $W_k(z)$ and $T(z)$, these problems are further reduced to finding the coefficients of $z^n$ in formal power series of the form $\frac{T^{\alpha}(z)}{\big[1-T(z)\big]^{\beta}}$ for some nonnegative integers $\alpha,\beta$. In this section, we collect a few facts about $T(z)$ and then outline  how to 
determine
the coefficients of $z^n$ in a formal power series of the form $\frac{T^{\alpha}(z)}{\big[1-T(z)\big]^{\beta}}$.

\begin{lemma}\label{lem: facts of T(z)}
	The formal power series $T(z)\coloneq \sum_{n\ge1} n^{n-1}\frac{z^n}{n!}$ has the following properties:
	\begin{enumerate}
		\item[(1)] It is well-known (see, e.g., equation~(44), p.~127 of \cite{FS2009}) that $T(z)$ satisfies 
		\be\label{eq:T1}
		T(z)=ze^{T(z)}\,.
		\ee

		\item[(2)] The derivative of $T(z)$ satisfies(see, e.g., formula~(2.7) in the proof of Lemma~2.1 of \cite{Stark2011})
		\be\label{eq: T'}
		T'(z)=\frac{T(z)}{z\big[1-T(z)\big]}\,.
		\ee
		
		\item[(3)] For an integer $k\in[1,n]$, we have\footnote{Here and in the sequel, we use $\big[z^n\big]f(z)\coloneq f_n$ to denote the coefficient of $z^n$ in the formal power series $f(z)=\sum_{n\ge0}f_nz^n$.}  (see, e.g., Proposition~1 of \cite{FGKP1995})
		\be\label{eq:T3}
		\big[z^n\big]T^k(z)=\frac{kn^{n-k-1}}{(n-k)!}.
		\ee
		
		\item[(4)] For an integer $k\in[1,n]$, we have (see, e.g.,  formula~(2.13) in Lemma~2.3 of \cite{Stark2011}) 
		\be\label{eq:T4}
		\big[z^n\big]\frac{T^k(z)}{1-T(z)}=\frac{n^{n-k}}{(n-k)!}\,.
		\ee
	\end{enumerate}
\end{lemma}
Formulas \eqref{eq:T3} and \eqref{eq:T4} provide the coefficients of $z^n$ in the power series $\frac{T^{\alpha}(z)}{\big[1-T(z)\big]^{\beta}}$ when $\beta\in\{0,1\}$. For $\beta\ge2$, we require the method of singularity analysis \cite{FO1990},  and the procedure proceeds as follows. 
\begin{itemize}
	\item[(a)] The formal power series $T(z)$ has radius of convergence $e^{-1}$. If we set $x=ez$, then  $T(z)$ can be viewed as an analytic function of $x$ in a domain $\Delta(\eta,\phi)\coloneq\{x\in\bbC\colon  |x|\leq 1+\eta,\big|\mathrm{Arg}(x-1)\big| \geq \phi \}$  for some $\eta>0,\phi\in(0,\frac{\pi}{2})$ (see Example VI.8 in \cite{FS2009}). 
	Moreover, if we write $w=\sqrt{1-ez}$ and let  $x=ez\to 1$ in $\Delta(\eta,\phi)$, then $T(z)$ admits the following asymptotic expansion at its dominant singularity $e^{-1}$ (see, e.g., formula~(4.3) in \cite{Stark2011}):
	\be\label{eq:expansion of T(z)}
	T(z)=1-\sqrt{2}w+\frac{2}{3}w^2-\frac{11}{36}\sqrt{2}w^3+\frac{43}{135}w^4-\frac{769}{4320}\sqrt{2}w^5+O(w^6)
	\ee

	\item[(b)] Using \eqref{eq:expansion of T(z)} and Newton's binomial theorem, we can derive the asymptotic expansion of $\frac{T^{\alpha}(z)}{\big[1-T(z)\big]^{\beta}}$ in powers of $w$ as $x=ez\to1$ in $\Delta(\eta,\phi)$, i.e., 
	\be\label{eq: computable coef}
	\frac{T^{\alpha}(z)}{\big[1-T(z)\big]^{\beta}}=c_0w^{-\beta}+c_1w^{-\beta+1}+c_2 w^{-\beta+2}+\cdots\,,
	\ee
	with explicitly computable constants $c_0,c_1,c_2\ldots$. For instance, 
	\[
	c_0=2^{-\frac{\beta}{2}},\quad c_1=2^{-\frac{\beta}{2}}\big(\frac{\sqrt{2}\beta}{3}-\sqrt{2}\alpha\big)\,,
	\text{ and } c_2=2^{-\frac{\beta}{2}}\frac{36\alpha^2+60\alpha+4\beta^2-7\beta-24\alpha\beta}{36}\,.
	\]
	The computation of the constants $c_i$ is deferred to the supplementary material. 
	\item[(c)] The singular analysis method developed in \cite{FO1990} then allows us to extract the asymptotic behavior of the coefficients of $z^n$ in $\frac{T^{\alpha}(z)}{\big[1-T(z)\big]^{\beta}}$. Indeed, formulas (2.1) and (2.2) in \cite{FO1990} show that for $s\in\bbR\setminus\{0,1,2,3,\ldots\}$,  
	\[
	\big[z^n\big](1-z)^{s}=\binom{n-s-1}{n}\sim \frac{n^{-s-1}}{\Gamma(-s)}\bigg[1+\frac{s(s+1)}{2n}+\frac{s(s+1)(s+2)(3s+1)}{24n^2}+O(n^{-3})\bigg]\,.
	\]
	Applying this with $s=\frac{k}{2}$,  we obtain for $k\in\bbR\setminus\{0,2,4,6,8,\ldots\}$, 
	\be\label{eq:coef of w}
	\big[z^n\big]w^{k}=\big[z^n\big](1-ez)^{\frac{k}{2}}
	=\frac{e^nn^{-\frac{k}{2}-1}}{\Gamma(-\frac{k}{2})}\bigg[1+\frac{k(k+2)}{8n}+O(n^{-2})\bigg]\,.
	\ee

	\item[(d)] Finally, using Corollary 3 in \cite{FO1990} we  deduce the asymptotic behavior of $\big[z^n\big]\frac{T^{\alpha}(z)}{\big[1-T(z)\big]^{\beta}}$ for integers $\alpha\ge0,\beta\ge2$:
	\begin{align}\label{eq:T_alpha_beta}
	\big[z^n\big]\frac{T^{\alpha}(z)}{\big[1-T(z)\big]^{\beta}}
	&=c_0\big[z^n\big]w^{-\beta}+c_1\big[z^n\big]w^{-\beta+1}+c_2\big[z^n\big]w^{-\beta+2}+O(e^nn^{\frac{\beta}{2}-\frac{5}{2}})\nonumber\\
	&=c_0\frac{e^{n}n^{\frac{\beta}{2}-1}}{\Gamma(\frac{\beta}{2})}\big[1+\frac{\beta(\beta-2)}{8n}+O(n^{-2})\big]+c_1\frac{e^{n}n^{\frac{\beta-1}{2}-1}}{\Gamma(\frac{\beta-1}{2})}\big[1+O(n^{-1})\big] \nonumber\\
	& \quad +c_2\frac{e^nn^{\frac{\beta-2}{2}-1}}{\Gamma(\frac{\beta-2}{2})}\big[1+O(n^{-1})\big]\cdot \mathbf{1}_{\beta>2}+O(e^nn^{\frac{\beta}{2}-\frac{5}{2}}) \nonumber\\
	&=\frac{2^{-\frac{\beta}{2}}}{\Gamma(\frac{\beta}{2})}e^nn^{\frac{\beta}{2}-1}
	+\frac{2^{-\frac{\beta}{2}}\big(\frac{\sqrt{2}\beta}{3}-\sqrt{2}\alpha\big)}{\Gamma\big(\frac{\beta-1}{2}\big)}e^nn^{\frac{\beta}{2}-\frac{3}{2}} \nonumber\\
	&\quad + \bigg[\frac{2^{-\frac{\beta}{2}} \beta(\beta-2) }{8\Gamma(\frac{\beta}{2})}+\frac{c_2}{\Gamma(\frac{\beta-2}{2})}\cdot \mathbf{1}_{\beta>2}\bigg] e^nn^{\frac{\beta}{2}-2} +O(e^nn^{\frac{\beta}{2}-\frac{5}{2}})\,.
	\end{align} 
\end{itemize}

\subsection{The case $k=0$ of Proposition~\ref{prop: key prop for k-excess}}

\subsubsection{The asymptotics of $C_{n,n}$}\label{sec: number of unicyclic subgraphs}

\begin{proposition}\label{prop: asymptotic of cnn}
	The number  of unicyclic subgraphs of $K_n$ has the following asymptotic behavior:
	\be\label{eq: cnn}
	C_{n,n}=\sqrt{\frac{\pi}{8}}n^{n-\frac{1}{2}}-\frac{7}{6}n^{n-1}+\frac{1}{24}\sqrt{\frac{\pi}{2}}\,n^{n-\frac{3}{2}}+O(n^{n-2})\,.
	\ee
\end{proposition}
The asymptotic expansion \eqref{eq: cnn} is simply a combination of known results;  we present the details here for readers' convenience. 
\begin{proof}[Proof of Proposition~\ref{prop: asymptotic of cnn}]
	
	Since $T(z)$ satisfies \eqref{eq:T1}, Lagrange's inversion formula gives 
	\[
	\log\frac{1}{1-T(z)}=\sum_{n=1}^{\infty}Q(n)n^{n-1}\frac{z^n}{n!}\,,
	\]
	where $Q_n$ denotes Ramanujan's $Q$-function defined by
	\[
	Q(n)=1+\frac{n-1}{n}+\frac{(n-1)(n-2)}{n^2}+\frac{(n-1)(n-2)(n-3)}{n^3}+\cdots\,.
	\]
	Moreover, $Q(n)$ admits the following asymptotic expansion (Theorem 2 in \cite{FGKP1995}):
	\be\label{eq: aymp of Qn}
	Q(n)\sim \sqrt{\frac{\pi n}{2}}-\frac{1}{3}+\frac{1}{12}\sqrt{\frac{\pi}{2n}}-\frac{4}{135n}+\cdots\,.
	\ee
	Applying \eqref{eq:T3} with $k=2$ yields
	\[
	\big[z^n\big]T^2(z)=\frac{2n^{n-3}}{(n-2)!}\,.
	\]
	We now derive \eqref{eq: cnn}. Using $	W_0(z)=\frac{1}{2}\log\frac{1}{1-T(z)}-\frac{1}{2}T(z)-\frac{1}{4}T^2(z)$ we have
	\begin{align*}
	C_{n,n}&=n!\big[z^n\big]W_0(z)=\frac{1}{2}n!\big[z^n\big]\log\frac{1}{1-T(z)}-\frac{1}{2}n!\big[z^n\big]T(z)-\frac{1}{4}n!\big[z^n\big]T^2(z)\\
	&=\frac{1}{2}Q(n)n^{n-1}-\frac{1}{2}n^{n-1}-\frac{1}{4}n!\cdot \frac{2n^{n-3}}{(n-2)!}\\
	&\stackrel{\eqref{eq: aymp of Qn}}{=}\sqrt{\frac{\pi}{8}}n^{n-\frac{1}{2}}-\frac{7}{6}n^{n-1}+\frac{1}{24}\sqrt{\frac{\pi}{2}}\,n^{n-\frac{3}{2}}+O(n^{n-2})\, . \qedhere
	\end{align*}
	
\end{proof}

\subsubsection{The  asymptotics of $C_{n,n}^{e,f}$}
\label{subsubsec: 4.3.2}

Suppose $e$ and $f$ are a pair of adjacent edges of $K_n$. In this subsection we establish the following asymptotic formula for $C_{n,n}^{e,f}\coloneq \card{\{\omega\in \sC^{(0)}(K_n)\colon \omega(e)=\omega(f)=1 \}}$.
\begin{proposition}\label{prop: asymp of cnnef}
	For a pair of adjacent edges $e,f$ of $K_n$,  we have 
	\be\label{eq: asymp of cnnef}
	C_{n,n}^{e,f}
	=\frac{3\sqrt{2\pi}}{4}n^{n-\frac{5}{2}}-\frac{7}{2}n^{n-3}+
	\frac{37\sqrt{2\pi}}{16}n^{n-\frac{7}{2}}+O(n^{n-4})\,.
	\ee
\end{proposition}

We start with a straightforward extension of \cite[Lemma 2.2]{Stark2011}. 
\begin{lemma}\label{lem:generalisation of stark lem2.2}
	Let $p,q$ be two  integers satisfying $1\leq q\leq p$,  and let $\{m_i\}_{i=1}^{p}$  be a sequence of integers with $m_i>0$ for $1\leq i \leq q$ and $m_i=0$ for $q+1\leq i\leq p$, such that $d\coloneq \sum_{i=1}^{q}m_i\le n$. Suppose $S_i \subseteq V(K_n)$ are mutually disjoint subsets with $\card{S_i}=m_i$ for all $1\le i \le q$. Let $k_1,\ldots,k_p$ be $p$ integers at least $-1$. Then the number of subgraphs of $K_n$ with $p$ components that satisfy the following three conditions:
	\begin{itemize}
		\item for each $i\in[1,q]$, each vertex in $S_i$ is contained in the same component;
		\item for $i \neq j$, the vertices in $S_i$ and $S_j$ are contained in distinct components;
		\item for each $i\in[1,p]$, the component containing $S_i$ has excess $k_i$;
	\end{itemize}
	is
	\be\label{eq:lemma2.2 1}
	\frac{(n-d)!}{(p-q)!}\big[z^{n-d}\big]\prod_{q+1\le i\le p}W_{k_i}(z)\prod_{1\le j\le q}W_{k_j}^{(m_j)}(z),
	\ee
	where $W_{k_j}^{(m_j)}(z)$ is the $m_j$-th derivative of $W_{k_j}(z)$.
\end{lemma}
\begin{proof}
	The number of subgraphs of $K_n$ that satisfy the conditions of the lemma is equal to 
	\be    \label{eq:lemma2.2 2}
	\frac{1}{(p-q)!}\sum_{\substack{a_1+\cdots +a_p=n,\\a_i\ge m_i}}\binom{n-d}{a_1-m_1,\cdots,a_p-m_p}\prod_{1\le i \le p} C_{a_i,a_i+k_i}.
	\ee
	It is straightforward to  verify that \eqref{eq:lemma2.2 1} equals \eqref{eq:lemma2.2 2}.
\end{proof}

\begin{lemma}\label{lem: decomposition into cases}
	Suppose $e=(u,v)$ and $f=(u,w)$ are a pair of adjacent edges in $K_n$ with  common vertex $u$. Let $I_1$ be the number of unicyclic subgraphs of $K_n$ such that $e$ and $f$ both lie on the cycle. Let $I_2$ be the number of unicyclic subgraphs of $K_n$ such that $e$ lies on the cycle but $f$ does not. Let $I_3$ be the number of unicyclic subgraphs of $K_n$ such that neither $e$ nor $f$ lies on the cycle and $u$ lies in the component containing the cycle after deleting $e$ and $f$ from the unicyclic subgraph. Then we have 
	\[
	C_{n,n}^{e,f}
	=I_1+2I_2+3I_3\,.
	\]
\end{lemma}
\begin{proof}
	When counting  the number of unicyclic subgraphs that contain both $e$ and $f$,
	there are three cases to consider (see Fig.~\ref{fig:adjacentcase} for a systematic illustration):
	\begin{enumerate}
		\item[(1)] Both $e$ and $f$ lie on the unique cycle;
		\item[(2)] Exactly one of $e$ and $f$ lies on the unique cycle; 
		\item[(3)] Neither $e$ nor $f$ lies on the unique cycle. 
	\end{enumerate}
	\begin{figure}[H]
		\begin{minipage}{0.45\textwidth}
			\centering
			\includegraphics[scale=0.80]{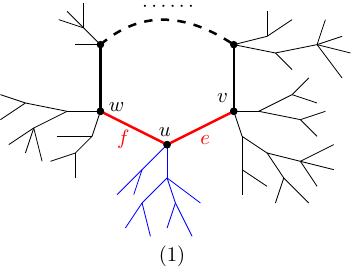}
		\end{minipage}%
		\hspace{1.7cm}
		\begin{minipage}{0.45\textwidth}
			\centering
			\includegraphics[scale=0.80]{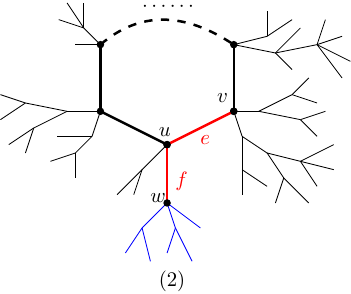}
			
		\end{minipage}
		
		\vskip 10mm 
		
		\centering
		\includegraphics[scale=0.8]{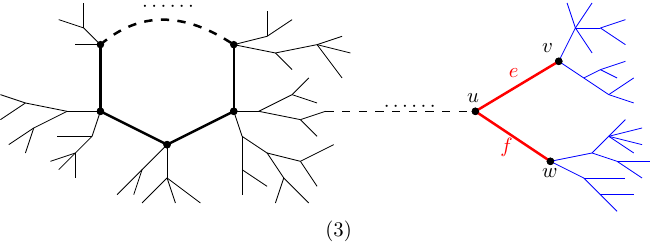}
		
		\caption{Three Cases}
		\label{fig:adjacentcase}

	\end{figure}

	The contribution of case (1) is just $I_1$. The contribution of case (2) arises from two subcases: (i) $e$ is on the cycle but $f$ is not; (ii) $f$ is on the cycle but $e$ is not. By symmetry the contribution of case (2) is $2I_2$. For case (3), consider the subgraphs obtained after deleting both $e$ and $f$. This deletion results in three components containing $u,v,w$ respectively, and exactly one of these components contains the cycle. The term $I_3$ counts the subcase where the component of $u$ contains the cycle. By symmetry, the contribution of case (3) is then $3I_3$. 	
\end{proof}

\begin{lemma}\label{lem:I1}
	Suppose $e=(u,v)$ and $f=(u,w)$ are a pair of adjacent edges in $K_n$ with  common vertex $u$. Let $I_1$ be the number of unicyclic subgraphs of $K_n$ in which both $e$ and $f$ lie on the cycle. Then
	\[
	I_1=(n-3)!\big[z^{n-3}\big]W_{-1}'(z)W_{-1}''(z)=n^{n-3}\,.
	\]
\end{lemma}
\begin{proof}
	If we delete $e$ and $f$ from a unicyclic subgraph in which both $e$ and $f$ lie on the cycle, we obtain two trees: one containing $v,w$ and the other containing $u$. Thus $I_1$ equals the number of $2$-forests of $K_n$ such that $v,w$ lie in one tree and $u$ lies in the other. Applying Lemma~\ref{lem:generalisation of stark lem2.2} with $p=q=2$, $S_1=\{u\},S_2=\{v,w\}$, and $k_1=k_2=-1$, we obtain
	\[
	I_1=(n-3)!\big[z^{n-3}\big]W_{-1}'(z)W_{-1}''(z)\,.
	\]
	
	Since $W_{-1}(z)=T(z)-\frac{1}{2}T^2(z)$, 
	it follows that
	\begin{equation}\label{eq: W-1'}
	W_{-1}'(z)=T'(z)-T(z)T'(z)\stackrel{\eqref{eq: T'}}{=}\frac{T(z)}{z}\,,
	\end{equation}
	and
	\begin{equation}\label{eq: W_1''}
	W_{-1}''(z)=\frac{zT'(z)-T(z)}{z^2}\stackrel{\eqref{eq: T'}}{=}\frac{T^2(z)}{z^2\big[1-T(z)\big]}\,.
	\end{equation}
	Hence $W_{-1}'(z)W_{-1}''(z)=\frac{T^3(z)}{z^3\big[1-T(z)\big]}$, and therefore
	\[
	\big[z^{n-3}\big]W_{-1}'(z)W_{-1}''(z)=\big[z^n\big]\frac{T^3(z)}{1-T(z)}\,.
	\]
	Applying \eqref{eq:T4} with  $k=3$, we have for $n\ge3$,
	\[
	\big[z^{n-3}\big]W_{-1}'(z)W_{-1}''(z)=\big[z^n\big]\frac{T^3(z)}{1-T(z)}=\frac{n^{n-3}}{(n-3)!}\,.
	\]
	Therefore 
	\[
	I_1=(n-3)!\big[z^{n-3}\big]W_{-1}'(z)W_{-1}''(z)=n^{n-3}\,. \qedhere
	\]
\end{proof}

\begin{lemma}\label{lem:I2}
	Suppose $e=(u,v)$ and $f=(u,w)$ are a pair of adjacent edges in $K_n$ with  common vertex $u$. Let $I_2$ be the number of unicyclic subgraphs of $K_n$ such that $e$ lies on the cycle but $f$ does not. 
	Then 
	\[
	I_2=(n-3)!\big[z^{n-3}\big]W_{-1}'(z)W_{-1}''(z)-(n-3)!\big[z^{n-3}\big]\big(W_{-1}'(z)\big)^3=n^{n-3}-3n^{n-4}\,.
	\]
\end{lemma}
\begin{proof}
	If we delete $e$ and $f$ from a unicyclic subgraph of $K_n$ where $e$ lies on the cycle but $f$ does not, we obtain two trees: one containing the endpoints of $e$ (i.e., $u,v$) and the other containing the vertex $w$. Furthermore, the tree containing $u,v$ can no longer use the edge $e$. Let $N$ denote the number of $2$-forests of $K_n$ such that one tree contains $u,v$ and the other contains $w$. Let $\widehat{N}$ denote the number of $2$-forests of $K_n$ such that one tree contains $u,v$ and the edge $e=(u,v)$, and the other tree contains $w$. Then 
	$I_2=N-\widehat{N}$. Applying Lemma~\ref{lem:generalisation of stark lem2.2} with $p=q=2$, $S_1=\{u,v\}$, $S_2=\{w\}$, and $k_1=k_2=-1$, we have 
	\[
	N=(n-3)!\big[z^{n-3}\big]W_{-1}'(z)W_{-1}''(z)\,.
	\]
	From Lemma~\ref{lem:I1}, we have $N=(n-3)!\big[z^{n-3}\big]W_{-1}'(z)W_{-1}''(z)=n^{n-3}$.

	The number $\widehat{N}$ is equivalent to the number of $3$-forests of $K_n$ where each of the three trees contains exactly one of the vertices $u,v,w$. Hence, applying Lemma~\ref{lem:generalisation of stark lem2.2} with $p=q=3$, $S_1=\{u\}$, $S_2=\{v\}$, $S_3=\{w\}$, and $k_1=k_2=k_3=-1$, we have 
	\[
	\widehat{N}=(n-3)!\big[z^{n-3}\big]\big(W_{-1}'(z)\big)^3\,.
	\]
	By \eqref{eq: W-1'}, we have 
	\[
	\big[z^{n-3}\big]\big(W_{-1}'(z)\big)^3=\big[z^n\big]T^3(z)\,.
	\]
	
	Applying \eqref{eq:T3} with $k=3$, we obtain
	\be\label{eq:Nhat}
	\big[z^{n-3}\big]\big(W_{-1}'(z)\big)^3=\big[z^n\big]T^3(z)=\frac{3n^{n-4}}{(n-3)!}\,.
	\ee  
	Hence 
	\[
	\widehat{N}=(n-3)!\big[z^{n-3}\big]\big(W_{-1}'(z)\big)^3=3n^{n-4}\,.
	\]
	Therefore 
	\[
	I_2=N-\widehat{N}=n^{n-3}-3n^{n-4}\,.  \qedhere
	\]
\end{proof}

\begin{lemma}\label{lem:I3}
	Suppose $e=(u,v)$ and $f=(u,w)$ are a pair of adjacent edges in $K_n$ with  common vertex $u$.  Let $I_3$ be the number of unicyclic subgraphs of $K_n$ such that neither $e$ nor $f$ lies on the cycle, and $u$ lies in the component containing the cycle after deleting $e$ and $f$ from the unicyclic subgraph. 
	Then 
	\[
	I_3=(n-3)!\big[z^{n-3}\big]W_0'(z)\big[W_{-1}'(z)\big]^2=\frac{\sqrt{2\pi}}{4}n^{n-\frac{5}{2}}-\frac{13}{6}n^{n-3}+O(n^{n-\frac{7}{2}})\,.
	\]
\end{lemma}
\begin{proof}
	Applying Lemma~\ref{lem:generalisation of stark lem2.2} with $p=q=3,S_1=\{u\},S_2=\{v\},S_3=\{w\}$, and $k_1=0,k_2=k_3=-1$, we have that 
	\[
	I_3=(n-3)!\big[z^{n-3}\big]W_0'(z)\big[W_{-1}'(z)\big]^2\,.
	\]
	Since $W_0(z)=\frac{1}{2}\log\frac{1}{1-T(z)}-\frac{1}{2}T(z)-\frac{1}{4}T^2(z)$, we obtain 
	\begin{align}\label{eq: W0prime}
	W_0'(z)&=\frac{1}{2}\cdot \frac{T'(z)}{1-T(z)}-\frac{1}{2}T'(z)-\frac{1}{2}T'(z)T(z) \nonumber\\
	&\stackrel{\eqref{eq: T'}}{=} \frac{1}{2}\cdot \frac{T^3(z)}{z\big[1-T(z)\big]^2}
	\end{align}
	Combining this with \eqref{eq: W-1'} we  have
	\[
	W_0'(z)\big[W_{-1}'(z)\big]^2=\frac{1}{2}\cdot \frac{T^5(z)}{z^3\big[1-T(z)\big]^2}\,,
	\]
	and consequently
	\[
	\big[z^{n-3}\big]W_0'(z)\big[W_{-1}'(z)\big]^2=\frac{1}{2}\cdot \big[z^n\big] \frac{T^5(z)}{\big[1-T(z)\big]^2}\,.
	\]
	Applying \eqref{eq:T_alpha_beta} with $\alpha=5$ and $\beta=2$, we have 
	\[
	\big[z^n\big] \frac{T^5(z)}{\big[1-T(z)\big]^2}
	=e^n\bigg[\frac{1}{2}-\frac{13}{6}\sqrt{\frac{2}{\pi}}n^{-\frac{1}{2}}+0+O(n^{-\frac{3}{2}})\bigg]\,.
	\]
	By Stirling's formula $n!=\frac{n^n\sqrt{2\pi n}}{e^n}\big(1+\frac{1}{12n}+O(n^{-2})\big)$, we have
	\be\label{eq: n minus 3 factorial}
	(n-3)!=\frac{n!}{n(n-1)(n-2)}=\frac{n^{n-3}\sqrt{2\pi n}}{e^n}\big[1+\frac{37}{12}n^{-1}+O(n^{-2})\big]\,.
	\ee
	We thus obtain that 
	\begin{align*}
	I_3&=(n-3)!\big[z^{n-3}\big]W_0'(z)\big[W_{-1}'(z)\big]^2\\
	&=\frac{n^{n-3}\sqrt{2\pi n}}{e^n}\big[1+\frac{37}{12}n^{-1}+O(n^{-2})\big]\cdot \frac{1}{2}\cdot e^n\bigg[\frac{1}{2}-\frac{13}{6}\sqrt{\frac{2}{\pi}}n^{-\frac{1}{2}}+O(n^{-\frac{3}{2}})\bigg]\\
	&=\frac{\sqrt{2\pi}}{4}n^{n-\frac{5}{2}}-\frac{13}{6}n^{n-3}+\frac{37\sqrt{2\pi}}{48}n^{n-\frac{7}{2}}+O(n^{n-4})\,.  \qedhere
	\end{align*}
\end{proof}

\begin{proof}[Proof of Proposition~\ref{prop: asymp of cnnef}]
	By Lemma~\ref{lem: decomposition into cases}, 
	\[
	C_{n,n}^{e,f}
	=I_1+2I_2+3I_3\,.
	\]
	Substituting the estimates of $I_1,I_2,I_3$ from Lemmas~\ref{lem:I1}, \ref{lem:I2} and \ref{lem:I3}, we  obtain 
	\[
	C_{n,n}^{e,f}
	=\frac{3\sqrt{2\pi}}{4}n^{n-\frac{5}{2}}-\frac{7}{2}n^{n-3}+
	\frac{37\sqrt{2\pi}}{16}n^{n-\frac{7}{2}}+O(n^{n-4})\,. \qedhere
	\]
\end{proof}

\subsubsection{Proof of Proposition~\ref{prop: key prop for k-excess} for $k=0$ }

\begin{proof}[Proof of the $k=0$ case in Proposition~\ref{prop: key prop for k-excess}]
	The $k=0$ case of \eqref{eq: cnnk asym} and \eqref{eq: cnnkef asym} have already been proved in Propositions~\ref{prop: asymptotic of cnn} and \ref{prop: asymp of cnnef} respectively. 
	As for the $k=0$ case of \eqref{eq: p1 kUC asym}, we use these two estimates to derive it:
	\begin{align*}
	p_1&=\frac{ \frac{3\sqrt{2\pi}}{4}n^{n-\frac{5}{2}}-\frac{7}{2}n^{n-3}+
		\frac{37\sqrt{2\pi}}{16}n^{n-\frac{7}{2}}+O(n^{n-4}) }{\sqrt{\frac{\pi}{8}}n^{n-\frac{1}{2}}-\frac{7}{6}n^{n-1}+\frac{1}{24}\sqrt{\frac{\pi}{2}}\,n^{n-\frac{3}{2}}+O(n^{n-2})} \\
	&=\frac{3}{n^2}\cdot \frac{1-\frac{14}{3\sqrt{2\pi}}n^{-\frac{1}{2}} +\frac{37}{12}n^{-1}+O(n^{-\frac{3}{2}}) }{ 1-\frac{14}{3\sqrt{2\pi}}n^{-\frac{1}{2}} + \frac{1}{12}n^{-1} +O(n^{-\frac{3}{2}})  } \\
	&=\frac{3}{n^2}\cdot \big[1+\frac{3}{n}+O(n^{-\frac{3}{2}})\big]\\
	&=\frac{3}{n^2}+\frac{9}{n^3}+O(n^{-\frac{7}{2}}) \,. \qquad \qedhere
	\end{align*}
\end{proof}

\subsection{The cases $1 \le k \le 5$ of Proposition~\ref{prop: key prop for k-excess}}\label{sec: k ge 1 cases}

\subsubsection{The asymptotics of $C_{n,n+k}$ for $k\in\{1,2,\ldots,5\}$}

We first collect a few facts about the asymptotics of $C_{n,n+k}\coloneq \card{\sC^{(k)}(K_n)}$.  
\begin{proposition}\label{prop: asymptotic of cnnk}
	For $k\ge-1$, we have
	\be\label{eq: cnnk from Wright1977}
	C_{n,n+k}=\rho_k n^{n+\frac{3k-1}{2}}\big[ 1+O(n^{-\frac{1}{2}})  \big]\,,
	\ee
	where $\rho_{-1}=1$, and for $k\ge0$, 
	\[
	\rho_k=\frac{\sqrt{\pi}\cdot 2^{(1-3k)/2}\sigma_k }{\Gamma(1+\frac{3k}{2})}\,,
	\]
	with the sequence $\{\sigma_i\}$ given by
	\[
	4\sigma_0=1, \quad 16\sigma_1=5,\quad 16\sigma_2=15,
	\]
	and 
	\be\label{eq:recursive-sigma}
	\sigma_{k+1}=\frac{3(k+1)\sigma_k}{2}+\sum_{s=1}^{k-1}\sigma_s\sigma_{k-s}\quad (k\ge2)\,.
	\ee
	In particular, for all $k\ge-1$, we have 
	\be\label{eq: cnnk}
	C_{n,n+k}\asymp n^{n+\frac{3k-1}{2}}\,.
	\ee
\end{proposition}
\begin{proof}
	The case $k=-1$ follows directly from Cayley's formula $C_{n,n-1}=n^{n-2}$. The case of $k\ge0$ in \eqref{eq: cnnk from Wright1977} appeared in Wright \cite[page~329]{Wright1977one}, and we  provide the details in the supplementary material for readers' convenience. Furthermore, the case  $k=0$ in \eqref{eq: cnnk from Wright1977} is a weaker form of \eqref{eq: cnn}. 
\end{proof}

\begin{proposition}\label{prop: asymptotic of cnnk for small k}
	For $k\in\{1,2,\ldots,5\}$, 
	we have the following asymptotic behavior of $C_{n,n+k}$:
	\be\label{eq: cnn1}
	C_{n,n+1}=  \frac{5}{24}n^{n+1} -\frac{7\sqrt{2\pi}}{24}n^{n+\frac{1}{2}} +\frac{25}{36}n^n+O\left(n^{n-\frac{1}{2}}\right)   \,,\nonumber
	\ee
	and
	\be\label{eq: cnn2}
	C_{n,n+2}=\frac{5\sqrt{2\pi}}{256}n^{n + \frac{5}{2}} - \frac{35}{144}n^{n +2}+\frac{1559\sqrt{2\pi}}{9216}n^{n + \frac{3}{2}}+O(n^{n+1})    \,,\nonumber
	\ee
	and 
	\be\label{eq: cnn3}
	C_{n,n+3}=\frac{221}{24192}n^{n + 4} - \frac{35\sqrt{2\pi}}{1536}n^{n + \frac{7}{2}}+\frac{25283}{181440}n^{n + 3}+O(n^{n+\frac{5}{2}})  \,,\nonumber
	\ee
	and 
	\be\label{eq: cnn4}
	C_{n,n+4}=\frac{113\sqrt{2\pi}}{196608}n^{n + \frac{11}{2}} - \frac{221}{20736}n^{n + 5}+\frac{30569\sqrt{2\pi}}{2359296}n^{n + \frac{9}{2}}+O(n^{n+4}) \,,\nonumber
	\ee
	and 
	\be\label{eq: cnn5}
	C_{n,n+5}=\frac{16565}{83026944}n^{n + 7} - \frac{791\sqrt{2\pi}}{1179648}n^{n + \frac{13}{2}}+\frac{374713}{62270208}n^{n + 6}+O(n^{n+\frac{11}{2}}) \,.\nonumber
	\ee
\end{proposition}

The idea of proving Proposition~\ref{prop: asymptotic of cnnk for small k} is first to derive an expression for $W_k(z)$ in terms of $T(z)$,  then apply the procedure outlined after Lemma~\ref{lem: facts of T(z)}. Specifically,  let $\theta=\theta(z)\coloneq 1-T(z)$\footnote{In \cite{Wright1977one}, the notation $G=G(z)$ corresponds to $T(z)$  in the present paper.};  \cite[Theorem 4]{Wright1977one} shows that  for $k\ge1$, there exist constants $\{c_{k,s} \colon  s=-3k,-3k+1,\ldots,2\}$ such that 
\be\label{eq: thm4 in Wrightone}
W_k=\sum_{s=-3k}^{2}c_{k,s}\theta^s\,,
\ee
Wright \cite{Wright1977one} provided a recursive method for determining the explicit expressions  of $W_k$ in terms of $\theta$, and  the result we need are summarized in Lemma~\ref{lem: W_k in terms of theta} below. 

We frequently work with formal power series of  the form on the right-hand side of \eqref{eq: thm4 in Wrightone}. For convenience, we summarize how to extract coefficient information from such series using the aforementioned procedure in Lemma~\ref{lem: coef using theta expansion}, whose proof is deferred to  the supplementary material. 
\begin{lemma}\label{lem: coef using theta expansion}
	Let $\theta=\theta(z)=1-T(z)$, and let $F(z)=\sum_{s=-m}^{t}f_{s}\theta^{s}$ where $m,t$ are  positive integers and  $\{f_s\colon s=-m,\ldots,t\}$ are constants with $f_{-m}>0$. Then   
	\be\label{eq: coef using theta expansion}
	\big[z^n\big]F(z)=	f_{-m}\frac{2^{-\frac{m}{2}}}{\Gamma(\frac{m}{2})}e^nn^{\frac{m}{2}-1}
	+g_{m,1}e^nn^{\frac{m}{2}-\frac{3}{2}}+g_{m,2}e^nn^{\frac{m}{2}-2}	+O\left(e^nn^{\frac{m}{2}-\frac{5}{2}}\right)\,,	
	\ee
	where \[
	g_{m,1}=\frac{2^{-\frac{m}{2}}\cdot \frac{\sqrt{2}m}{3}}{\Gamma\left(\frac{m-1}{2}\right)}f_{-m} +\frac{2^{-\frac{m-1}{2}}}{\Gamma(\frac{m-1}{2})}f_{-m+1}
	\]
	and
	\[
	g_{m,2}=\frac{2^{-\frac{m}{2}}(2m^2+m)}{18\Gamma\left(\frac{m-2}{2}\right)}f_{-m} +\frac{2^{-\frac{m-1}{2}}\cdot \frac{\sqrt{2}(m-1)}{3}}{\Gamma\left(\frac{m-2}{2}\right)}f_{-m+1}+\frac{2^{-\frac{m-2}{2}}}{\Gamma(\frac{m-2}{2})}f_{-m+2}\,.
	\] 
	We adopt the convention that $\Gamma(0)=\infty$ and $\Gamma(-\frac{1}{2})=-2\sqrt{\pi}$.

\end{lemma}

\begin{lemma}\label{lem: W_k in terms of theta}
	Let $\theta=\theta(z)= 1-T(z)$. For $k\in\{-1,0,1,2,\ldots,5\}$, $W_k$ has the following expressions:
	\be\label{eq: W_minus1 in terms of theta}
	W_{-1}=\frac{1}{2}-\frac{1}{2}\theta^2\,,
	\ee
	and 
	\be\label{eq: W_0 in terms of theta}
	W_0
	=-\frac{1}{2}\log \theta -\frac{3}{4}+\theta-\frac{1}{4}\theta^2\,,
	\ee
	and 
	\be\label{eq: W_1 in terms of theta}
	W_1=\frac{1}{24}\big[5\theta^{-3}-19\theta^{-2}+26\theta^{-1}-14+\theta+\theta^2\big]\,,
	\ee
	and
	\begin{align}\label{eq: W_2 in terms of theta}
	W_2&=\frac{1}{48}\big[15\theta^{-6}-65\theta^{-5}+108\theta^{-4}-87\theta^{-3}+42\theta^{-2}-23\theta^{-1}+12-\theta-\theta^2 \big]
	\end{align}
	and 
	\begin{align}\label{eq: W_3 in terms of theta}
	W_3&=\frac{1105}{1152 \theta^{9}} - \frac{1945}{384 \theta^{8}} + \frac{6353}{576 \theta^{7}} - \frac{233}{18 \theta^{6}} + \frac{8929}{960 \theta^{5}} - \frac{1521}{320 \theta^{4}} + \frac{719}{360 \theta^{3}} - \frac{67}{96 \theta^{2}}+ \frac{115}{384 \theta} \nonumber
	\\& - \frac{293}{1920} + \frac{13 \theta}{960} + \frac{19 \theta^{2}}{1440}
	\end{align}
	and
	\begin{align}\label{eq: W_4 in terms of theta}
	W_4&=\frac{565}{128 \theta^{12}} - \frac{21295}{768 \theta^{11}} + \frac{172337}{2304 \theta^{10}} - \frac{263105}{2304 \theta^{9}} + \frac{255437}{2304 \theta^{8}} - \frac{215153}{2880 \theta^{7}} + \frac{220763}{5760 \theta^{6}} - \frac{92893}{5760 \theta^{5}} \nonumber
	\\&+ \frac{16741}{2880 \theta^{4}} - \frac{21977}{11520 \theta^{3}} + \frac{6577}{11520 \theta^{2}} - \frac{827}{3840 \theta} + \frac{1241}{11520} - \frac{19 \theta}{1920} - \frac{3 \theta^{2}}{320}
	\end{align}
	and
	\begin{align}\label{eq: W_5 in terms of theta}
	W_5&=\frac{82825}{3072 \theta^{15}} - \frac{603965}{3072 \theta^{14}} + \frac{323385}{512 \theta^{13}} - \frac{49027387}{41472 \theta^{12}} + \frac{119862917}{82944 \theta^{11}} - \frac{171249163}{138240 \theta^{10}} + \frac{82217167}{103680 \theta^{9}} \nonumber
	\\ &- \frac{10454209}{25920 \theta^{8}} + \frac{165837761}{967680 \theta^{7}} - \frac{61001153}{967680 \theta^{6}} + \frac{10019617}{483840 \theta^{5}} - \frac{28639}{4608 \theta^{4}} + \frac{735343}{414720 \theta^{3}} - \frac{199879}{414720 \theta^{2}} \nonumber
	\\&+ \frac{361}{2160 \theta} - \frac{59761}{725760} + \frac{5611 \theta}{725760} + \frac{863 \theta^{2}}{120960}\,.
	\end{align}
\end{lemma}
\begin{proof}
	From Theorem~\ref{thm: FSS2004 thm1}, we have
	\[
	W_{-1}(z)=T(z)-\frac{1}{2}T^2(z)=\frac{1}{2}-\frac{1}{2}\theta^2\,,
	\]
	and 
	\[
	W_0(z)=\frac{1}{2}\log\frac{1}{1-T(z)}-\frac{1}{2}T(z)-\frac{1}{4}T^2(z)
	=-\frac{1}{2}\log \theta -\frac{3}{4}+\theta-\frac{1}{4}\theta^2\,.
	\]
	
	The expressions for $W_1(z)$ and $W_2(z)$ appear in \cite[p.~322]{Wright1977one} as $24W_1(z)=\theta^{-3}(1-\theta)^4(5+\theta)$ and $48W_2(z)=\theta^{-6}(1-\theta)^4(15-5\theta-2\theta^2-5\theta^3-\theta^4)$. 
	
	For $k\ge1$, 
	Wright \cite{Wright1977one} provided a recursive procedure to derive the explicit expressions for $W_k$ in terms of $\theta$, outlined as follows:  
	\begin{itemize}
		\item[(1)] Let $\sD:\bbC[[z]]\to\bbC[[z]]$ denote the operator on formal power series defined by
		\[
		\sD f(z)\coloneq \sum_{n=1}^{\infty}(nf_n)z^n\,
		\]
		for $f(z)=\sum_{n=0}^{\infty}f_nz^n\in\bbC[[z]]$, (equivalently, $\sD f(z)=zf'(z)$). Wright \cite[Eq.(7), p.~320]{Wright1977one} showed that the power series $W_k$ satisfy the following recursive relation:
		\[
		2(\sD+k+1)W_{k+1}=(\sD^2-3\sD-2k)W_k+\sum_{h=-1}^{k+1}(\sD W_h)(\sD W_{k-h})\,.
		\]
		
		\item[(2)] 
		Fix $k\ge0$ and assume we have derived the expressions for $W_h$ in terms of $\theta$ for all $h\leq k$.  Define the  formal power series  $J_k=J_k(z)$  by 
		\[
		2J_k\coloneq (\sD^2-3\sD-2k)W_k+\sum_{h=0}^{k}(\sD W_h)(\sD W_{k-h})\,.
		\]
		Wright \cite[Eq.(14), p.~322]{Wright1977one} established that the operator $\sD$ satisfies:
		\[
		\sD f(z)=\left(1-\frac{1}{\theta}\right)\frac{d f}{d\theta}\,, \quad \text{where }\frac{df}{d\theta}=\frac{f'(z)}{\theta'(z)} \text{ for all }f\in\bbC[[z]].
		\]
		In particular, since $W_0(z)=\frac{1}{2}\log\frac{1}{\theta}-\frac{3}{4}+\theta-\frac{1}{4}\theta^2$, we have
		\be\label{eq:DW0}
		\sD W_0=(1-\theta^{-1})\cdot \big[-\frac{1}{2}\theta^{-1}+1-\frac{1}{2}\theta\big]=\frac{(1-\theta)^3}{2\theta^2}\,.
		\ee
		For $h\in\{1,\ldots,k\}$, using the expression of $W_h$ from \eqref{eq: thm4 in Wrightone},   we obtain 
		\begin{align*}
		\sD W_h&=(1-\theta^{-1})\cdot\sum_{s=-3h}^{2}sc_{h,s}\theta^{s-1}\\
		&=3hc_{h,-3h}\theta^{-3h-2}+\sum_{s=-3h-1}^{0} \big[(s+1)c_{h,s+1}-(s+2)c_{h,s+2}\big] \theta^{s}+2c_{h,2}\theta\,.
		\end{align*}
		Applying the identity $\sD f=\left(1-\frac{1}{\theta}\right)\frac{d f}{d\theta}$ to $f=\sD W_k$ gives
		\begin{align*}
		\sD^2 W_k&=(1-\theta^{-1})\cdot  \left[-3k(3k+2)c_{k,-3k}\theta^{-3k-3}\right]\\
		&\quad +(1-\theta^{-1})\cdot\bigg[ \sum_{s=-3k-1}^{0} s\big[(s+1)c_{k,s+1}-(s+2)c_{k,s+2}\big] \theta^{s-1}+2c_{k,2} \bigg]\\
		&=3k(3k+2)c_{k,-3k}\theta^{-3k-4}+\sum_{s=-3k-3}^{0}\widetilde{c}_{k,s}\theta^{s}\,,
		\end{align*}
		where $\widetilde{c}_{k,s}$ are  constants  expressible in terms of $c_{k,s}$. 
		Substituting these results into the definition of $J_k$, we find constants $\widehat{c}_{k,s}$ (computable from $c_{h,s}$ and $\widetilde{c}_{h,s}$) such that 
		\begin{align*}
		J_k(z)&=j_k\theta^{-3k-4}+\sum_{s=-3k-3}^{2}\widehat{c}_{k,s}\theta^{s}\,,
		\end{align*}
		with 
		\[
		j_k=\frac{9k(k+1)}{2}c_{k,-3k}+\sum_{h=1}^{k-1}\frac{9h(k-h)}{2}c_{h,-3h}c_{k-h,-3(k-h)}\,.
		\]
		Define the function $\cJ=\cJ_k$ by:
		\[
		\cJ_k(t)\coloneq j_kt^{-3k-4}+\sum_{s=-3k-3}^{2}\widehat{c}_{k,s}t^{s}\,.
		\]
		\item[(3)] 
		Theorem 2 of Wright \cite{Wright1977one}\footnote{The term $J$ in \cite[Theorem 2]{Wright1977one} corresponds to the function $\cJ$ defined above.} gives the following expression for  $W_{k+1}$ in terms of $\theta$ (valid for $k\ge0$):
		\be\label{eq:thm2 in Wrightone}
		(1-\theta)^{k+1}W_{k+1}
		=\int_{\theta}^{1}(1-t)^k\cJ_k(t)dt\,.
		\ee
		Wright \cite[Theorem 3]{Wright1977one} proved that the Laurent series expansion of the integrand has no $t^{-1}$ term. A careful computation confirms that \eqref{eq: thm4 in Wrightone} holds for $k+1$ with 
		\be\label{eq: recursive cks}
		c_{k+1,-3(k+1)}=\frac{j_k}{3k+3}=\frac{3k}{2}c_{k,-3k}+\sum_{h=1}^{k-1}\frac{3h(k-h)}{2(k+1)}c_{h,-3h}c_{k-h,-3(k-h)}\,,
		\ee		
		and other constants $c_{k+1,s}$	expressible using $\{c_{h,t}\colon h\leq k\}$.
	\end{itemize}
	
	While hand computation of  $W_k$ in terms of $\theta$ by hand is tedious, the above procedure enables straightforward derivation of the expressions in \eqref{eq: W_3 in terms of theta},   \eqref{eq: W_4 in terms of theta} and  \eqref{eq: W_5 in terms of theta} using Maple. \qedhere
	
\end{proof}

\begin{proof}[Proof of Proposition~\ref{prop: asymptotic of cnnk for small k}]
	Combining Lemma~\ref{lem: W_k in terms of theta}, Lemma~\ref{lem: coef using theta expansion}, and  Stirling's formula \[
	n!=\sqrt{2\pi}e^{-n} n^{n+\frac{1}{2}}\Big[1+\frac{1}{12n}+O\big(n^{-2}\big)\Big]
	\]
	yields the desired asymptotic behaviors. We  illustrate the derivation with the $k=1$ case as an example.
	From \eqref{eq: W_1 in terms of theta} and  Lemma~\ref{lem: coef using theta expansion}, we obtain 
	\[
	\big[z^n\big]W_1(z)=\frac{5}{24}\frac{1}{\sqrt{2\pi}}e^nn^{\frac{1}{2}} -\frac{7}{24}e^n+\frac{195}{288}\frac{1}{\sqrt{2\pi}}e^nn^{-\frac{1}{2}}+O\big(e^nn^{-1}\big) \,.
	\]
	It follows that
	\begin{align*}
	&C_{n,n+1}=n!\big[z^n\big]W_1(z)\\
	&=\sqrt{2\pi}e^{-n} n^{n+\frac{1}{2}}\Big[1+\frac{1}{12n}+O\big(n^{-2}\big)\Big]\cdot \left[ \frac{5}{24}\frac{1}{\sqrt{2\pi}}e^nn^{\frac{1}{2}} -\frac{7}{24}e^n+\frac{195}{288}\frac{1}{\sqrt{2\pi}}e^nn^{-\frac{1}{2}}+O\big(e^nn^{-1}\big)  \right] \\
	&=\frac{5}{24}n^{n+1} -\frac{7\sqrt{2\pi}}{24}n^{n+\frac{1}{2}} +\frac{25}{36}n^n+O\left(n^{n-\frac{1}{2}}\right)   \,. \qedhere
	\end{align*}
\end{proof}

\subsubsection{The asymptotics of $C_{n,n+k}^{e,f}$ for $k\in\{1,2,\ldots,5\}$}

Recall that $C_{n,n+k}^{e,f}\coloneq \card{\big\{ \omega\in\sC^{(k)}(K_n) \colon \omega(e)=\omega(f)=1  \big\}}$, where $e,f$ are a pair of adjacent edges in $K_n$.  In this subsection, we aim to show that $C_{n,n+k}^{e,f}\asymp n^{n+\frac{3k-5}{2}}$ and derive the following asymptotic behavior for $C_{n,n+k}^{e,f}$ when  $k\in\{1,2,3,4,5\}$. 

\begin{proposition}\label{prop: asymptotic of cnnk containing e and f}
	Suppose $e,f$ are a pair of adjacent edges in $K_n$. Then for  $k\in\{1,2,\ldots,5\}$, we have the following estimates for $C_{n,n+k}^{e,f}$:
	\be\label{eq: cnn1ef}
	C_{n,n+1}^{e,f}={ \frac{5}{8}n^{n-1}-\frac{7\sqrt{2\pi}}{8}n^{n-\frac{3}{2}}+\frac{35}{6}n^{n-2}+O(n^{n-\frac{5}{2}}) }  \,,\nonumber
	\ee
	and
	\be\label{eq: cnn2ef}
	C_{n,n+2}^{e,f}
	=\frac{15\sqrt{2\pi}}{256}n^{n + \frac{1}{2}} - \frac{35}{48}n^{n}+\frac{3179\sqrt{2\pi}}{3072}n^{n - \frac{1}{2}}+O(n^{n-1})\,,\nonumber
	\ee
	and
	\be\label{eq: cnn3ef}
	C_{n,n+3}^{e,f}=\frac{221}{8064}n^{n + 2} - \frac{35\sqrt{2\pi}}{512}n^{n + \frac{3}{2}}+\frac{45173}{60480}n^{n + 1}+O(n^{n+\frac{1}{2}}) \,,\nonumber
	\ee
	and
	\be\label{eq: cnn4ef}
	C_{n,n+4}^{e,f}= \frac{113\sqrt{2\pi}}{65536}n^{n + \frac{7}{2}} - \frac{221}{6912}n^{n + 3}+\frac{50909\sqrt{2\pi}}{786432}n^{n + \frac{5}{2}}+O(n^{n+2}) \,,\nonumber
	\ee
	and	
	\be\label{eq: cnn5ef}
	C_{n,n+5}^{e,f}= \frac{16565}{27675648}n^{n + 5} - \frac{791\sqrt{2\pi}}{393216}n^{n + \frac{9}{2}}+\frac{1196681}{41513472}n^{n + 4}+O(n^{n+\frac{7}{2}}) \,,\nonumber
	\ee
	
\end{proposition}

We first decompose the set $\big\{ \omega\in\sC^{(k)}(K_n) \colon \omega(e)=\omega(f)=1  \big\}$ into subsets according to the number of connected components after deleting $e,f$ and then deal with each case separately. Suppose $k\ge1$ and $e,f$ are a pair of adjacent edges of $K_n$. Recall that $\ncc{\omega}$ denotes the number of connected components in the graph $\big(V(K_n),\eta(\omega)\big)$ for $\omega\in\{0,1\}^{E(K_n)}$. Let $\Lambda_i$ for $i\in\{1,2,3\}$ be defined by 
\be\label{eq: Ji def}
\Lambda_i=\Lambda_i(n,k)\coloneq  \left\{ \omega\in\sC^{(k)}(K_n)  \colon \omega(e)=\omega(f)=1,\text{ and } \ncc{\omega\setminus\{e,f\}} =i \right\} \,.
\ee
For $ \omega\in\sC^{(k)}(K_n)$ such that $\omega(e)=\omega(f)=1$, deleting $e,f$ from $\big(V(K_n),\eta(\omega)\big)$ yields a new subgraph with at most three connected components. Hence, we obtain 
\be\label{eq: cnnef decom into cases}
C_{n,n+k}^{e,f}=\card{\Lambda_1}+\card{\Lambda_2}+\card{\Lambda_3}\,.
\ee
Given a subset $S\subset E(K_n)$, we define a map $\Phi_S:\{0,1\}^{E(K_n)} \to \{0,1\}^{E(K_n)}$ as follows:
\be\label{eq: def of Phi}
\Phi_S(\omega)(g)=\left\{
\begin{array}{ccc}
	\omega(g), & \text{ if } & g\in E(K_n)\setminus S\,,\\
	&&\\
	0, & \text{ if } & g\in S\,.
\end{array}
\right.
\ee
For  edges $e,f\in E(K_n)$, for simplicity we write $\Phi_{e}$ and $\Phi_{e,f}$ to denote $\Phi_{\{e\}}$ and $\Phi_{\{e,f\}}$, respectively.

\begin{lemma}\label{lem: J3}
	Suppose $k\ge1$ and let $\Lambda_3=\Lambda_3(n,k)$ denote the subset defined in \eqref{eq: Ji def} with $i=3$. Then  
	\[
	\card{\Lambda_3}=(n-3)!\big[z^{n-3}\big]\sum_{ \substack{k_1,k_2,k_3\ge-1\\ k_1+k_2+k_3=k-2} }W_{k_1}'(z)W_{k_2}'(z)W_{k_3}'(z)\,.
	\]
\end{lemma}
Lemma~\ref{lem: J3} is  analogous to Lemma~\ref{lem:I3} and we include its proof for completeness. 
\begin{proof}[Proof of Lemma~\ref{lem: J3}]
	Let $e=(u,v)$ and $f=(u,w)$ be adjacent edges with  common vertex $u$.
	For $\omega\in \Lambda_3$,  set $\omega'\coloneq \Phi_{e,f}(\omega)$ and  consider the  subgraph $\big(V(K_n),\eta(\omega')\big)$. In this subgraph, the three vertices $u,v,w$  lie in   three distinct connected components. 
	Note that restricting the map $\Phi_{e,f}$ on $\Lambda_3$ gives a bijection between $\Lambda_3$ and the set $\big\{ \omega'\in\{0,1\}^{E(K_n)}\colon \ncc{\omega'}=3, \card{\eta(\omega')}=n+k-2,\text{ and }u,v,w \text{ lie in different  connected components of } \big(V(K_n),\eta(\omega')\big) \big\}$. % (the latter condition implies that $\omega'(e)=\omega'(f)=0$). 
	
	Denote the excess  of the three connected components by  $k_1,k_2,k_3$. The subgraph $\big(V(K_n),\eta(\omega')\big)$ has excess $k-2$, so $k_1+k_2+k_3=k-2$. Moreover, since each component is connected, we have $k_1,k_2,k_3\ge-1$. Applying Lemma~\ref{lem:generalisation of stark lem2.2} with $p=q=3,S_1=\{u\},S_2=\{v\},S_3=\{w\}$, and $k_1,k_2,k_3\ge-1$ satisfying $k_1+k_2+k_3=k-2$, we obtain the desired result:
	\[
	\card{\Lambda_3}=(n-3)!\big[z^{n-3}\big]\sum_{ \substack{k_1,k_2,k_3\ge-1\\ k_1+k_2+k_3=k-2} }W_{k_1}'(z)W_{k_2}'(z)W_{k_3}'(z)\,. \quad \qedhere
	\]
\end{proof}

For $k\ge1$, define
\be\label{eq: def of R_3}
R_3=R_3(k)\coloneq (n-3)!\big[z^{n-3}\big]\sum_{ \substack{k_1,k_2,k_3\ge-1\\ k_1+k_2+k_3=k-2} }W_{k_1}'(z)W_{k_2}'(z)W_{k_3}'(z)\,.
\ee
We state the following asymptotic estimate for $R_3$; its proof is deferred to the supplementary material. 
\begin{lemma}\label{lem: R_3}
	Suppose $k\ge1$. Then 
	\be\label{eq:R_3}
	R_3=3\rho_kn^{n+\frac{3k-5}{2}}+O\left(  n^{n+\frac{3k-6}{2}} \right)\,.
	\ee
\end{lemma}

\begin{lemma}\label{lem: J1}
	Suppose $k\ge1$ and let $\Lambda_1=\Lambda_1(n,k)$ denote the subset defined in \eqref{eq: Ji def} with $i=1$. Then 
	\[
	\card{\Lambda_1}=\bigg(1-\frac{4(n+k-2)}{n(n-1)}\bigg)C_{n,n+k-2}+C_{n,n+k-2}^{e,f}\,.
	\]
\end{lemma}
Unlike the $k=0$ case (where deleting $e,f$ results in at least two connected components), 
Lemma~\ref{lem: J1} has no analogue for $k=0$. 
\begin{proof}[Proof of Lemma~\ref{lem: J1}]
	Restricting the map $\Phi_{e,f}$ on $\Lambda_1$ yields a bijection between $\Lambda_1$ and the set $\big\{\omega'\in\sC^{(k-2)}(K_n)  \colon \omega'(e)=\omega'(f)=0 \big\}$. Thus,
	\begin{align*}
	\card{\Lambda_1}&=\card{  \big\{\omega'\in\sC^{(k-2)}(K_n)  \colon \omega'(e)=\omega'(f)=0 \big\} }\\
	&=\card{  \sC^{(k-2)}(K_n)}-\card{  \big\{\omega'\in\sC^{(k-2)}(K_n)  \colon \omega'(e)=1 \big\} }-\card{  \big\{\omega'\in\sC^{(k-2)}(K_n)  \colon \omega'(f)=1 \big\} }\\
	&\quad +\card{  \big\{\omega'\in\sC^{(k-2)}(K_n)  \colon \omega'(e)=\omega'(f)=1 \big\} }\\
	&=\card{  \sC^{(k-2)}(K_n)}-\card{  \sC^{(k-2)}(K_n)}\cdot \bbP_{(k-2)\mathrm{UC}}\big[\omega(e)=1\big]-\card{  \sC^{(k-2)}(K_n)}\cdot \bbP_{(k-2)\mathrm{UC}}\big[\omega(f)=1\big]\\
	&\quad+C_{n,n+k-2}^{e,f}\,,
	\end{align*}
	where $\bbP_{(k-2)\mathrm{UC}}$ denotes to the uniform probability measure on  $\sC^{(k-2)}(K_n)$. 
	
	Each element in $\sC^{(k-2)}(K_n)$ contains $n+k-2$ open edges. By symmetry, we have $\bbP_{(k-2)\mathrm{UC}}\big[\omega(e)=1\big]=\frac{n+k-2}{\binom{n}{2}}=\frac{2(n+k-2)}{n(n-1)}$. Plugging this into the above expression gives the desired result. 
\end{proof}                                                                                                                                            
\begin{lemma}\label{lem: J2}
	Suppose $k\ge1$ and let $\Lambda_2=\Lambda_2(n,k)$ denote the subset defined in \eqref{eq: Ji def} with $i=2$. Then 
	\[
	\card{\Lambda_2}=3(n-3)!\big[z^{n-3}\big]\sum_{l=-1}^{k-1}W_{l}'(z)W_{k-2-l}''(z)
	-2\card{\widehat{\Lambda}_2}\,,
	\] 
	where 
	\[
	\widehat{\Lambda}_2=\widehat{\Lambda}_2(n,k)\coloneq \left\{\omega\in\sC^{(k-1)}(K_n) \colon \omega(e)=\omega(f)=1,\text{ and $f$ is not on any cycle }  \right\} \,.
	\]
\end{lemma}

Whereas the $k=0$ case relies on Lemmas~\ref{lem:I1} and \ref{lem:I2}, Lemma~\ref{lem: J2} provides a parallel result for $k\ge1$.  In its proof below, $\card{Q_1}$ corresponds to $I_1$ in  Lemma~\ref{lem:I1}, while the symmetric cases $\card{Q_2}$ and $\card{Q_3}$ correspond $I_2$ in Lemma~\ref{lem:I2}.

\begin{proof}[Proof of Lemma~\ref{lem: J2}]
	Restricting the map $\Phi_{e,f}$ to $\Lambda_2$ yields a bijection between $\Lambda_2$ and the set   $Q\coloneq \left\{\omega'\in\{0,1\}^{E(K_n)}\colon \ncc{\omega'}=2,\omega'(e)=\omega'(f)=0, \text{ and } \omega'\cup\{e,f\}\in\sC^{(k)}(K_n)  \right\}$.
	
	Let $e=(u,v)$ and $f=(u,w)$ be adjacent edges with common vertex $u$. Since  $\ncc{\omega'\cup\{e,f\}}=1$ and $\ncc{\omega'}=2$, the three vertices $u,v,w$ must lie in exactly two distinct connected components of the subgraph $\big(V(K_n),\eta(\omega')\big)$. The set $Q$ thus decomposes into three disjoint subsets:
	\begin{enumerate}
		\item[(1)] $Q_1\coloneq \left\{ \omega'\in \{0,1\}^{E(K_n)}  \colon  \ncc{\omega'} =2,\, \card{\eta(\omega')}=n+k-2,\text{ and $u\not\leftrightarrow v$, $v\leftrightarrow w$}  \right\}$, where $v\leftrightarrow w$ means $v$ and $w$ lie in the same connected component, and 
		$u\not\leftrightarrow v$ means that $u$ and $v$ lie in different connected components of $\big(V(K_n),\eta(\omega')\big)$. The  conditions $\ncc{\omega'}=2$, $u\not\leftrightarrow v$ and $v\leftrightarrow w$  imply that $\omega'(e)=\omega'(f)=0$ and $\ncc{\omega'\cup\{e,f\}}=1$. 
		
		\item[(2)] $Q_2\coloneq \left\{ \omega'\in\{0,1\}^{E(K_n)}   \colon \ncc{\omega'} =2,\,\card{\eta(\omega')}=n+k-2, \omega'(e)=0,\text{ and $u\leftrightarrow v$, $u\not\leftrightarrow w$}  \right\}$. Note that for any $\omega'\in Q_2$, we have  $\omega'(f)=0$ and $\ncc{\omega'\cup\{e,f\}}=1$. 
		
		\item[(3)]  $Q_3\coloneq \left\{ \omega'\in\{0,1\}^{E(K_n)}   \colon \ncc{\omega'} =2,\,\card{\eta(\omega')}=n+k-2,\,\omega'(f)=0,\text{ and $u\leftrightarrow w$, $u\not\leftrightarrow v$}  \right\}$. Note that for any  $\omega'\in Q_3$, we have  $\omega'(e)=0$ and $\ncc{\omega'\cup\{e,f\}}=1$.

	\end{enumerate}
	By symmetry, $\card{Q_2}=\card{Q_3}$. Hence 
	\[
	\card{\Lambda_2}=\card{Q}=\card{Q_1}+\card{Q_2}+\card{Q_3}=\card{Q_1}+2\card{Q_2}\,.
	\]
	
	For $\omega'\in Q_1$, let $l$ denote the excess  of  the connected component containing $u$; the remaining component then has excess  $k-2-l$. Since $k-2-l\ge-1$, the range of $l$ is $\{-1,0,\ldots, k-1\}$. 
	Applying Lemma~\ref{lem:generalisation of stark lem2.2} with $p=q=2$, $S_1=\{u\},S_2=\{v,w\}$,  $k_1=l$ and $k_2=k-2-l$, we obtain
	\[
	\card{Q_1}=\sum_{l=-1}^{k-1}(n-3)!\big[z^{n-3}\big]W_l'(z)W_{k-2-l}''(z)=(n-3)!\big[z^{n-3}\big]\sum_{l=-1}^{k-1}W_l'(z)W_{k-2-l}''(z)\,.
	\]

	To compute $\card{Q_2}$, we first note that 
	\begin{align*}
	\card{Q_2}&=\card{\left\{ \omega'\in\{0,1\}^{E(K_n)}   \colon \ncc{\omega'} =2,\,\card{\eta(\omega')}=n+k-2, \text{ and $u\leftrightarrow v$, $u\not\leftrightarrow w$}  \right\}}\\
	&\quad -\card{\left\{ \omega'\in\{0,1\}^{E(K_n)}   \colon \ncc{\omega'} =2,\,\card{\eta(\omega')}=n+k-2, \omega'(e)=1,\text{ and $u\leftrightarrow v$, $u\not\leftrightarrow w$}  \right\}}
	\end{align*}
	An argument analogous to the calculation of $\card{Q_1}$ yields  
	\begin{multline*}
	\card{\left\{ \omega'\in\{0,1\}^{E(K_n)}   \colon \ncc{\omega'} =2,\,\card{\eta(\omega')}=n+k-2, \text{ and $u\leftrightarrow v$, $u\not\leftrightarrow w$}  \right\}}\\
	=(n-3)!\big[z^{n-3}\big]\sum_{l=-1}^{k-1}W_l'(z)W_{k-2-l}''(z)\,.
	\end{multline*}
	Recall that $\widehat{\Lambda}_2\coloneq \left\{ \omega\in\sC^{(k-1)}(K_n) \colon \omega(e)=\omega(f)=1,\text{ and $f$ is not on any cycle } \right\}$. Restricting the map $\Phi_f$ to $\widehat{\Lambda}_2$ gives a bijection between $\widehat{\Lambda}_2$ and the set 
	\[
	\left\{ \omega'\in\{0,1\}^{E(K_n)}   \colon \ncc{\omega'} =2,\,\card{\eta(\omega')}=n+k-2, \omega'(e)=1,\text{ and $u\leftrightarrow v$, $u\not\leftrightarrow w$}  \right\}.
	\]
	Thus, 
	\[
	\card{Q_2}=(n-3)!\big[z^{n-3}\big]\sum_{l=-1}^{k-1}W_l'(z)W_{k-2-l}''(z)
	-\card{\widehat{\Lambda}_2}\,.
	\]
	Combining these results, we find
	\begin{align*}
	\card{\Lambda_2}&=\card{Q_1}+2\card{Q_2}\\
	&=3(n-3)!\big[z^{n-3}\big]\sum_{l=-1}^{k-1}W_l'(z)W_{k-2-l}''(z)-2\card{\widehat{\Lambda}_2}\,.\qedhere
	\end{align*}
\end{proof}                                                        

For $k\ge1$, define
\be\label{eq: def of R_2}
R_2=R_2(k)\coloneq 3(n-3)!\big[z^{n-3}\big]\sum_{l=-1}^{k-1}W_l'(z)W_{k-2-l}''(z)\,.
\ee                      
We state the following asymptotic estimate for $R_2$; its proof is deferred to the supplementary material. 
\begin{lemma}\label{lem: R_2}
	Suppose $k\ge1$. Then
	\be\label{eq:R_2}
	R_2=3\rho_{k-1}n^{n+\frac{3k-6}{2}}\big[1+O(n^{-\frac{1}{2}})\big]\,.
	\ee
\end{lemma}                                                                                                                                            
Now we are ready to determine the asymptotic order of $C_{n,n+k}^{e,f}$. 
\begin{lemma}\label{lem: negligible terms}
	Recall that $C_{n,n+k}^{e,f}\coloneq \card{\big\{ \omega\in\sC^{(k)}(K_n) \colon \omega(e)=\omega(f)=1  \big\}}$, where  $e,f$ are a pair of adjacent edges of $K_n$. For any fixed $k\ge-1$, we have
	\be\label{eq: orders of cnnkef}
	C_{n,n+k}^{e,f}\asymp n^{n+\frac{3k-5}{2}}\,.
	\ee  
	Furthermore, for $k\ge1$ we have the following slightly finer estimate:
	\be\label{eq: orders of cnnkef finer}
	C_{n,n+k}^{e,f}=3\rho_kn^{n+\frac{3k-5}{2}}+O\left(  n^{n+\frac{3k-6}{2}} \right)\,,
	\ee	
	where $\rho_k$ is defined in \eqref{eq: cnnk from Wright1977}.                                                                                              \end{lemma}                                       
\begin{proof}
	We  prove \eqref{eq: orders of cnnkef} by induction on $k$. First, we verify the base cases  $k=-1,0$. 
	\begin{itemize}
		\item 	For $k=-1$, the quantity $C_{n,n-1}^{e,f}$ counts the number of spanning trees of $K_n$ containing $e,f$. This  is equal to the number of $3$-forests such that  $u,v,w$ lie in distinct connected components. Applying Lemma~\ref{lem:generalisation of stark lem2.2} with $p=q=3$, $S_1=\{u\}$, $S_2=\{v\}$, $S_3=\{w\}$, and $k_1=k_2=k_3=-1$, we have 
		\[
		C_{n,n-1}^{e,f}=(n-3)!\big[z^{n-3}\big]\big(W_{-1}'(z)\big)^3\stackrel{\eqref{eq:Nhat}}{=}3n^{n-4}\,.
		\]
		
		\item 	For $k=0$,  Proposition~\ref{prop: asymp of cnnef} gives
		\[
		C_{n,n}^{e,f}=\frac{3\sqrt{2\pi}}{4}n^{n-\frac{5}{2}}-\frac{7}{2}n^{n-3}+
		\frac{37\sqrt{2\pi}}{16}n^{n-\frac{7}{2}}+O(n^{n-4})\,.
		\]
	\end{itemize}
	Thus,  \eqref{eq: orders of cnnkef} holds for $k=-1,0$.
	
	For the inductive step, fix $k\ge1$ and assume  \eqref{eq: orders of cnnkef} holds for all $-1,0,\ldots,k-1$. We aim to show that it holds for $k$. 
	
	By \eqref{eq: cnnef decom into cases} and Lemmas~\ref{lem: J3}, \ref{lem: J1} and \ref{lem: J2}, we have  
	\begin{align}\label{eq: cnnkef in terms of coef}
	C_{n,n+k}^{e,f}&=\bigg(1-\frac{4(n+k-2)}{n(n-1)}\bigg)C_{n,n+k-2}+C_{n,n+k-2}^{e,f} \nonumber\\
	&\quad +(n-3)!\big[z^{n-3}\big]\sum_{ \substack{k_1,k_2,k_3\ge-1\\ k_1+k_2+k_3=k-2} }W_{k_1}'(z)W_{k_2}'(z)W_{k_3}'(z)\nonumber\\
	&\quad +3(n-3)!\big[z^{n-3}\big]\sum_{l=-1}^{k-1}W_{l}'(z)W_{k-2-l}''(z)
	-2\card{\widehat{\Lambda}_2(n,k)}\nonumber\\
	&\quad =R_1+R_2+R_3+C_{n,n+k-2}^{e,f}-2\card{\widehat{\Lambda}_2(n,k)}\,,
	\end{align}	
	where $R_1\coloneq \bigg(1-\frac{4(n+k-2)}{n(n-1)}\bigg)C_{n,n+k-2}$. 
	By the induction hypothesis, $C_{n,n+k-2}^{e,f}\asymp n^{n+\frac{3k-11}{2}}$. From the definition of of $\widehat{\Lambda}_2$ in Lemma~\ref{lem: J2} and the induction hypothesis, we have $\card{\widehat{\Lambda}_2(n,k)}\leq C_{n,n+k-1}^{e,f} \asymp n^{n+\frac{3k-8}{2}}$. To prove \eqref{eq: orders of cnnkef} for $k$, it suffices to show  $R_1\asymp n^{n+\frac{3k-7}{2}}$, $R_2\asymp n^{n+\frac{3k-6}{2}}$ and $R_3\asymp n^{n+\frac{3k-5}{2}}$. 
	
	From \eqref{eq: cnnk} we obtain 
	\begin{align}\label{eq: R1}
	R_1=\bigg(1-\frac{4(n+k-2)}{n(n-1)}\bigg) C_{n,n+k-2}\asymp n^{n+\frac{3k-7}{2}}\,.
	\end{align}
	The estimates $R_2\asymp n^{n+\frac{3k-6}{2}}$ and $R_3\asymp n^{n+\frac{3k-5}{2}}$ follow from Lemmas~\ref{lem: R_2} and \ref{lem: R_3}, respectively.  This completes the inductive step for \eqref{eq: orders of cnnkef}.

	Finally, the finer estimate \eqref{eq: orders of cnnkef finer} is obtained by substituting  $R_3=3\rho_kn^{n+\frac{3k-5}{2}}+O\left(  n^{n+\frac{3k-6}{2}} \right)$ (Lemma~\ref{lem: R_3}),   $R_2\asymp n^{n+\frac{3k-6}{2}}$ (Lemma~\ref{lem: R_2}), and the above bounds for $C_{n,n+k-2}^{e,f}$, $\card{\widehat{\Lambda}_2(n,k)}$ and $R_1$  into \eqref{eq: cnnkef in terms of coef}. 
\end{proof}

\begin{proof}[Proof of Proposition~\ref{prop: asymptotic of cnnk containing e and f}]
	From the proof of Lemma~\ref{lem: negligible terms}, we know that for $k\ge1$, 
	\be\label{eq: 4.37}
	C_{n,n+k}^{e,f}=R_1+R_2+R_3+C_{n,n+k-2}^{e,f}-2\card{\widehat{\Lambda}_2(n,k)}\,,
	\ee
	with
	$C_{n,n+k-2}^{e,f}=C_{n,n+k}^{e,f}\cdot  O\big(n^{-\frac{3}{2}}\big)$ and  $\card{\widehat{\Lambda}_2(n,k)}=C_{n,n+k}^{e,f}\cdot  O\big(n^{-\frac{3}{2}}\big)$. This gives 
	\be\label{eq: 4.37new}
	C_{n,n+k}^{e,f}=R_1+R_2+R_3+C_{n,n+k}^{e,f}\cdot  O\big(n^{-\frac{3}{2}}\big)\,.
	\ee
	Using  \eqref{eq: cnnk from Wright1977}  and the asymptotic order $C_{n,n+k}^{e,f}\asymp n^{n+\frac{3k-5}{2}}$,  we derive the finer estimate for $R_1$:
	\be\label{eq: R_1 finer estimate}
	R_1=\rho_{k-2}n^{n+\frac{3k-7}{2}}+C_{n,n+k}^{e,f}\cdot  O\big(n^{-\frac{3}{2}}\big)\,\quad \text{ for } \quad k\ge1\,.
	\ee
	To obtain the precise asymptotic expansions  in Proposition~\ref{prop: asymptotic of cnnk containing e and f}, we require finer estimates of $R_2$ and $R_3$ than those provided in Lemmas~\ref{lem: R_2} and \ref{lem: R_3}. For small values of $k$ (specifically $k\in\{1,2,\ldots,5\}$), these estimates can be computed manually (or straightforwardly using Maple). We illustrate the procedure with the $k=1$ case and skip the details of other cases. 
	
	First, from \eqref{eq: T'} we have
	\be\label{eq: theta'}
	\theta'(z)=-T'(z)=-\frac{1-\theta}{z\theta}\,.
	\ee 
	Combining \eqref{eq: theta'} with the explicit expressions for $W_k$ in Lemma~\ref{lem: W_k in terms of theta}  (see also \eqref{eq: W-1'} and \eqref{eq: W_1''}), we obtain 
	\be\label{eq:W-1'''}
	W_{-1}'(z)=-\theta\cdot \theta'(z)=\frac{1-\theta}{z}\,,\quad \text{and}\quad W_{-1}''(z)=\frac{(1-\theta)^2}{z^2\theta}\,,
	\ee
	\be\label{eq:W0'''}
	W_0'(z)=\frac{1}{2}\cdot \frac{(1-\theta)^3}{z\theta^2}\,,
	\quad \text{and} \quad W_0''(z)=\frac{1}{2}\cdot \frac{(1-\theta)^3(1+\theta)(2-\theta)}{z^2\theta^4}\,.
	\ee
	and
	\[
	W_1'(z)=\frac{1}{24z}\big[15\theta^{-5}-53\theta^{-4}+64\theta^{-3}-26\theta^{-2}-\theta^{-1}-1+2\theta\big]\,. 
	\]
	For $k=1$, we compute the sum over $k_1,k_2,k_3$ (with $k_1+k_2+k_3=-1$):
	\begin{align*}
	&\sum_{ \substack{k_1,k_2,k_3\ge-1\\ k_1+k_2+k_3=k-2} }z^3W_{k_1}'(z)W_{k_2}'(z)W_{k_3}'(z)\\
	=&3 \left[ \big(zW_{-1}'(z)\big)\cdot \big(zW_0'(z)\big)^2+\big(zW_{-1}'(z)\big)^2\cdot\big(zW_1'(z)\big) \right]\\
	=&\frac{1}{8}\cdot \left(15\theta^{-5}-77\theta^{-4}+143\theta^{-3}-81\theta^{-2}-95\theta^{-1}+185-123\theta+37\theta^2-4\theta^3\right)\,.
	\end{align*}
	Applying Lemma~\ref{lem: coef using theta expansion} to this expression gives, for $k=1$:
	\[
	\big[z^{n}\big]\sum_{ \substack{k_1,k_2,k_3\ge-1\\ k_1+k_2+k_3=k-2} }z^3W_{k_1}'(z)W_{k_2}'(z)W_{k_3}'(z)
	=\frac{5}{8}\frac{1}{\sqrt{2\pi}}e^nn^{\frac{3}{2}}-\frac{13}{8}e^nn+\frac{253}{32}\frac{1}{\sqrt{2\pi}}e^nn^{\frac{1}{2}}+O\big(e^n\big)\,.
	\]
	Recall the expansion for $(n-3)!$ (from \eqref{eq: n minus 3 factorial}):
	\[
	(n-3)!=\frac{n^{n-3}\sqrt{2\pi n}}{e^n}\big[1+\frac{37}{12}n^{-1}+O(n^{-2})\big]\,.
	\]
	We now compute $R_3$ for $k=1$:
	\begin{align*}
	R_3&=(n-3)!\big[z^{n-3}\big]\sum_{ \substack{k_1,k_2,k_3\ge-1\\ k_1+k_2+k_3=k-2} }W_{k_1}'(z)W_{k_2}'(z)W_{k_3}'(z)\\
	&=(n-3)!\big[z^{n}\big]\sum_{ \substack{k_1,k_2,k_3\ge-1\\ k_1+k_2+k_3=k-2} }z^3W_{k_1}'(z)W_{k_2}'(z)W_{k_3}'(z)\\
	&= \frac{5}{8}n^{n-1}-\frac{13}{8}\sqrt{2\pi}n^{n-\frac{3}{2}}+\frac{59}{6}n^{n-2}+O\big(n^{n-\frac{5}{2}}\big)\,.
	\end{align*}
	Next, we compute the sum for $R_2$ (with $l=-1,0$ for $k=1$):
	\begin{align*}
	&\sum_{l=-1}^{k-1}z^3W_{l}'(z)W_{k-2-l}''(z)\\
	=&\big(zW_{-1}'(z)\big)\cdot \big(z^2W_0''(z)\big) + \big(zW_0'(z)\big)\cdot \big(z^2W_{-1}''(z)\big)\\
	=&\theta^{-4}-3\theta^{-3}+\theta^{-2}+6\theta^{-1}-9+5\theta-\theta^2\,.
	\end{align*}
	Applying Lemma~\ref{lem: coef using theta expansion} to this expression give,  for $k=1$:
	\[
	\big[z^n\big]\sum_{l=-1}^{k-1}z^3W_{l}'(z)W_{k-2-l}''(z) =
	\frac{1}{4}e^nn-\frac{5}{3}\frac{1}{\sqrt{2\pi}}e^nn^{\frac{1}{2}}-\frac{1}{2}e^n+O\big(e^nn^{-\frac{1}{2}}\big)\,.
	\]
	We then calculate $R_2$ for  $k=1$: 
	\begin{align*}
	R_2&=3(n-3)!\big[z^{n-3}\big]\sum_{l=-1}^{k-1}W_{l}'(z)W_{k-2-l}''(z)\\
	&=3(n-3)!\big[z^{n}\big]\sum_{l=-1}^{k-1}z^3W_{l}'(z)W_{k-2-l}''(z)\\
	&=\frac{3}{4}\sqrt{2\pi}n^{n-\frac{3}{2}}-5n^{n-2}+\frac{13}{16}\sqrt{2\pi}n^{n-\frac{5}{2}}+O\big(n^{n-3}\big)\,.
	\end{align*}
	Finally, we combine the estimates for $R_1$, $R_2$ and $R_3$ to obtain the asymptotic expansion for $C_{n,n+1}^{e,f}$:
	\begin{align*}
	C_{n,n+1}^{e,f}&=R_1+R_2+R_3+O\big(n^{n-\frac{5}{2}}\big)\\
	&=n^{n-2}+ \frac{5}{8}n^{n-1}-\frac{13}{8}\sqrt{2\pi}n^{n-\frac{3}{2}}+\frac{59}{6}n^{n-2}
	+\frac{3}{4}\sqrt{2\pi}n^{n-\frac{3}{2}}-5n^{n-2}+O\big(n^{n-\frac{5}{2}}\big)\\
	&=\frac{5}{8}n^{n-1}-\frac{7}{8}\sqrt{2\pi}n^{n-\frac{3}{2}}+\frac{35}{6}n^{n-2}+O\big(n^{n-\frac{5}{2}}\big)\,.  \qedhere
	\end{align*}

\end{proof}

\subsubsection{Proof of Proposition~\ref{prop: key prop for k-excess} for  $k\in\{1,2,\ldots,5\}$}
\begin{proof}[Proof of Proposition~\ref{prop: key prop for k-excess} for $k\in\{1,2,\ldots,5\}$]
	For $k\in\{1,2,\ldots,5\}$, the asymptotic estimates of $C_{n,n+k}$ and $C_{n,n+k}^{e,f}$ given in \eqref{eq: cnnk asym} and \eqref{eq: cnnkef asym} have been established in Propositions~\ref{prop: asymptotic of cnnk for small k} and \ref{prop: asymptotic of cnnk containing e and f}, respectively. 
	The estimate \eqref{eq: p1 kUC asym} for $k\in\{1,2,\ldots,5\}$ then follows from a direct computation of the ratios  $\frac{C_{n,n+k}^{e,f}}{C_{n,n+k}}$.
\end{proof}

\begin{remark}\label{rem: pNC for adjacent pairs}
	From \eqref{eq: cnnk from Wright1977} and \eqref{eq: orders of cnnkef finer}, we  can obtain the p-NC property in Theorem~\ref{thm: p-NC for kUC} for a pair of adjacent edges $e,f$ for large $n$. Indeed, these two estimates imply that 
	\[
	p_1\coloneq \kUC\big[\omega(e)=\omega(f)=1\big]=\frac{C_{n,n+k}^{e,f}}{C_{n,n+k}}=\frac{3}{n^2}+O\big(n^{-\frac{5}{2}}\big)\,.
	\]
	Combined with Lemma~\ref{lem: first and second moment for k cycles}, this proves the desired p-NC property for the adjacent case when $n$ is sufficiently large.
\end{remark}

However, establish  the p-NC property for non-adjacent pairs requires a more precise estimate, which motivates the stronger bound \eqref{eq: p1 kUC asym}. Deriving \eqref{eq: p1 kUC asym} for arbitrary $k \ge 6$ in turn demands finer estimates of $R_2$ and $R_3$ than those provided in Lemma~\ref{lem: R_2} and Lemma~\ref{lem: R_3}. The required calculations are long and tedious, and are therefore deferred to the supplementary material. Nevertheless, the underlying ideas used to obtain these finer estimates are analogous to those in the proofs of  Lemmas~\ref{lem: R_2} and~\ref{lem: R_3}.

\section{Further remarks and an open question} \label{sec: kUC-remarks and ques}

We conclude with several remarks concerning the extent to which
Theorem~\ref{thm: p-NC for kUC} depends on the symmetry of the complete graph.
In particular, neither pairwise negative correlation nor the natural
monotonicity of one-edge marginals with respect to the excess holds for
general finite graphs.  Both failures can already be seen in the
unicyclic case by planar duality from a counterexample for
two-component spanning forests.  We then return to the corresponding
questions on the complete graph.

\subsection{A unicyclic counterexample via planar duality}

Pairwise negative correlation is not preserved, in general, under
conditioning on the number of open edges.   In the forest part of the
combined arXiv version~\cite{TZ2026}, we give a finite planar
graph $G$ and two distinct edges $e,f\in E(G)$ for which the uniform
probability measure on two-component spanning forests satisfies
\begin{equation}\label{eq:2forest-positive-correlation}
\mathbb P_{2\mathrm{UF}}
\bigl(\omega(e)=\omega(f)=1\bigr)
>
\mathbb P_{2\mathrm{UF}}
\bigl(\omega(e)=1\bigr)
\mathbb P_{2\mathrm{UF}}
\bigl(\omega(f)=1\bigr).
\end{equation}
The example is based on the graph appearing in
\cite[Example~2]{HSW2022}.
We explain that planar duality transfers this counterexample directly
to the present setting.

Let $G^\dagger$ be the planar dual of $G$.  For a two-component
spanning forest $F\subseteq E(G)$, define
\[
U^\dagger(F)
:=
\bigl\{
g^\dagger:
g\in E(G)\setminus F
\bigr\}
\subseteq E(G^\dagger),
\]
where $g^\dagger$ denotes the edge of $G^\dagger$ dual to $g$.

The standard planar forest--dual connected-subgraph correspondence
shows that $U^\dagger(F)$ is connected.  Moreover, since $F$ has
$|V(G)|-2$ edges, Euler's formula gives
\[
|U^\dagger(F)|
=
|E(G)|-|V(G)|+2
=
|V(G^\dagger)|.
\]
Thus $U^\dagger(F)$ is a connected spanning subgraph of $G^\dagger$
having the same number of edges as vertices, and hence is unicyclic.

Conversely, if $U$ is a connected unicyclic spanning subgraph of
$G^\dagger$, then the primal edges corresponding to
$E(G^\dagger)\setminus U$ form a two-component spanning forest of $G$.
Consequently,
\begin{equation}\label{eq:planar-dual-bijection}
F
\longmapsto
U^\dagger(F)
\end{equation}
is a bijection between the set of two-component spanning forests of
$G$ and the set of connected unicyclic spanning subgraphs of
$G^\dagger$.

For $g\in E(G)$, write
\[
X_g:=\mathbf 1_{\{g\in F\}},
\qquad
Y_{g^\dagger}
:=
\mathbf 1_{\{g^\dagger\in U^\dagger(F)\}}.
\]
By construction,
\[
Y_{g^\dagger}=1-X_g.
\]
Hence, for distinct edges $e,f\in E(G)$,
\[
\operatorname{Cov}
\bigl(
Y_{e^\dagger},Y_{f^\dagger}
\bigr)
=
\operatorname{Cov}
\bigl(
1-X_e,1-X_f
\bigr)
=
\operatorname{Cov}(X_e,X_f).
\]
Therefore \eqref{eq:2forest-positive-correlation} implies that the
dual edges $e^\dagger,f^\dagger$ are positively correlated under the
uniform probability measure on connected unicyclic spanning subgraphs
of $G^\dagger$.

We have thus obtained a finite planar graph for which the uniform
connected-unicyclic measure fails the p-NC property.  In particular,
the p-NC property for a fixed-excess connected measure cannot in
general be deduced from the corresponding unconditioned measure by a
general truncation principle.  The conclusion of
Theorem~\ref{thm: p-NC for kUC} therefore reflects additional structure of the
complete graph.

\subsection{Dependence of the marginal on the excess}

There is a related monotonicity question for one-edge marginals.  Given
a finite connected graph $G=(V,E)$ and an edge $e\in E$, one may ask
whether
\begin{equation}\label{eq:marginal-monotonicity-question}
\mathbb P_{k\mathrm{UC}}
\bigl(
\omega(e)=1
\bigr)
\end{equation}
is nondecreasing in the excess $k$.  This question was raised in the
earlier combined arXiv version~\cite{TZ2026}.  The answer is
negative for general finite graphs.

As shown in the forest part of the combined arXiv version
\cite{TZ2026}, the same example also satisfies, for one of the
distinguished edges $e$,
\begin{equation}\label{eq:forest-marginal-counterexample}
\mathbb P_{2\mathrm{UF}}
\bigl(\omega(e)=1\bigr)
>
\mathbb P_{\mathrm{UST}}
\bigl(\omega(e)=1\bigr).
\end{equation}
Thus, on the forest side, increasing the number of
components from one to two may increase the inclusion probability of a
fixed edge.

Passing to planar duals reverses edge inclusion.  More precisely, for
the corresponding dual edge $e^\dagger$, the bijection
\eqref{eq:planar-dual-bijection} gives
\[
\mathbf 1_{\{e^\dagger\in U^\dagger(F)\}}
=
1-\mathbf 1_{\{e\in F\}}.
\]
The usual planar duality between spanning trees of $G$ and spanning
trees of $G^\dagger$ therefore turns
\eqref{eq:forest-marginal-counterexample} into
\begin{equation}\label{eq:uc-marginal-counterexample}
\mathbb P_{0\mathrm{UC}}
\bigl(
\omega(e^\dagger)=1
\bigr)
<
\mathbb P_{\mathrm{UST}}
\bigl(
\omega(e^\dagger)=1
\bigr).
\end{equation}
Hence the one-edge marginal need not be nondecreasing when one passes
from spanning trees, corresponding to excess $-1$, to connected
unicyclic spanning subgraphs, corresponding to excess $0$.

For edge-transitive graphs, however, the marginal is determined
completely by symmetry.  If $G=(V,E)$ is finite, connected, and
edge-transitive, then all one-edge marginals under
$\mathbb P_{k\mathrm{UC}}$ are equal.  Since every configuration in
$\mathcal C^{(k)}(G)$ has exactly $|V|+k$ edges,
\[
\sum_{e\in E}
\mathbb P_{k\mathrm{UC}}
\bigl(
\omega(e)=1
\bigr)
=
|V|+k.
\]
Therefore, for every $e\in E$,
\begin{equation}\label{eq:edge-transitive-marginal}
\mathbb P_{k\mathrm{UC}}
\bigl(
\omega(e)=1
\bigr)
=
\frac{|V|+k}{|E|}.
\end{equation}
Thus, whenever the corresponding measures are defined, the marginal is
strictly increasing in $k$.  In particular, the monotonicity in
\eqref{eq:marginal-monotonicity-question} holds automatically for
complete graphs.

\subsection{The fixed-excess problem on complete graphs}

Theorem~\ref{thm: p-NC for kUC} proves the p-NC property for every fixed
$k\ge0$ when $n$ is sufficiently large.  The restriction to large $n$
comes from the asymptotic nature of the coefficient estimates used in
the proof, rather than from a known obstruction.  In particular, we do
not know whether the restriction is necessary.

This leads to the following question.

\begin{question}\label{q:all-nk}
	For which pairs $(n,k)$ satisfying
	\[
	0\le k\le \binom{n}{2}-n
	\]
	does the uniform probability measure on connected spanning subgraphs
	of $K_n$ with excess $k$ satisfy the p-NC property?  In particular,
	does the p-NC property hold for every admissible pair $(n,k)$?
\end{question}

The complete graph has substantially more symmetry than the examples
above.  In particular, there are only two orbits of pairs of distinct
edges, according to whether the two edges are adjacent or disjoint,
and the one-edge marginal is given explicitly by
\[
\mathbb P_{k\mathrm{UC}}
\bigl(
\omega(e)=1
\bigr)
=
\frac{n+k}{\binom n2}.
\]
It would be interesting to determine whether the asymptotic conclusion
of Theorem~\ref{thm: p-NC for kUC} in fact holds throughout the entire
admissible range of $(n,k)$.

%
%
%\bibliography{kUC_ref}
%\bibliographystyle{plain}
%
%\end{document}

%% If kUC.tex does NOT already contain its bibliography, use:
%%
%\bibliographystyle{plain}
%\bibliography{kUC_ref}

\putbib[NCUS_ref]

\end{bibunit}

\clearpage

%%%%%%%%%%%%%%%%%%%%%%%%%%%%%%%%%%%%%%%%%%%%%%%%%%%%%%%%%%%%
%%%%%%%%%%%%%%%%%%%%%%%%%%%%%%%%%%%%%%%%%%%%%%%%%%%%%%%%%%%%
%% SUPPLEMENTARY MATERIAL FOR PART III
%%%%%%%%%%%%%%%%%%%%%%%%%%%%%%%%%%%%%%%%%%%%%%%%%%%%%%%%%%%%
%%%%%%%%%%%%%%%%%%%%%%%%%%%%%%%%%%%%%%%%%%%%%%%%%%%%%%%%%%%%

\part*{Supplementary Material for Part III}
\addcontentsline{toc}{part}{Supplementary Material for Part III}

\startsupplement

\begin{center}
    {\LARGE\bfseries
    Supplementary Material:\\[0.3em]
    Negative Correlation for Uniform $k$-Excess Connected
    Spanning Subgraphs of the Complete Graph}
\end{center}

\medskip

\begin{center}
Pengfei Tang \qquad Zibo Zhang
\end{center}

\bigskip

\begin{bibunit}[plain]

The main paper is extracted from the earlier combined arXiv version
\cite{TZ2026}. The present document contains the technical proofs
and computations deferred from the main paper. Most of the material below appeared in the appendix of \cite{TZ2026}, with an additional introductory section
included here to make the supplement self-contained.

	\section{Introduction}\label{sec:intro}
	
	In this section we collect the definitions, notation, equations, lemmas and propositions from the main paper that are used in the supplementary material. 
	
\subsection{Notation from the main paper}	
		In the proof of Theorem~1.1 in the main paper, we rely on exponential generating functions and introduce the following objects:
	\begin{itemize}
		\item $C_{n,n+k}$: the number of connected spanning subgraphs of $K_n$ with excess $k$.
		\item $C_{n,n+k}^{e,f}$: the number of such subgraphs that contain a fixed pair of adjacent edges $e$ and $f$.
		\item $\bbC[[z]]$: the ring of formal power series with complex coefficients.
		\item $[z^n]f(z)$: the coefficient of $z^n$ in $f(z)\in\bbC[[z]]$.
		\item $W_k(z) = \sum_{n\ge1} C_{n,n+k}\frac{z^n}{n!}$: the exponential generating function for connected subgraphs with excess $k$.
		\item $T(z) = \sum_{n\ge1}\frac{n^{n-1}}{n!}z^n$: the exponential generating function for rooted trees.
		\item $\theta = 1 - T(z)$: a convenient formal power series.
		
		\item 
		For $k\ge1$, define
		\be\label{eq: def of R_3}
		R_3=R_3(k)\coloneq (n-3)!\big[z^{n-3}\big]\sum_{ \substack{k_1,k_2,k_3\ge-1\\ k_1+k_2+k_3=k-2} }W_{k_1}'(z)W_{k_2}'(z)W_{k_3}'(z)\,.
		\ee
		
		\item For $k\ge1$, define
		\be\label{eq: def of R_2}
		R_2=R_2(k)\coloneq 3(n-3)!\big[z^{n-3}\big]\sum_{l=-1}^{k-1}W_l'(z)W_{k-2-l}''(z)\,.
		\ee 
		
		\item $\hO(\theta^p)$: a finite sum of terms of the form $c\theta^s$ with $s\ge p$.
	\end{itemize}

\subsection{Results from the main paper that will be needed for the deferred proofs}
We collect below several facts that will be used in the deferred proofs:
\begin{enumerate}
	\item Theorem~\ref{thm: FSS2004 thm1} (Theorem~2.3 in the main paper).
	\begin{theorem}\label{thm: FSS2004 thm1}
		\begin{enumerate}
			\item[(i)] The exponential generating function of unrooted trees is 
			\[
			W_{-1}(z)=T(z)-\frac{1}{2}T^2(z)\,.
			\] 
			
			\item[(ii)] The exponential generating function of connected graphs with excess $0$ (unicyclic graphs) is 
			\[
			W_0(z)=\frac{1}{2}\log\frac{1}{1-T(z)}-\frac{1}{2}T(z)-\frac{1}{4}T^2(z)\,.
			\]
			
			\item[(iii)] The exponential generating function of connected graphs with excess $k\geq1$ is a rational function of $T(z)$: there exist polynomials $A_k$ such that 
			\[
			W_k(z)=\frac{A_k\big(T(z)\big)}{\big[1-T(z)\big]^{3k}}\,.
			\] 
			For instance (see page~135 of \cite{FS2009}),
			\[
			W_1(z)=\frac{1}{24}\frac{T^4(z)\big(6-T(z)\big)}{\big[1-T(z)\big]^3}\,,\quad
			W_2(z)=\frac{1}{48}\frac{T^4(z)\big[2+28T(z)-23T^2(z)+9T^3(z)-T^4(z)\big]}{\big[1-T(z)\big]^6}
			\,.
			\]
		\end{enumerate}
	\end{theorem}

\item Lemma~\ref{lem: facts of T(z)} (Lemma~2.4 in the main paper).
	\begin{lemma}\label{lem: facts of T(z)}
		The formal power series $T(z)\coloneq \sum_{n\ge1} n^{n-1}\frac{z^n}{n!}$ has the following properties:
		\begin{enumerate}
			\item[(1)] It is well-known (see, e.g., equation~(44), p.~127 of \cite{FS2009}) that $T(z)$ satisfies 
			\be\label{eq:T1}
			T(z)=ze^{T(z)}\,.
			\ee

			\item[(2)] The derivative of $T(z)$ satisfies(see, e.g., formula~(2.7) in the proof of Lemma~2.1 of \cite{Stark2011})
			\be\label{eq: T'}
			T'(z)=\frac{T(z)}{z\big[1-T(z)\big]}\,.
			\ee
			
			\item[(3)] For an integer $k\in[1,n]$, we have  (see, e.g., Proposition~1 of \cite{FGKP1995})
			\be\label{eq:T3}
			\big[z^n\big]T^k(z)=\frac{kn^{n-k-1}}{(n-k)!}.
			\ee
			
			\item[(4)] For an integer $k\in[1,n]$, we have (see, e.g.,  formula~(2.13) in Lemma~2.3 of \cite{Stark2011}) 
			\be\label{eq:T4}
			\big[z^n\big]\frac{T^k(z)}{1-T(z)}=\frac{n^{n-k}}{(n-k)!}\,.
			\ee
		\end{enumerate}
	\end{lemma}

\item Equation~\ref{eq:expansion of T(z)} (Equation~(2.8) in the main paper).
Write $w=\sqrt{1-ez}$. The formal power series $T(z)$ admits the following asymptotic expansion at its dominant singularity $e^{-1}$ (see, e.g., formula~(4.3) in \cite{Stark2011}):
\be\label{eq:expansion of T(z)}
T(z)=1-\sqrt{2}w+\frac{2}{3}w^2-\frac{11}{36}\sqrt{2}w^3+\frac{43}{135}w^4-\frac{769}{4320}\sqrt{2}w^5+O(w^6)
\ee

\item Equality~\eqref{eq:T_alpha_beta} (Equality~(2.11) in the main paper).
For integers $\alpha\ge0,\beta\ge2$:
\begin{align}\label{eq:T_alpha_beta}
\big[z^n\big]\frac{T^{\alpha}(z)}{\big[1-T(z)\big]^{\beta}}
&=c_0\big[z^n\big]w^{-\beta}+c_1\big[z^n\big]w^{-\beta+1}+c_2\big[z^n\big]w^{-\beta+2}+O(e^nn^{\frac{\beta}{2}-\frac{5}{2}})\nonumber\\
&=c_0\frac{e^{n}n^{\frac{\beta}{2}-1}}{\Gamma(\frac{\beta}{2})}\big[1+\frac{\beta(\beta-2)}{8n}+O(n^{-2})\big]+c_1\frac{e^{n}n^{\frac{\beta-1}{2}-1}}{\Gamma(\frac{\beta-1}{2})}\big[1+O(n^{-1})\big] \nonumber\\
& \quad +c_2\frac{e^nn^{\frac{\beta-2}{2}-1}}{\Gamma(\frac{\beta-2}{2})}\big[1+O(n^{-1})\big]\cdot \mathbf{1}_{\beta>2}+O(e^nn^{\frac{\beta}{2}-\frac{5}{2}}) \nonumber\\
&=\frac{2^{-\frac{\beta}{2}}}{\Gamma(\frac{\beta}{2})}e^nn^{\frac{\beta}{2}-1}
+\frac{2^{-\frac{\beta}{2}}\big(\frac{\sqrt{2}\beta}{3}-\sqrt{2}\alpha\big)}{\Gamma\big(\frac{\beta-1}{2}\big)}e^nn^{\frac{\beta}{2}-\frac{3}{2}} \nonumber\\
&\quad + \bigg[\frac{2^{-\frac{\beta}{2}} \beta(\beta-2) }{8\Gamma(\frac{\beta}{2})}+\frac{c_2}{\Gamma(\frac{\beta-2}{2})}\cdot \mathbf{1}_{\beta>2}\bigg] e^nn^{\frac{\beta}{2}-2} +O(e^nn^{\frac{\beta}{2}-\frac{5}{2}})\,.
\end{align}

\item Proposition~\ref{prop: asymptotic of cnn} (Proposition~2.5 in the main paper).
\begin{proposition}\label{prop: asymptotic of cnn}
	The number  of unicyclic subgraphs of $K_n$ has the following asymptotic behavior:
	\be\label{eq: cnn}
	C_{n,n}=\sqrt{\frac{\pi}{8}}n^{n-\frac{1}{2}}-\frac{7}{6}n^{n-1}+\frac{1}{24}\sqrt{\frac{\pi}{2}}\,n^{n-\frac{3}{2}}+O(n^{n-2})\,.
	\ee
\end{proposition}

\item Equation~\eqref{eq: W-1'} (Equation~(2.17) in the main paper).
	\begin{equation}\label{eq: W-1'}
W_{-1}'(z)=T'(z)-T(z)T'(z)\stackrel{\eqref{eq: T'}}{=}\frac{T(z)}{z}\,,
\end{equation}

\item 	Equation~\eqref{eq: n minus 3 factorial} (Equation~(2.21) in the main paper).
\be\label{eq: n minus 3 factorial}
(n-3)!=\frac{n!}{n(n-1)(n-2)}=\frac{n^{n-3}\sqrt{2\pi n}}{e^n}\big[1+\frac{37}{12}n^{-1}+O(n^{-2})\big]\,.
\ee

\item Equation~\eqref{eq: thm4 in Wrightone} (Equation~(2.25) in the main paper).
\be\label{eq: thm4 in Wrightone}
W_k=\sum_{s=-3k}^{2}c_{k,s}\theta^s\,,
\ee

\item Lemma~\ref{lem: W_k in terms of theta} (Lemma~2.15 in the main paper).

\begin{lemma}\label{lem: W_k in terms of theta}
	Let $\theta=\theta(z)= 1-T(z)$. For $k\in\{-1,0,1,2,\ldots,5\}$, $W_k$ has the following expressions:
	\be\label{eq: W_minus1 in terms of theta}
	W_{-1}=\frac{1}{2}-\frac{1}{2}\theta^2\,,
	\ee
	and 
	\be\label{eq: W_0 in terms of theta}
	W_0
	=-\frac{1}{2}\log \theta -\frac{3}{4}+\theta-\frac{1}{4}\theta^2\,,
	\ee
	and 
	\be\label{eq: W_1 in terms of theta}
	W_1=\frac{1}{24}\big[5\theta^{-3}-19\theta^{-2}+26\theta^{-1}-14+\theta+\theta^2\big]\,,
	\ee
	and
	\begin{align}\label{eq: W_2 in terms of theta}
	W_2&=\frac{1}{48}\big[15\theta^{-6}-65\theta^{-5}+108\theta^{-4}-87\theta^{-3}+42\theta^{-2}-23\theta^{-1}+12-\theta-\theta^2 \big]
	\end{align}
	and 
	\begin{align}\label{eq: W_3 in terms of theta}
	W_3&=\frac{1105}{1152 \theta^{9}} - \frac{1945}{384 \theta^{8}} + \frac{6353}{576 \theta^{7}} - \frac{233}{18 \theta^{6}} + \frac{8929}{960 \theta^{5}} - \frac{1521}{320 \theta^{4}} + \frac{719}{360 \theta^{3}} - \frac{67}{96 \theta^{2}}+ \frac{115}{384 \theta} \nonumber
	\\& - \frac{293}{1920} + \frac{13 \theta}{960} + \frac{19 \theta^{2}}{1440}
	\end{align}
	and
	\begin{align}\label{eq: W_4 in terms of theta}
	W_4&=\frac{565}{128 \theta^{12}} - \frac{21295}{768 \theta^{11}} + \frac{172337}{2304 \theta^{10}} - \frac{263105}{2304 \theta^{9}} + \frac{255437}{2304 \theta^{8}} - \frac{215153}{2880 \theta^{7}} + \frac{220763}{5760 \theta^{6}} - \frac{92893}{5760 \theta^{5}} \nonumber
	\\&+ \frac{16741}{2880 \theta^{4}} - \frac{21977}{11520 \theta^{3}} + \frac{6577}{11520 \theta^{2}} - \frac{827}{3840 \theta} + \frac{1241}{11520} - \frac{19 \theta}{1920} - \frac{3 \theta^{2}}{320}
	\end{align}
	and
	\begin{align}\label{eq: W_5 in terms of theta}
	W_5&=\frac{82825}{3072 \theta^{15}} - \frac{603965}{3072 \theta^{14}} + \frac{323385}{512 \theta^{13}} - \frac{49027387}{41472 \theta^{12}} + \frac{119862917}{82944 \theta^{11}} - \frac{171249163}{138240 \theta^{10}} + \frac{82217167}{103680 \theta^{9}} \nonumber
	\\ &- \frac{10454209}{25920 \theta^{8}} + \frac{165837761}{967680 \theta^{7}} - \frac{61001153}{967680 \theta^{6}} + \frac{10019617}{483840 \theta^{5}} - \frac{28639}{4608 \theta^{4}} + \frac{735343}{414720 \theta^{3}} - \frac{199879}{414720 \theta^{2}} \nonumber
	\\&+ \frac{361}{2160 \theta} - \frac{59761}{725760} + \frac{5611 \theta}{725760} + \frac{863 \theta^{2}}{120960}\,.
	\end{align}
\end{lemma}

\item Equation~\eqref{eq:DW0} (Equation~(2.34) in the main paper).
	\be\label{eq:DW0}
\sD W_0=(1-\theta^{-1})\cdot \big[-\frac{1}{2}\theta^{-1}+1-\frac{1}{2}\theta\big]=\frac{(1-\theta)^3}{2\theta^2}\,.
\ee

\item Equation~\eqref{eq:thm2 in Wrightone} (Equation~(2.35) in the main paper).
\be\label{eq:thm2 in Wrightone}
(1-\theta)^{k+1}W_{k+1}
=\int_{\theta}^{1}(1-t)^k\cJ_k(t)dt\,.
\ee

\item Equation~\eqref{eq: recursive cks} (Equation~(2.36) in the main paper).
	\be\label{eq: recursive cks}
c_{k+1,-3(k+1)}=\frac{j_k}{3k+3}=\frac{3k}{2}c_{k,-3k}+\sum_{h=1}^{k-1}\frac{3h(k-h)}{2(k+1)}c_{h,-3h}c_{k-h,-3(k-h)}\,,
\ee

\item 	Equation~\eqref{eq: orders of cnnkef} (Equation~(2.44) in the main paper).
\be\label{eq: orders of cnnkef}
C_{n,n+k}^{e,f}\asymp n^{n+\frac{3k-5}{2}}\,.
\ee

\item Equation~\eqref{eq: 4.37new} (Equation~(2.49) in the main paper).
	\be\label{eq: 4.37new}
C_{n,n+k}^{e,f}=R_1+R_2+R_3+C_{n,n+k}^{e,f}\cdot  O\big(n^{-\frac{3}{2}}\big)\,.
\ee

\item Equations~\eqref{eq: theta'}, \eqref{eq:W-1'''} and \eqref{eq:W0'''} (Equations~(2.51), (2.52) and (2.53) in the main paper).
	\be\label{eq: theta'}
\theta'(z)=-T'(z)=-\frac{1-\theta}{z\theta}\,.
\ee 
and
	\be\label{eq:W-1'''}
W_{-1}'(z)=-\theta\cdot \theta'(z)=\frac{1-\theta}{z}\,,\quad \text{and}\quad W_{-1}''(z)=\frac{(1-\theta)^2}{z^2\theta}\,,
\ee
and
\be\label{eq:W0'''}
W_0'(z)=\frac{1}{2}\cdot \frac{(1-\theta)^3}{z\theta^2}\,,
\quad \text{and} \quad W_0''(z)=\frac{1}{2}\cdot \frac{(1-\theta)^3(1+\theta)(2-\theta)}{z^2\theta^4}\,.
\ee

\end{enumerate}

\subsection{Claims whose proofs are deferred} \label{sec: p-NC for kUC}

\begin{enumerate}
	\item The case of $k\ge6$ for  Proposition~\ref{prop: key prop for k-excess} (Proposition~2.1 in the main paper): 
	\begin{proposition}\label{prop: key prop for k-excess}
		Suppose $k\ge0$ is a fixed integer,  edges $e$ and $f$ are a pair of adjacent edges of $K_n$, and $p_1\coloneq \kUC\big[\omega(e)=\omega(f)=1\big]$. Then there exists constants $\big\{\alpha_{k,i},\beta_{k,i},i=1,2,3  \}$  such that as $n\to\infty$ 
		\be\label{eq: cnnk asym}
		C_{n,n+k}=\sum_{i=1}^{3}\alpha_{k,i}n^{n+\frac{3k-i}{2}}+O\left( n^{n+\frac{3k-4}{2}} \right)\,,
		\ee
		and
		\be\label{eq: cnnkef asym}
		C_{n,n+k}^{e,f}
		=\sum_{i=1}^{3}\beta_{k,i}n^{n+\frac{3k-4-i}{2}}+O\left( n^{n+\frac{3k-8}{2}} \right)
		\ee
		Moreover, these constants $\big\{\alpha_{k,i},\beta_{k,i},i=1,2,3  \}$ satisfy  
		\be\label{eq: p1 kUC asym}
		p_1=\frac{	 C_{n,n+k}^{e,f}}{C_{n,n+k}}=\frac{3}{n^2}+\frac{9(k+1)}{n^3}+O(n^{-\frac{7}{2}})\,.
		\ee
	\end{proposition}
	
	\item Computation of  the coefficients $c_i$'s in \eqref{eq: computable coef} (Equation (2.9) in the main paper).
	\be\label{eq: computable coef}
	\frac{T^{\alpha}(z)}{\big[1-T(z)\big]^{\beta}}=c_0w^{-\beta}+c_1w^{-\beta+1}+c_2 w^{-\beta+2}+\cdots\,,
	\ee
	with explicitly computable constants $c_0,c_1,c_2\ldots$. For instance, 
	\[
	c_0=2^{-\frac{\beta}{2}},\quad c_1=2^{-\frac{\beta}{2}}\big(\frac{\sqrt{2}\beta}{3}-\sqrt{2}\alpha\big)\,,
	\text{ and } c_2=2^{-\frac{\beta}{2}}\frac{36\alpha^2+60\alpha+4\beta^2-7\beta-24\alpha\beta}{36}\,.
	\]
	
	\item  Details in the proof of Proposition~\ref{prop: asymptotic of cnnk} (Proposition 2.12 in the main paper),  the case of  $k\ge0$ in \eqref{eq: cnnk from Wright1977}.
	\begin{proposition}\label{prop: asymptotic of cnnk}
		For $k\ge-1$, we have
		\be\label{eq: cnnk from Wright1977}
		C_{n,n+k}=\rho_k n^{n+\frac{3k-1}{2}}\big[ 1+O(n^{-\frac{1}{2}})  \big]\,,
		\ee
		where $\rho_{-1}=1$, and for $k\ge0$, 
		\[
		\rho_k=\frac{\sqrt{\pi}\cdot 2^{(1-3k)/2}\sigma_k }{\Gamma(1+\frac{3k}{2})}\,,
		\]
		with the sequence $\{\sigma_i\}$ given by
		\[
		4\sigma_0=1, \quad 16\sigma_1=5,\quad 16\sigma_2=15,
		\]
		and 
		\be\label{eq:recursive-sigma}
		\sigma_{k+1}=\frac{3(k+1)\sigma_k}{2}+\sum_{s=1}^{k-1}\sigma_s\sigma_{k-s}\quad (k\ge2)\,.
		\ee
		In particular, for all $k\ge-1$, we have 
		\be\label{eq: cnnk}
		C_{n,n+k}\asymp n^{n+\frac{3k-1}{2}}\,.
		\ee
	\end{proposition}

	\item Proof of Lemma~\ref{lem: coef using theta expansion} ( Lemma~2.14 in the main paper).
	\begin{lemma}\label{lem: coef using theta expansion}
		Let $\theta=\theta(z)=1-T(z)$, and let $F(z)=\sum_{s=-m}^{t}f_{s}\theta^{s}$ where $m,t$ are  positive integers and  $\{f_s\colon s=-m,\ldots,t\}$ are constants with $f_{-m}>0$. Then   
		\be\label{eq: coef using theta expansion}
		\big[z^n\big]F(z)=	f_{-m}\frac{2^{-\frac{m}{2}}}{\Gamma(\frac{m}{2})}e^nn^{\frac{m}{2}-1}
		+g_{m,1}e^nn^{\frac{m}{2}-\frac{3}{2}}+g_{m,2}e^nn^{\frac{m}{2}-2}	+O\left(e^nn^{\frac{m}{2}-\frac{5}{2}}\right)\,,	
		\ee
		where \[
		g_{m,1}=\frac{2^{-\frac{m}{2}}\cdot \frac{\sqrt{2}m}{3}}{\Gamma\left(\frac{m-1}{2}\right)}f_{-m} +\frac{2^{-\frac{m-1}{2}}}{\Gamma(\frac{m-1}{2})}f_{-m+1}
		\]
		and
		\[
		g_{m,2}=\frac{2^{-\frac{m}{2}}(2m^2+m)}{18\Gamma\left(\frac{m-2}{2}\right)}f_{-m} +\frac{2^{-\frac{m-1}{2}}\cdot \frac{\sqrt{2}(m-1)}{3}}{\Gamma\left(\frac{m-2}{2}\right)}f_{-m+1}+\frac{2^{-\frac{m-2}{2}}}{\Gamma(\frac{m-2}{2})}f_{-m+2}\,.
		\] 
		We adopt the convention that $\Gamma(0)=\infty$ and $\Gamma(-\frac{1}{2})=-2\sqrt{\pi}$.

	\end{lemma}
	
	\item Proof of Lemma~\ref{lem: R_3} (Lemma 2.18 in the main paper).
	\begin{lemma}\label{lem: R_3}
		Suppose $k\ge1$. Then 
		\be\label{eq:R_3}
		R_3=3\rho_kn^{n+\frac{3k-5}{2}}+O\left(  n^{n+\frac{3k-6}{2}} \right)\,.
		\ee
	\end{lemma}
	
	\item Proof of Lemma~\ref{lem: R_2} (Lemma 2.21 in the main paper).
	\begin{lemma}\label{lem: R_2}
		Suppose $k\ge1$. Then
		\be\label{eq:R_2}
		R_2=3\rho_{k-1}n^{n+\frac{3k-6}{2}}\big[1+O(n^{-\frac{1}{2}})\big]\,.
		\ee
	\end{lemma}

\end{enumerate}

\section{Deferred proofs from Section~2 in the main paper}

\begin{proof}[Computation of the coefficients $c_i$'s in \eqref{eq: computable coef}]
	First, recall Newton's binomial theorem for negative powers:
	\[
	(x+a)^{-n}=\sum_{k=0}^{\infty}\binom{-n}{k}x^ka^{-n-k}=\sum_{k=0}^{\infty}(-1)^k\binom{n+k-1}{k}x^ka^{-n-k}\,,
	\]
	where the second equality follows from the identity:
	\begin{align*}
	\binom{-n}{k}&=\frac{(-n)(-n-1)\cdots (-n-(k-1))}{k!}\\
	&=(-1)^k\frac{(n+k-1)(n+k-2)\cdots n}{k!}\\
	&=(-1)^k\binom{n+k-1}{k}\,.
	\end{align*}
	In particular, setting $a=1$ and $x=o(1)$, we obtain the expansion:
	\[
	(1+x)^{-n}=1-nx+\frac{n(n+1)}{2}x^2+O(x^3)\,.
	\]
	Using this binomial theorem for negative powers, we compute $\big[1-T(z)\big]^{-\beta}$ for $\beta\ge2$:
	\begin{align*}
	\big[1-T(z)\big]^{-\beta}
	&\stackrel{\eqref{eq:expansion of T(z)}}{=}\big[\sqrt{2}w-\frac{2}{3}w^2+\frac{11}{36}\sqrt{2}w^3+O(w^4)\big]^{-\beta}\\
	&=2^{-\frac{\beta}{2}}w^{-\beta}\big[1-\frac{\sqrt{2}}{3}w+\frac{11}{36}w^2+O(w^3)\big]^{-\beta}\\
	&=2^{-\frac{\beta}{2}}w^{-\beta}\cdot \bigg[1+\beta\big(\frac{\sqrt{2}}{3}w-\frac{11}{36}w^2\big)+\frac{\beta(\beta+1)}{2}\big(-\frac{\sqrt{2}}{3}w+\frac{11}{36}w^2\big)^2+O(w^3)\bigg]\\
	&=2^{-\frac{\beta}{2}}w^{-\beta}\cdot\bigg[1+\frac{\sqrt{2}\beta}{3}w+\frac{4\beta^2-7\beta}{36}w^2+O(w^3)\bigg]
	\end{align*}
	On the other hand, we use Newton's binomial theorem for positive powers to expand $T^{\alpha}(z)$:
	\begin{align*}
	T^{\alpha}(z)&\stackrel{\eqref{eq:expansion of T(z)}}{=}\big[1-\sqrt{2}w+\frac{2}{3}w^2-\frac{11}{36}\sqrt{2}w^3+O(w^4)\big]^{\alpha}\\
	&=1-\sqrt{2}\alpha w+\big(\alpha^2+\frac{5}{3}\alpha\big)w^2+O(w^3)\,.
	\end{align*}
	Multiplying these two expansions together, we obtain the desired expression:
	\begin{align*}
	&\frac{T^{\alpha}(z)}{\big[1-T(z)\big]^{\beta}}\\
	=&2^{-\frac{\beta}{2}}w^{-\beta}\cdot\bigg[1+\frac{\sqrt{2}\beta}{3}w+\frac{4\beta^2-7\beta}{36}w^2+O(w^3)\bigg]\cdot \bigg[1-\sqrt{2}\alpha w+\big(\alpha^2+\frac{5}{3}\alpha\big)w^2+O(w^3)\bigg]\\
	=&2^{-\frac{\beta}{2}}w^{-\beta}+2^{-\frac{\beta}{2}}\big(\frac{\sqrt{2}\beta}{3}-\sqrt{2}\alpha\big)w^{-\beta+1}\\
	&\quad +2^{-\frac{\beta}{2}}\frac{36\alpha^2+60\alpha+4\beta^2-7\beta-24\alpha\beta}{36} w^{-\beta+2}  +O(w^{-\beta+3})\,. \qedhere
	\end{align*}
\end{proof}

To prove Lemma~\ref{lem: coef using theta expansion}, we first analyze the positive powers of $\theta$.
\begin{lemma}\label{lem: coef of theta positive powers}
	Recall that $\theta=\theta(z)\coloneq 1-T(z)$. We have 
	\be\label{eq:coef of theta}
	\big[z^n\big]\theta=-\big[z^n\big]T(z)=-\frac{n^{n-1}}{n!}=-\frac{1}{\sqrt{2\pi}}e^nn^{-\frac{3}{2}}\cdot \Big[1-\frac{1}{12n}+O\big(n^{-2}\big)\Big]\,,
	\ee
	and for any fixed  integer $m\ge2$,
	\be\label{eq:coef of theta positive powers}
	\big[z^n\big]\theta^m=O\big(e^nn^{-\frac{5}{2}}\big)\,.
	\ee
\end{lemma}
\begin{proof}
	We start by recalling Stirling's formula, which states
	\[
	n!=\sqrt{2\pi}e^{-n} n^{n+\frac{1}{2}}\Big[1+\frac{1}{12n}+O\big(n^{-2}\big)\Big]\,.
	\] 
	Taking the reciprocal of both sides, we obtain: 
	\[
	\frac{1}{n!}=\frac{1}{\sqrt{2\pi}}e^n n^{-n-\frac{1}{2}}\Big[1-\frac{1}{12n}+O\big(n^{-2}\big)\Big]\,.
	\]
	By the definition of $T(z)$,  we have
	\[
	\big[z^n\big]\theta=-\big[z^n\big]T(z)=-\frac{n^{n-1}}{n!}\,.
	\]
	Substituting the reciprocal Stirling expansion above into this expression gives the desired asymptotic form for $\big[z^n\big]\theta$.

	Next, we consider  $m\ge2$ and compute  $\big[z^n\big]\theta^m$. Using the binomial theorem to expand $\theta^m=\big(1-T(z)\big)^m$, we have
	\[
	\big[z^n\big]\theta^m=\sum_{k=1}^{m}\binom{m}{k}(-1)^k\big[z^n\big]T^k(z)\,
	\]
	(taking the sum from $k=1$ since the $k=0$ term is $1$, whose coefficient $\big[z^n\big]$ is $0$ for $n\ge1$). Extracting the coefficient of $z^n$ and applying \eqref{eq:T3}, we get
	\begin{align*}
	\big[z^n\big]\theta^m&=\sum_{k=1}^{m}\binom{m}{k}(-1)^k\big[z^n\big]T^k(z)\\
	&\stackrel{\eqref{eq:T3}}{=}\sum_{k=1}^{m}\binom{m}{k}(-1)^k\frac{kn^{n-k-1}}{(n-k)!} \,.
	\end{align*}
	To simplify this sum, we rewrite $(n-k)!$ using the product form $(n-k)!=\frac{n!}{\prod_{j=0}^{k-1}(n-j)}$, which gives:
	\[
	\frac{kn^{n-k-1}}{(n-k)!}=\frac{kn^{n-k-1}\prod_{j=0}^{k-1}(n-j)}{n!}\,.
	\] 
	Expanding the product $\prod_{j=0}^{k-1}(n-j)=n^k+O\big(n^{k-1}\big)$ (since it is a polynomial of degree $k$ in $n$ with leading term $n^k$), we substitute back to find:
	\begin{align*}
	\big[z^n\big]\theta^m
	&=\frac{1}{n!}\sum_{k=1}^{m}\binom{m}{k}(-1)^kkn^{n-k-1}\left(n^k+O\big(n^{k-1}\big)\right)\\
	&=\frac{1}{n!}\left[n^{n-1}\cdot \sum_{k=1}^{m}(-1)^k\binom{m}{k}k+O\big(n^{n-2}\big) \right]\\
	&=\frac{1}{n!}O(n^{n-2})=O\big(e^nn^{-\frac{5}{2}}\big)\,,
	\end{align*}
	where we use the fact that 
	\[
	\sum_{k=1}^{m}(-1)^k\binom{m}{k}k=-m\sum_{j=0}^{m-1}(-1)^{j}\binom{m-1}{j}=-m(1-1)^{m-1}=0\,. \qedhere
	\]

\end{proof}

Next, we analyze the negative powers of $\theta$.
\begin{lemma}\label{lem: coef of theta negative powers}
	For any integer $\beta\ge1$, we have 
	\be\label{eq:coef of theta negative powers}
	\big[z^n\big]\theta^{-\beta}=\frac{2^{-\frac{\beta}{2}}}{\Gamma(\frac{\beta}{2})}e^nn^{\frac{\beta}{2}-1}+\frac{2^{-\frac{\beta}{2}}\cdot \frac{\sqrt{2}\beta}{3}}{\Gamma\left(\frac{\beta-1}{2}\right)}e^n n^{\frac{\beta}{2}-\frac{3}{2}}
	+\frac{2^{-\frac{\beta}{2}}(2\beta^2+\beta)}{18\Gamma\left(\frac{\beta-2}{2}\right)}e^nn^{\frac{\beta}{2}-2}+O\left(e^nn^{\frac{\beta}{2}-\frac{5}{2}}\right)\,.
	\ee
\end{lemma}
\begin{proof}
	We proceed by considering different cases for $\beta\ge1$, starting with the simplest case $\beta=1$.
	For $\theta^{-1}$, we use the identity $\theta^{-1}=1+\frac{T(z)}{1-T(z)}$ and apply \eqref{eq:T4} to extract the coefficient of $z^n$:
	\[
	\big[z^n\big]\theta^{-1}=\big[z^n\big]\big(1+\frac{T(z)}{1-T(z)}\big)\stackrel{\eqref{eq:T4}}{=}\frac{n^{n-1}}{(n-1)!}=\frac{1}{\sqrt{2\pi}}e^nn^{-\frac{1}{2}}\cdot \Big[1-\frac{1}{12n}+O\big(n^{-2}\big)\Big]\,,
	\]
	which verifies \eqref{eq:coef of theta negative powers}  for $\beta=1$ (as we verify later using the gamma function conventions).
	
	Next, we consider the case $\beta=2$.
	Applying the expansion \eqref{eq:T_alpha_beta} with $\alpha=0$ and $\beta=2$, we obtain
	\[
	\big[z^n\big]\theta^{-2}=\frac{1}{2}e^n+\frac{\sqrt{2}}{3\sqrt{\pi}}e^nn^{-\frac{1}{2}}+O\left(e^nn^{-\frac{3}{2}}\right)\,.
	\]
	
	For $\beta\ge3$, we again use 
	\eqref{eq:T_alpha_beta} with $\alpha=0$, which gives
	\[
	\big[z^n\big]\theta^{-\beta}=\frac{2^{-\frac{\beta}{2}}}{\Gamma(\frac{\beta}{2})}e^nn^{\frac{\beta}{2}-1}+\frac{2^{-\frac{\beta}{2}}\cdot \frac{\sqrt{2}\beta}{3}}{\Gamma\left(\frac{\beta-1}{2}\right)}e^n n^{\frac{\beta}{2}-\frac{3}{2}}
	+\left[ \frac{ 2^{-\frac{\beta}{2}}\beta(\beta-2)}{8\Gamma\left(\frac{\beta}{2}\right)} 
	+ \frac{c_2}{\Gamma\left(\frac{\beta-2}{2}\right)}\right]e^nn^{\frac{\beta}{2}-2}+O\left(e^nn^{\frac{\beta}{2}-\frac{5}{2}}\right)\,,
	\]	
	where $c_2=\frac{2^{-\frac{\beta}{2}}\cdot (4\beta^2-7\beta)}{36}$. To simplify this expression, we use the gamma function identity $\Gamma\left(\frac{\beta}{2}\right)=\frac{\beta-2}{2}\Gamma\left(\frac{\beta-2}{2}\right)$ (valide for $\beta\ge3$). Substituting this identity into the coefficient of $e^nn^{\frac{\beta}{2}-2}$, we can simplify the above expression to obtain \eqref{eq:coef of theta negative powers} for $\beta\ge3$:
	\[
	\big[z^n\big]\theta^{-\beta}=\frac{2^{-\frac{\beta}{2}}}{\Gamma(\frac{\beta}{2})}e^nn^{\frac{\beta}{2}-1}+\frac{2^{-\frac{\beta}{2}}\cdot \frac{\sqrt{2}\beta}{3}}{\Gamma\left(\frac{\beta-1}{2}\right)}e^n n^{\frac{\beta}{2}-\frac{3}{2}}
	+\frac{2^{-\frac{\beta}{2}}(2\beta^2+\beta)}{18\Gamma\left(\frac{\beta-2}{2}\right)}e^nn^{\frac{\beta}{2}-2}+O\left(e^nn^{\frac{\beta}{2}-\frac{5}{2}}\right)\,.
	\]

	Finally, we verify that  \eqref{eq:coef of theta negative powers} holds for $\beta\in\{1,2\}$ using the standard conventions for the gamma function: $\Gamma(0)=\infty$ and $\Gamma(-\frac{1}{2})=-2\sqrt{\pi}$. For $\beta=1$, substituting $\Gamma\big(\frac{1}{2}\big)=\sqrt{\pi}$, $\Gamma(0)=\infty$ and $\Gamma(-\frac{1}{2})=-2\sqrt{\pi}$ into \eqref{eq:coef of theta negative powers} recovers the asymptotic form for $\big[z^n\big]\theta^{-1}$ derived earlier. For $\beta=2$, using $\Gamma(1)=1$, $\Gamma\big(\frac{1}{2}\big)=\sqrt{\pi}$ and $\Gamma(0)=\infty$ recovers the expansion for $\big[z^n\big]\theta^{-2}$. Thus, \eqref{eq:coef of theta negative powers} holds for all $\beta\ge1$. 
\end{proof}

\begin{proof}[Proof of Lemma~\ref{lem: coef using theta expansion}]
	By combining the results of Lemma~\ref{lem: coef of theta positive powers} and Lemma~\ref{lem: coef of theta negative powers}, we can directly derive the desired conclusion. We proceed by considering different cases for the value of  $m$ (the highest absolute value of the negative exponents in the theta expansion of $F(z)$).
	
	\noindent
	\textbf{Case 1: $m\ge3$}.\\
	Suppose $F(z)=f_{-m}\theta^{-m}+f_{-m+1}\theta^{-m+1}+\cdots+f_t\theta^t$ (where $t>0$). By extracting the coefficient of $z^n$ from each term in $F(z)$ and substituting the asymptotic expansions from Lemmas~\ref{lem: coef of theta positive powers} and \ref{lem: coef of theta negative powers},  we obtain:
	\begin{align*}
	\big[z^n\big]F(z)&=f_{-m}\frac{2^{-\frac{m}{2}}}{\Gamma(\frac{m}{2})}e^nn^{\frac{m}{2}-1}
	+\Big[ \frac{2^{-\frac{m}{2}}\cdot \frac{\sqrt{2}m}{3}}{\Gamma\left(\frac{m-1}{2}\right)}f_{-m} +\frac{2^{-\frac{m-1}{2}}}{\Gamma(\frac{m-1}{2})}f_{-m+1}\Big]e^nn^{\frac{m}{2}-\frac{3}{2}} \\
	&\quad 
	+\Big[\frac{2^{-\frac{m}{2}}(2m^2+m)}{18\Gamma\left(\frac{m-2}{2}\right)}f_{-m} +\frac{2^{-\frac{m-1}{2}}\cdot \frac{\sqrt{2}(m-1)}{3}}{\Gamma\left(\frac{m-2}{2}\right)}f_{-m+1}+\frac{2^{-\frac{m-2}{2}}}{\Gamma(\frac{m-2}{2})}f_{-m+2}\Big]e^nn^{\frac{m}{2}-2} \\
	&\quad +O\left(e^nn^{\frac{m}{2}-\frac{5}{2}} \right)\,,
	\end{align*}
	which matches the desired expression \eqref{eq: coef using theta expansion} for $m\ge3$. 
	
	\noindent
	\textbf{Case 2: $m=1$ and $m=2$.}\\
	For $m=1$ and $m=2$, the verification of \eqref{eq: coef using theta expansion}  is  straightforward, as we now show:
	\begin{itemize}
		\item[(1)]  If $m=1$, then $F(z)=f_{-1}\theta^{-1}+f_0+f_1\theta+f_2\theta^2+\cdots+f_t\theta^t$.  Using the expansions from Lemmas~\ref{lem: coef of theta positive powers} and \ref{lem: coef of theta negative powers}, we find:
		\[
		\big[z^n\big]F(z)=f_{-1}\cdot \frac{1}{\sqrt{2\pi}}e^nn^{-\frac{1}{2}}-\big[\frac{f_{-1}}{12}+f_1\big]\cdot \frac{1}{\sqrt{2\pi}}e^nn^{-\frac{3}{2}}+O\left(e^nn^{-\frac{5}{2}}\right)\,.
		\]
		\item[(2)] If $m=2$, then $F(z)=f_{-2}\theta^{-2}+f_{-1}\theta^{-1}+f_0+f_1\theta+f_2\theta^2+\cdots+f_t\theta^t$. Similarly, substituting the relevant asymptotic expansions gives:
		\[
		\big[z^n\big]F(z)=\frac{f_{-2}}{2}e^n+\big[\frac{2}{3}f_{-2}+f_{-1}\big]\cdot \frac{1}{\sqrt{2\pi}}e^nn^{-\frac{1}{2}}+O\left(e^nn^{-\frac{3}{2}}\right)\,. \qedhere
		\]
	\end{itemize}	
	
\end{proof}

\begin{proof}[Proof of the case $k\ge0$ in Proposition~\ref{prop: asymptotic of cnnk}]
	The details of the proof of \eqref{eq: cnnk from Wright1977} in \cite{Wright1977one} were omitted; we provide these details here for readers' convenience. We proceed by considering cases based on the value of $k$, starting with the base cases $k=0,1,2$ before verifying the recursive formula for $\sigma_k$.
	
	First, the case $k=0$ has already been established in Proposition~\ref{prop: asymptotic of cnn}.
	
	\noindent
	\textbf{Case $k=1$.}\\
	From Theorem~\ref{thm: FSS2004 thm1}, we recall the expression for $W_1(z)$:
	\[
	W_1(z)=\frac{1}{24}\frac{T^4(z)\big(6-T(z)\big)}{\big[1-T(z)\big]^3}\,.
	\]
	Using the expansion \eqref{eq:T_alpha_beta} to extract the coefficient of $z^n$, we compute 
	\begin{align*}
	C_{n,n+1}&=n!\big[z^n\big]W_1(z)\\
	&=\frac{1}{24}\cdot \frac{n^n\sqrt{2\pi n}}{e^n}\big[1+\frac{1}{12n}+O(n^{-2})\big]\cdot 
	\left\{5\cdot \frac{2^{-\frac{3}{2}}}{\Gamma(3/2)}e^nn^{\frac{3}{2}-1}\big(1+O(n^{-\frac{1}{2}})\big)  \right\}\\
	&=\frac{5}{24}n^{n+1}\big(1+O(n^{-\frac{1}{2}})\big)\,.
	\end{align*}
	
	\noindent
	\textbf{Case $k=2$.}\\
	Similarly, from Theorem~\ref{thm: FSS2004 thm1}, the expression for $W_2(z)$ is:
	\[
	W_2(z)=\frac{1}{48}\frac{T^4(z)\big[2+28T(z)-23T^2(z)+9T^3(z)-T^4(z)\big]}{\big[1-T(z)\big]^6}\,.
	\]
	Applying \eqref{eq:T_alpha_beta} again, we obtain:
	\begin{align*}
	C_{n,n+2}&=n!\big[z^n\big]W_2(z)\\
	&=\frac{1}{48}\cdot \frac{n^n\sqrt{2\pi n}}{e^n}\big[1+\frac{1}{12n}+O(n^{-2})\big]\cdot 
	\left\{15\cdot \frac{2^{-\frac{6}{2}}}{\Gamma(6/2)}e^nn^{\frac{6}{2}-1}\big(1+O(n^{-\frac{1}{2}})\big)  \right\}\\
	&=\frac{5\sqrt{2\pi}}{256}n^{n+\frac{5}{2}}\big(1+O(n^{-\frac{1}{2}})\big)\,.
	\end{align*}
	
	\noindent
	\textbf{Recursive formula \eqref{eq:recursive-sigma}  for $\sigma_k$.}\\
	Next, we verify the recursive formula \eqref{eq:recursive-sigma} for $\sigma_k$. 
	For $k\ge1$, applying \eqref{eq: coef using theta expansion}  to  $W_k(z)$ as in the form of \eqref{eq: thm4 in Wrightone}, we obtain 
	\begin{align*}
	C_{n,n+k}&=n!\big[z^n\big]W_k(z)\\
	&=c_{k,-3k}\cdot n! \cdot \frac{2^{-\frac{3k}{2}}}{\Gamma\big(\frac{3k}{2}\big)}e^nn^{\frac{3k}{2}-1} \cdot\big[1+O(n^{-\frac{1}{2}})\big]\\
	&=c_{k,-3k}\sqrt{2\pi}\frac{2^{-\frac{3k}{2}}}{\Gamma\big(\frac{3k}{2}\big)}n^{n+\frac{3k-1}{2}} \cdot\big[1+O(n^{-\frac{1}{2}})\big]\,.
	\end{align*}
	Hence, by the definition of $\rho_k$ in \eqref{eq: cnnk from Wright1977}, we have
	\be\label{eq: rhok in terms of ck3k}
	\rho_k=c_{k,-3k}\sqrt{2\pi}\frac{2^{-\frac{3k}{2}}}{\Gamma\big(\frac{3k}{2}\big)}\,.
	\ee
	Now, consider the given expression for $\rho_k$:
	\[
	\rho_k=\frac{\sqrt{\pi}\cdot 2^{(1-3k)/2}\sigma_k }{\Gamma(1+\frac{3k}{2})}\,.
	\] 
	From this, one can derive the initial values of $\sigma_k$ satisfying $4\sigma_0=1$,  $16\sigma_1=5$ and $16\sigma_2=15$.
	Using the gamma function identity $\Gamma\big(1+\frac{3k}{2}\big)=\frac{3k}{2}\Gamma\big(\frac{3k}{2}\big)$, we substitute $\rho_k$ from \eqref{eq: rhok in terms of ck3k} and solve for $\sigma_k$:
	\[
	\sigma_k=\frac{3k}{2}\cdot c_{k,-3k}\,.
	\]
	We now derive the recursive formula \eqref{eq:recursive-sigma} using the recursive relation for $c_{k,-3k}$ given in \eqref{eq: recursive cks}:
	\begin{align*}
	\sigma_{k+1}&=c_{k+1,-3(k+1)}\cdot \frac{3(k+1)}{2}\\
	&\stackrel{\eqref{eq: recursive cks}}{=} \frac{3(k+1)}{2}\cdot \frac{3k}{2}c_{k,-3k}+ \frac{3(k+1)}{2}\cdot \sum_{h=1}^{k-1}\frac{3h(k-h)}{2(k+1)}c_{h,-3h}c_{k-h,-3(k-h)}\\
	&=\frac{3(k+1)}{2}\cdot\sigma_k+\sum_{h=1}^{k-1}\big(  \frac{3h}{2}\cdot c_{h,-3h} \big)\cdot \big(  \frac{3(k-h)}{2}\cdot c_{k-h,-3(k-h)}\big)\\
	&=\frac{3(k+1)}{2}\cdot\sigma_k+\sum_{h=1}^{k-1}\sigma_h\sigma_{k-h}\,. \qedhere
	\end{align*}
\end{proof}

Before proving Lemmas~\ref{lem: R_2} and \ref{lem: R_3}, we first analyze the derivatives of the exponential generating functions $W_k$.  We formalize their properties in the following lemma.
\begin{lemma}\label{lem: derivatives of W_k}	
	Recall $\theta=\theta(z)=1-T(z)$. For $W_{-1}'(z)$, we have by \eqref{eq: W-1'} that 
	\[
	W_{-1}'(z)=\frac{1-\theta}{z}\,.
	\]	
	For $k\ge1$, there exist constants $d_{k,s}$ such that 
	\be\label{eq: form of W_k'}
	zW_{k}'(z)=3kc_{k,-3k}\theta^{-3k-2}+\sum_{s=-3k+1}^{1} d_{k,s}\theta^s\,,
	\ee	
	where the constant $c_{k,-3k}$ is from \eqref{eq: thm4 in Wrightone}.
	
	\noindent For $k\ge-1$, there exist constants $\widetilde{d}_{k,s}$ such that 
	\be\label{eq: form of W_k''}
	z^2W_{k}''(z)=\widetilde{d}_{k,-3k-4}\theta^{-3k-4}+\sum_{s=-3k-3}^{1}\widetilde{d}_{k,s}\theta^s\,,
	\ee
	and $\widetilde{d}_{-1,-1}=1,\widetilde{d}_{0,-4}=1$ and $\widetilde{d}_{k,-3k-4}=(9k^2+6k)c_{k,-3k}$ for $k\ge1$.
\end{lemma}
\begin{proof}
	We 	begin by recalling explicit expressions for  the first and second derivatives of $W_{-1}(z)$ and $W_0(z)$	from \eqref{eq:W-1'''} and \eqref{eq:W0'''}:
	\[
	W_{-1}'(z)=-\theta\cdot \theta'(z)=\frac{1-\theta}{z}\,,\quad \text{and}\quad W_{-1}''(z)=\frac{(1-\theta)^2}{z^2\theta}\,,
	\]
	and 
	\[
	W_0'(z)=\frac{1}{2}\cdot \frac{(1-\theta)^3}{z\theta^2}\,,
	\quad \text{and} \quad W_0''(z)=\frac{1}{2}\cdot \frac{(1-\theta)^3(1+\theta)(2-\theta)}{z^2\theta^4}\,.
	\]
	These verifies \eqref{eq: form of W_k''} for the cases $k=-1$ and $k=0$.
	
	For $m\ge1$, we use the expansion of $W_m(z)$ given in \eqref{eq: thm4 in Wrightone}:
	\[
	W_m=\sum_{s=-3m}^{2}c_{m,s}\theta^s\,.
	\]
	To find $W_m'(z)$, we differentiate term-by-term and use $\theta'=\frac{\theta-1}{z\theta}$ from \eqref{eq: theta'}:
	\begin{align}\label{eq:Wm'}
	W_m'(z)&=\sum_{s=-3m}^{2}c_{m,s}s\theta^{s-1}\cdot \theta'\nonumber\\
	&=\frac{\theta-1}{z}\sum_{s=-3m}^{2}c_{m,s}s\theta^{s-2}\,.  
	\end{align}
	Rewriting \eqref{eq:Wm'} by isolating the leading term (corresponding to $s=-3m$, which gives the highest negative power of $\theta$), we obtain:
	\be\label{eq:Wm'2}
	W_m'(z)=\frac{1}{z}\left[3mc_{m,-3m}\theta^{-3m-2}+\sum_{s=-3m+1}^{1}d_{m,s}\theta^s  \right]\,,
	\ee
	for some constants $d_{m,s}$. Multiplying both sides by $z$ confirms \eqref{eq: form of W_k'}.
	
	To derive the form of $W_m''(z)$ for $m\ge1$, we 
	differentiate \eqref{eq:Wm'} with respect to $z$ and again use $\theta'=\frac{\theta-1}{z\theta}$ from \eqref{eq: theta'}:
	\begin{align}\label{eq:W''}
	W_m''(z)&=\frac{\theta' z-(\theta-1)}{z^2}\sum_{s=-3m}^{2}c_{m,s}s\theta^{s-2}+\frac{\theta-1}{z}
	\sum_{s=-3m}^{2}c_{m,s}s(s-2)\theta^{s-3}\theta'\nonumber\\
	&\stackrel{\eqref{eq: theta'}}{=}-\frac{(1-\theta)^2}{z^2\theta}\sum_{s=-3m}^{2}c_{m,s}s\theta^{s-2}+\frac{(1-\theta)^2}{z^2\theta}\sum_{s=-3m}^{2}c_{m,s}s(s-2)\theta^{s-3}\nonumber\\
	&=\frac{(1-\theta)^2}{z^2}\sum_{s=-3m}^2c_{m,s}\theta^{s-4}\big[s^2-2s-s\theta\big]\\
	&=\frac{1}{z^2}\left[ c_{m,-3m}(9m^2+6m)\theta^{-3m-4}+\sum_{s=-3m-3}^{1}\overline{c}_{m,s}\theta^s \right]\,,\nonumber
	\end{align}
	for some constants $\overline{c}_{m,s}$. Here, the leading term coefficient $c_{m,-3m}(9m^2+6m)$ arises from substituting $s=-3m$ into the summand, as this yields the highest negative power of $\theta$, i.e., $\theta^{-3m-4}$. 
	
	Combining this result  with the explicit expressions for $W_{-1}''(z)$ and $W_0''(z)$ (which matches the form \eqref{eq: form of W_k''} with the given constants $\widetilde{d}_{k,s}$), we conclude that \eqref{eq: form of W_k''} holds for all $k\ge-1$.
\end{proof}

\begin{proof}[Proof of Lemma~\ref{lem: R_3}]
	For $k\ge1$, define $\sT_k\coloneq \left\{(k_1,k_2,k_3)\colon k_i\ge-1,\text{ and }k_1+k_2+k_3=k-2\right\}$. Note that 
	\begin{align*}
	R_3&=(n-3)!\sum_{(k_1,k_2,k_3)\in\sT_k}\big[z^{n-3}\big]W_{k_1}'(z)W_{k_2}'(z)W_{k_3}'(z)\\
	&=(n-3)!\sum_{(k_1,k_2,k_3)\in\sT_k}\big[z^{n}\big]\big(zW_{k_1}'(z)\big)\big(zW_{k_2}'(z)\big)\big(zW_{k_3}'(z)\big)\,.
	\end{align*}
	By Lemma~\ref{lem: coef using theta expansion}, the leading order of $\big[z^n\big]\big(zW_{k_1}'(z)\big)\big(zW_{k_2}'(z)\big)\big(zW_{k_3}'(z)\big)$ comes from the lowest order of $\theta$ in the expression of the product $\big(zW_{k_1}'(z)\big)\big(zW_{k_2}'(z)\big)\big(zW_{k_3}'(z)\big)$ in terms of $\theta$. Let $H:\{-1,0,1,\ldots\}\to \bbR$ be the function given by $H(-1)=0$ and $H(s)=-3s-2$ for $s\ge0$. Then by Lemma~\ref{lem: derivatives of W_k},	the lowest order of $\theta$ in the expression of the product $\big(zW_{k_1}'(z)\big)\big(zW_{k_2}'(z)\big)\big(zW_{k_3}'(z)\big)$ in terms of $\theta$ is given by $H(k_1)+H(k_2)+H(k_3)$. 
	
	If there are at least two terms in the tuple $(k_1,k_2,k_3)\in\sT_k$ that are bigger than $-1$, say $k_1>-1$ and $k_2>-1$, then we can consider the new tuple $(-1,k_1+k_2+1,k_3)\in\sT_k$. For this new tuple, we have that 
	\begin{multline*}
	H(-1)+H(k_1+k_2+1)+H(k_3)=0-3(k_1+k_2+1)-2+H(k_3)\\
	<-3k_1-2-3k_2-2+H(k_3)=H(k_1)+H(k_2)+H(k_3)\,.
	\end{multline*}
	Thus, the leading order of $\sum_{(k_1,k_2,k_3)\in\sT_k}\big[z^{n}\big]\big(zW_{k_1}'(z)\big)\big(zW_{k_2}'(z)\big)\big(zW_{k_3}'(z)\big)$ comes from the three tuples $(-1,-1,k),(-1,k,-1)$ and $(k,-1,-1)$. 
	
	Using Lemma~\ref{lem: derivatives of W_k}, there exist constants $h_{k,s}$ such that
	\[
	\big(zW_{-1}'(z)\big)^2\big(zW_{k}'(z)\big)=3kc_{k,-3k}\theta^{-3k-2}+\sum_{s=-3k-1}^{3}h_{k,s}\theta^s\,.
	\]
	Using Lemma~\ref{lem: coef using theta expansion}, we have that 
	\[
	\big[z^n\big]\big(zW_{-1}'(z)\big)^2\big(zW_{k}'(z)\big)=\frac{c_{k,-3k}2^{-\frac{3k}{2}}}{\Gamma\left(\frac{3k}{2}\right)}e^nn^{\frac{3k}{2}}\big[1+O(n^{-\frac{1}{2}})\big]\,.
	\]
	Therefore, we have the desired estimate: 
	\begin{align*}
	R_3&=3(n-3)! \cdot \frac{c_{k,-3k}2^{-\frac{3k}{2}}}{\Gamma\left(\frac{3k}{2}\right)}e^nn^{\frac{3k}{2}}\big[1+O(n^{-\frac{1}{2}})\big]\\
	&\stackrel{\eqref{eq: n minus 3 factorial}}{=} \frac{3c_{k,-3k}\sqrt{2\pi}}{2^{\frac{3k}{2}}\Gamma\left(\frac{3k}{2}\right)}n^{n+\frac{3k-5}{2}}\big[1+O(n^{-\frac{1}{2}})\big]\\
	&\stackrel{\eqref{eq: rhok in terms of ck3k}}{=}3\rho_kn^{n+\frac{3k-5}{2}}\big[1+O(n^{-\frac{1}{2}})\big]\,. \qedhere
	\end{align*}
\end{proof}

\begin{proof}[Proof of Lemma~\ref{lem: R_2}]
	The proof is similar to that of Lemma~\ref{lem: R_3}. 
	The case $k=1$ follows from the explicit computation in the main
	paper. Indeed, for $k=1$,
	\[
	R_2
	=
	\frac{3\sqrt{2\pi}}{4}n^{n-3/2}
	+O(n^{n-2})
	=
	3\rho_0 n^{n-3/2}\bigl(1+O(n^{-1/2})\bigr),
	\]
	since $\rho_0=\sqrt{\pi/8}$. We henceforth assume that $k\ge2$.

	Note that
	\begin{align*}
	R_2&=3(n-3)!\sum_{l=-1}^{k-1}\big[z^n\big]\big(zW_l'(z)\big)\big(z^2W_{k-2-l}''(z)\big)\,.
	\end{align*}
	Let $\widetilde{H}:\{-1,0,\ldots\}\to\bbR$ be the function given by $\widetilde{H}(s)=-3s-4$. Then by Lemma~\ref{lem: derivatives of W_k}, the lowest order of $\theta$ in the expression of the product $\big(zW_l'(z)\big)\big(z^2W_{k-2-l}''(z)\big)$ in terms of $\theta$ is given by $H(l)+\widetilde{H}(k-2-l)$  (where \(H\) is the function defined in the proof of Lemma~\ref{lem: R_3}). It is straightforward to check that the minimal value of $H(l)+\widetilde{H}(k-2-l)$ is achieved at  $l=-1$. Moreover, Lemma~\ref{lem: derivatives of W_k} actually yields that there exist constants $\widetilde{h}_{k,s}$ such that 
	\[
	\big(zW_{-1}'(z)\big)\big(z^2W_{k-1}''(z)\big)
	=c_{k-1,-3(k-1)}\big[9(k-1)^2+6(k-1)\big]\theta^{-3k-1}+\sum_{s=-3k}^{2}\widetilde{h}_{k,s}\theta^s\,.
	\]
	Using Lemma~\ref{lem: coef using theta expansion}, we have that 
	\[
	\big[z^n\big]\big(zW_{-1}'(z)\big)\big(z^2W_{k-1}''(z)\big)
	=3c_{k-1,-3(k-1)}(k-1)(3k-1)\frac{2^{-\frac{3k+1}{2}}}{\Gamma\left(\frac{3k+1}{2}\right)}e^nn^{\frac{3k-1}{2}}\big[1+O(n^{-\frac{1}{2}})\big]\,.
	\]
	Therefore, for $k\ge2$, we have the desired conclusion:
	\begin{align*}
	R_2&=3(n-3)!\cdot 3c_{k-1,-3(k-1)}(k-1)(3k-1)\frac{2^{-\frac{3k+1}{2}}}{\Gamma\left(\frac{3k+1}{2}\right)}e^nn^{\frac{3k-1}{2}}\big[1+O(n^{-\frac{1}{2}})\big]\\
	&\stackrel{\eqref{eq: n minus 3 factorial}}{=}
	9(k-1)(3k-1)\frac{2^{-\frac{3k+1}{2}}}{\Gamma\left(\frac{3k+1}{2}\right)}c_{k-1,-3(k-1)}\sqrt{2\pi}
	n^{n+\frac{3k-6}{2}}\big[1+O(n^{-\frac{1}{2}})\big]\\
	&\stackrel{\eqref{eq: rhok in terms of ck3k}}{=}3\rho_{k-1}n^{n+\frac{3k-6}{2}}\big[1+O(n^{-\frac{1}{2}})\big]\,. \qedhere
	\end{align*}
\end{proof}

\section{The case $k \ge 6$ of Proposition~\ref{prop: key prop for k-excess}}

The existence of $\big\{\alpha_{k,i},\beta_{k,i},i=1,2,3  \}$ such that \eqref{eq: cnnk asym} and \eqref{eq: cnnkef asym} hold follows from arguments similar to those in Section~\ref{sec: p-NC for kUC}. The main challenge here is to show that these constants satisfy \eqref{eq: p1 kUC asym}. By Lemma~\ref{lem: coef using theta expansion}, we need the information of the first three coefficients in the expansion of $W_k$'s in terms of $\theta$. 
Thus, for simplicity, we rewrite the expression of $W_k(k\geq 1)$ given in \eqref{eq: thm4 in Wrightone} as
\be\label{eq: three terms Wk}
W_k=\kcc\theta^{-3k}+\kticc\theta^{-3k+1}+\khcc\theta^{-3k+2}+\hO(\theta^{-3k+3})\,,
\ee
where $\hO(\theta^{p})$ is an abuse of notation denoting a finite sum  of terms of the form $c\theta^s$, where $c$ is a constant and $s$ is an integer with $s\ge p$. In particular, the term $\hO(\theta^{-3k+3})$ in \eqref{eq: three terms Wk}  corresponds to  $\sum_{s=-3k+3}^{2}c_{k,s}\theta^s$ in \eqref{eq: thm4 in Wrightone}.

Since we have already proved the cases $k\in\{1,2\ldots,5\}$ for Proposition~\ref{prop: key prop for k-excess}, in the remainder of this section, we  assume $k\ge6$, and our main task is deriving the following asymptotic estimates.
\begin{proposition}\label{prop: finer asym of cnnk and cnnkef}
	For any fixed $k\ge6$, we have 
	\be\label{eq: finer cnnk asym}
	\begin{aligned}
		C_{n,n+k}
		&=\sqrt{2\pi}\,n^{\,n+\frac{3k-1}{2}}
		\Bigg\{
		\kcc\frac{2^{-\frac{3k}{2}}}{\Gamma(\frac{3k}{2})}
		+\frac{2^{-\frac{3k-1}{2}}}{\Gamma(\frac{3k-1}{2})}\bigl(k\kcc+\kticc\bigr)\,n^{-\frac12}
		\\
		&\quad+\Bigg[
		\frac{2^{-\frac{3k-2}{2}}}{\Gamma(\frac{3k-2}{2})}
		\left(\frac{k(6k+1)}{12}\kcc+\frac{3k-1}{3}\kticc+\khcc\right)
		+\frac{1}{12}\kcc\frac{2^{-\frac{3k}{2}}}{\Gamma(\frac{3k}{2})}
		\Bigg]n^{-1}
		+O\!\left(n^{-\frac32}\right)
		\Bigg\}.
	\end{aligned}
	\ee
	and 
	\be\label{eq: finer cnnkef asym}
	\begin{aligned}
		C_{n,n+k}^{e,f}&
		=\sqrt{2\pi}\,n^{\,n+\frac{3k-5}{2}}
		\Bigg\{
		\frac{9k\,\kcc}{2^{\frac{3k+2}{2}}\Gamma\!\left(\frac{3k+2}{2}\right)}
		+\frac{3(3k-1)\bigl(\kticc+k\kcc\bigr)}{2^{\frac{3k+1}{2}}\Gamma\!\left(\frac{3k+1}{2}\right)}\,n^{-\frac12}\\&
		+\frac{(9k-6)\khcc+\bigl(9k^2-9k+2\bigr)\kticc+\frac14\bigl(18k^3-9k^2+34k+37\bigr)\kcc}
		{2^{\frac{3k}{2}}\Gamma\!\left(\frac{3k}{2}\right)}\,n^{-1}
		+O\!\left(n^{-\frac32}\right)
		\Bigg\}.
	\end{aligned}
	\ee
\end{proposition}

\begin{proof}[Proof of Proposition~\ref{prop: key prop for k-excess} for $k\ge6$ assuming Proposition~\ref{prop: finer asym of cnnk and cnnkef}]
	Equations \eqref{eq: finer cnnk asym} and \eqref{eq: finer cnnkef asym} justify 	\eqref{eq: cnnk asym} and \eqref{eq: cnnkef asym}, respectively. 
	Write
	\[
	a_0=\frac{9k\,\kcc}{2^{\frac{3k+2}{2}}\Gamma\big(\frac{3k+2}{2}\big)},\quad
	a_1=\frac{3(3k-1)(k\kcc+\kticc)}{2^{\frac{3k+1}{2}}\Gamma\big(\frac{3k+1}{2}\big)},
	\]
	\[
	a_2=\frac{(9k-6)\khcc+(9k^2-9k+2)\kticc+\frac14(18k^3-9k^2+34k+37)\kcc}
	{2^{\frac{3k}{2}}\Gamma\big(\frac{3k}{2}\big)}.
	\]
	and
	\[
	b_0=\frac{\kcc}{2^{\frac{3k}{2}}\Gamma\big(\frac{3k}{2}\big)},\quad
	b_1=\frac{k\kcc+\kticc}{2^{\frac{3k-1}{2}}\Gamma\big(\frac{3k-1}{2}\big)},
	\]
	\[
	b_2=
	\frac{\frac{k(6k+1)}{12}\kcc+\frac{3k-1}{3}\kticc+\khcc}{2^{\frac{3k-2}{2}}\Gamma\big(\frac{3k-2}{2}\big)}
	+\frac{1}{12}\frac{\kcc}{2^{\frac{3k}{2}}\Gamma\big(\frac{3k}{2}\big)}.
	\]
	It is straightforward to verify that 
	\[
	a_0=3b_0,\qquad a_1=3b_1, \quad \text{and}\quad a_2-3b_2=(9k+9)b_0.
	\]
	Using these identities, we can derive the asymptotic behavior of the ratio $\frac{C_{n,n+k}^{e,f}}{C_{n,n+k}}$ and confirm that
	\eqref{eq: p1 kUC asym} holds for $k\ge6$:
	\begin{align*}
	\frac{C_{n,n+k}^{e,f}}{C_{n,n+k}}
	&=n^{-2}\cdot \frac{a_0+a_1n^{-1/2}+a_2n^{-1}+O(n^{-\frac{3}{2}})}{b_0+b_1n^{-1/2}+b_2n^{-1}+O(n^{-\frac{3}{2}})}\\
	&=n^{-2}\cdot \big[3+\frac{a_2-3b_2}{b_0}n^{-1}+O(n^{-3/2})\big]\\
	&=\frac{3}{n^{2}}+\frac{9(k+1)}{n^{3}}+O\!\left(n^{-\frac{7}{2}}\right). \qedhere
	\end{align*}
\end{proof}

\begin{proof}[Proof of \eqref{eq: finer cnnk asym}]
	This is a direct application of Lemma~\ref{lem: coef using theta expansion} and Stirling's formula. Indeed, applying Lemma~\ref{lem: coef using theta expansion} to $W_k$ in the form of   \eqref{eq: three terms Wk},  we have 
	\begin{align*}
	\big[z^n\big]W_k&=e^nn^{\frac{3k-2}{2}}\Bigg\{
	\kcc\frac{2^{-\frac{3k}{2}}}{\Gamma(\frac{3k}{2})}
	+\frac{2^{-\frac{3k-1}{2}}}{\Gamma(\frac{3k-1}{2})}\bigl(k\kcc+\kticc\bigr)\,n^{-\frac12}
	\\
	&\quad+
	\frac{2^{-\frac{3k-2}{2}}}{\Gamma(\frac{3k-2}{2})}
	\left(\frac{k(6k+1)}{12}\kcc+\frac{3k-1}{3}\kticc+\khcc\right)
	n^{-1}
	+O\!\left(n^{-\frac32}\right)\Bigg\}\,.
	\end{align*}
	Recall Stirling's formula
	\[
	n!=e^{-n}\sqrt{2\pi}n^{n+\frac{1}{2}}\left[1+\frac{1}{12n}+O\big(n^{-2}\big)\right]\,.
	\]
	Since $C_{n,n+k}=n!\big[z^n\big]W_k$, multiplying the above two equations together yields \eqref{eq: finer cnnk asym}.
\end{proof}

The proof of \eqref{eq: finer cnnkef asym} is much more involved, and 
we begin with two recursive relations for the coefficients $\kcc$ and $\kticc$.

\begin{lemma}\label{lem:Recursive derivation of the W_k coefficient}
	For $r\ge 3$, 
	define $\bC_1(r)$ and $\bC_2(r)$ as follows:
	\be\label{eq:def-of-C-1}
	\mathbf{C}_1(r) = \sum_{j=1}^{r-2} 9j(r-1-j) \kjcc \kajccr,
	\ee
	and
	\be\label{eq:def-of-C-2}
	\mathbf{C}_2(r) = \sum_{j=1}^{r-2} \Big( 3j(3(r-1-j)-1)\kjcc \kajticcr + 3(r-1-j)(3j-1)\kjticc \kajccr \Big).
	\ee
	For simplicity, we denote $\bC_1=\bC_1(k)$ and $\bC_2=\bC_2(k)$. 
	Then, for $k\ge3$, the following recursive relations hold:
	\begin{align}\label{eq:recursive relation one}
	2\kcc=3(k-1)\kacc + \frac{\mathbf{C}_1}{3k},
	\end{align}
	and
	\begin{align}\label{eq:recursive relation two}
	2(\kticc-k\kcc)=\frac{-3(k-1)(3k^2+3k+1)\kacc + (3k-1)(3k-4)\katicc - (k+1)\mathbf{C}_1 + \mathbf{C}_2}{3k-1} .
	\end{align}
	
\end{lemma}
\begin{proof}
	Our proof is based on Formula \eqref{eq:thm2 in Wrightone} given in \cite{Wright1977one}:
	\be\label{eq:Wk-Wk-1}
	(1-\theta)^{k}W_{k}=\int_{\theta}^{1}(1-t)^{k-1}\mathcal{J}_{k-1}(t)dt,
	\ee
	where $k\geq 1$ and $\mathcal{J}_{k-1}(t)$ is derived from the relation $\mathcal{J}_{k-1}(\theta(z))=J_{k-1}(z)$, with the formal power series $J_{k-1}=J_{k-1}(z)$  defined as
	\be\label{eq:defJk-1}
	2J_{k-1}\coloneq (\sD^2-3\sD-2(k-1))W_{k-1}+\sum_{h=0}^{k-1}(\sD W_h)(\sD W_{k-1-h})\,.
	\ee
	Recall that in the proof of  \cref{lem: W_k in terms of theta},  the operator 
	$\sD$ satisfies the following relation:
	\[
	\sD f(z)=\left(1-\frac{1}{\theta}\right)\frac{d f}{d\theta}=-\theta^{-1}(1-\theta)\frac{d f}{d\theta}\,, \quad \text{where }\,\frac{df}{d\theta}=\frac{f'(z)}{\theta'(z)}\,,\,\forall\,\,f\in\bbC[[z]].
	\]
	Applying this operator to $W_k$ in the form of \eqref{eq: three terms Wk}, we obtain that 
	\begin{align*}
	\sD W_k&=-\theta^{-1}(1-\theta)\left(  -3k\kcc\theta^{-3k-1} -(3k-1)\kticc\theta^{-3k}-(3k-2)\khcc\theta^{-3k+1}+\hO\big(\theta^{-3k+2}\big)\right)\\
	&=(1-\theta)\left(3k\kcc\theta^{-3k-2}+(3k-1)\kticc\theta^{-3k-1}+(3k-2)\khcc\theta^{-3k}+\hO(\theta^{-3k+1})\right)\\
	&=3k\kcc\theta^{-3k-2}+\big((3k-1)\kticc-3k\kcc\big)\theta^{-3k-1}+\big((3k-2)\khcc-(3k-1)\kticc\big)\theta^{-3k}+\hO\big(\theta^{-3k+1}\big)\,,
	\end{align*}
	and applying this operator to $\sD W_k$, we obtain that
	\begin{align*}
	\sD^2W_k&=(1-\theta^{-1})\cdot\big(-3k(3k+2)\kcc\theta^{-3k-3}\big)\nonumber
	\\&\quad +(1-\theta^{-1})\bigg[(-3k-1)\big((3k-1)\kticc-3k\kcc\big)\theta^{-3k-2}+\hO(\theta^{-3k-1})\bigg]\\
	&=3k(3k+2)\kcc\theta^{-3k-4}+\Big[(3k+1)\big((3k-1)\kticc-3k\kcc\big)-3k(3k+2)\kcc\Big]\theta^{-3k-3}+\hO(\theta^{-3k-2})\\
	&=3k(3k+2)\kcc\theta^{-3k-4}+\big((3k+1)(3k-1)\kticc-3k(6k+3)\kcc\big)\theta^{-3k-3}+\hO(\theta^{-3k-2})\,.
	\end{align*}
	Analogously, for $\sD W_{k-1}$ and $\sD^2 W_{k-1}$, we derive:
	\begin{align}\label{eq:B7}
	\sD W_{k-1}=(1-\theta)\bigg[3(k-1)\kacc\theta^{-3k+1}+(3k-4)\katicc \theta^{-3k+2}+\hO(\theta^{-3k+3})\bigg],
	\end{align}
	and
	\begin{align}\label{eq:B8}
	\sD^2W_{k-1}
	&=(9k^2-12k+3)\kacc\,\theta^{-3k-1}+\Big(-(18k^2-27k+9)\kacc+(9k^2-18k+8)\katicc\Big)\theta^{-3k}\nonumber \\
	&\quad+\hO(\theta^{-3k+1}).
	\end{align}
	Next, we analyze the sum term in the definition of  $J_{k-1}$. We split  the sum to isolate the boundary terms:
	\[
	\sum_{h=0}^{k-1}(\sD W_h)(\sD W_{k-1-h})=\sum_{h=1}^{k-2}(\sD W_h)(\sD W_{k-1-h})+2\sD W_0\sD W_{k-1}.
	\]
	By \eqref{eq:DW0} and \eqref{eq:B7}, we obtain
	\[
	2\sD W_0\sD W_{k-1}=(1-\theta)^4\bigg[3(k-1)\kacc\theta^{-3k-1}+(3k-4)\katicc \theta^{-3k}+\hO(\theta^{-3k+1})\bigg],
	\]
	and
	\begin{align*}
	&\sum_{j=1}^{k-2}\sD W_{j}\sD W_{k-1-j}\\
	&=(1-\theta)^2\sum_{j=1}^{k-2}\big(3j\kjcc\theta^{-3j-2}+(3j-1)\kjticc\theta^{-3j-1}+\hO(\theta^{-3j})\big)\nonumber
	\\&\quad \cdot\Big(3(k-1-j)\kajcc\theta^{-3(k-1-j)-2}+(3(k-1-j)-1)\kajticc\theta^{-3(k-1-j)-1}+\hO(\theta^{-3(k-1-j)})\Big)\nonumber
	\\&=(1-\theta)^2 \Bigg[ \mathbf{C}_1 \theta^{-3k-1} + \mathbf{C}_2 \theta^{-3k} + \hO(\theta^{-3k+1}) \Bigg]\nonumber
	\\&=\mathbf{C}_1 \theta^{-3k-1} + (\mathbf{C}_2-2\mathbf{C}_1) \theta^{-3k} + \hO(\theta^{-3k+1})\,.
	\end{align*}

	Substituting the above expressions for $\sD W_{k-1}$, $\sD^2 W_{k-1}$, and  the sum term into \eqref{eq:defJk-1}, we obtain
	\begin{align*}
	2J_{k-1}(z)&=\big(9k(k-1)\kacc +\bC_1\big)\theta^{-3k-1}\\
	&\quad +\bigg[(3k-1)(3k-4)\katicc-3(k-1)(6k+1)\kacc+(\bC_2-2\bC_1)\bigg]\theta^{-3k}\\
	&\quad +\hO(\theta^{-3k+1})\nonumber.
	\end{align*}
	Hence deriving from the relation $\mathcal{J}_{k-1}(\theta(z))=J_{k-1}(z)$, we find:
	\begin{align*}
	2\mathcal{J}_{k-1}(t)&=\big(9k(k-1)\kacc +\bC_1\big)t^{-3k-1}\\
	&\quad +\bigg[(3k-1)(3k-4)\katicc-3(k-1)(6k+1)\kacc+(\bC_2-2\bC_1)\bigg]t^{-3k}\\
	&\quad +\hO(t^{-3k+1})\nonumber.
	\end{align*}
	Integrating $(1-t)^{k-1}\cJ_{k-1}(t)$, we obtain:
	$$\begin{aligned}
	\int_{\theta}^{1}(1-t)^{k-1}\mathcal{J}_{k-1}(t) dt &= \left[ \frac{3(k-1)}{2}\chi_{k-1} + \frac{\mathbf{C}_1}{6k} \right] \theta^{-3k} \\
	&\quad + \frac{1}{2(3k-1)}\bigg[ -3(k-1)(3k^2+3k+1)\chi_{k-1} \\
	&\quad\quad + (3k-1)(3k-4)\tilde{\chi}_{k-1} + \mathbf{C}_2 - (k+1)\mathbf{C}_1 \bigg] \theta^{-3k+1} \\
	&\quad + \hO(\theta^{-3k+2})\,.
	\end{aligned}$$
	On the other hand, expanding $(1-\theta)^kW_k$ using \eqref{eq: three terms Wk} gives:
	\[
	(1-\theta)^kW_k=\kcc\theta^{-3k}+(\kticc-k\kcc) \theta^{-3k+1}+\hO(\theta^{-3k+2})\,.
	\]
	By equating the two expressions for $(1-\theta)^kW_k$ (from \eqref{eq:Wk-Wk-1} and the expansion above) and comparing the coefficients of $\theta^{-3k}$ and $\theta^{-3k+1}$,	we  derive the desired recursive relations:
	\begin{align*} 
	2\kcc=3(k-1)\kacc + \frac{\mathbf{C}_1}{3k}, 
	\end{align*} 
	and
	\begin{align*}
	2(\kticc-k\kcc)=\frac{-3(k-1)(3k^2+3k+1)\kacc + (3k-1)(3k-4)\katicc - (k+1)\mathbf{C}_1 + \mathbf{C}_2}{3k-1} . \quad  \qedhere
	\end{align*}
\end{proof}

Lemma~\ref{lem:Recursive derivation of the W_k coefficient} has important consequences in converting convolution terms into linear combinations of the coefficients $\kcc$ and related terms. For later use in subsequent proofs, we record two such conversions in the following lemma.
\begin{lemma}\label{lem:simplify-convultion}
	For $k\ge6$, set 
	\be\label{eq:def-S-R-2}
	S \coloneq \sum_{j=1}^{k-3} 3j\bigl[9(k-2-j)^2+6(k-2-j)\bigr]\chi_j\chi_{k-2-j}
	\ee
	and
	\be\label{eq:def-S-R-3}
	S_3\coloneq \sum_{\substack{k_1,k_2,k_3\ge1\\k_1+k_2+k_3=k-2}}
	27\,k_1k_2k_3\,\chi_{k_1}\chi_{k_2}\chi_{k_3}.
	\ee
	Then the following simplifications hold: 
	\be\label{eq:S-R-2}
	S=\frac{3}{2}(k-1)(3k-2)\bigl(2\chi_{k-1}-3(k-2)\chi_{k-2}\bigr).
	\ee
	and
	\be\label{eq:S-R-3}
	S_3
	=2\mathbf C_1-\frac{3k}{2}\mathbf C_1(k-1)-\frac{15}{4}(k-2)\chi_{k-2}.
	\ee
\end{lemma}
\begin{proof}
	Let $m:=k-2$. Then
	\[
	S=\sum_{j=1}^{m-1}\Big(27j(m-j)^2+18j(m-j)\Big)\kjcc\chi_{m-j}
	=27S_2+18S_1,
	\]
	where
	\[
	S_2\coloneq \sum_{j=1}^{m-1}j(m-j)^2\kjcc\chi_{m-j},
	\qquad
	S_1\coloneq \sum_{j=1}^{m-1}j(m-j)\kjcc\chi_{m-j}\,.
	\]
	By substituting the index $j\mapsto m-j$,   we obtain
	\[
	S_2=\sum_{j=1}^{m-1}(m-j)j^2\kjcc\chi_{m-j}.
	\]
	Adding the original expression of $S_2$ to this substituted form gives
	\[
	2S_2=\sum_{j=1}^{m-1}\bigl(j(m-j)^2+(m-j)j^2\bigr)\kjcc\chi_{m-j}
	=\sum_{j=1}^{m-1}jm(m-j)\kjcc\chi_{m-j}=mS_1,
	\]
	whence  $S_2=\frac{m}{2}S_1$. Substituting this back into the expression for $S$, we obtain
	\[
	S=\Big(\frac{27m}{2}+18\Big)S_1=\frac{9(3m+4)}{2}S_1
	=\frac{9(3k-2)}{2}\sum_{j=1}^{k-3}j(k-2-j)\chi_j\chi_{k-2-j}.
	\]
	Recall the definition of $\bC_1(r)$ from \eqref{eq:def-of-C-1}, which gives 
	\[
	\mathbf C_1(k-1)=\sum_{j=1}^{k-3}9j(k-2-j)\chi_j\chi_{k-2-j}\,.
	\]
	Thus 
	\[
	\,S=\frac{3k-2}{2}\mathbf C_1(k-1).\,
	\]
	Next, applying the recursive relation \eqref{eq:recursive relation one} to $k-1$ (instead of $k$), we have 
	\[
	\mathbf C_1(k-1)=3(k-1)\bigl(2\chi_{k-1}-3(k-2)\chi_{k-2}\bigr)\,.
	\]
	Substituting  this into the  expression for $S$ yields \eqref{eq:S-R-2}:
	\[
	S=\frac{3}{2}(k-1)(3k-2)\bigl(2\chi_{k-1}-3(k-2)\chi_{k-2}\bigr).
	\]

	We now turn to proving \eqref{eq:S-R-3}. 
	Recall the definition of $S_3$ from \eqref{eq:def-S-R-3}:
	\[
	S_3\coloneq \sum_{\substack{k_1,k_2,k_3\ge1\\k_1+k_2+k_3=k-2}}
	27\,k_1k_2k_3\,\chi_{k_1}\chi_{k_2}\chi_{k_3}.
	\]
	To simplify the triple sum, we fix $k_1$ and set  $m:=k-2-k_1$. Since $k_2,k_3\ge1$, it follows that $m\ge2$ and $1\le k_1\le k-4$. Substituting $m=k-2-k_1$ into the sum, we rewrite $S_3$ as:
	\[
	\begin{aligned}
	S_3
	&=27\sum_{k_1=1}^{k-4} k_1\chi_{k_1}
	\sum_{\substack{k_2+k_3=k-2-k_1\\k_2,k_3\ge1}}
	k_2k_3\,\chi_{k_2}\chi_{k_3}\\
	&=27\sum_{k_1=1}^{k-4} k_1\chi_{k_1}
	\sum_{k_2=1}^{m-1} k_2(m-k_2)\chi_{k_2}\chi_{m-k_2},
	\qquad(m=k-2-k_1).
	\end{aligned}
	\]
	We now relate the inner sum to $\bC_1(r)$ by recalling its definition from \eqref{eq:def-of-C-1}:
	\[
	\mathbf C_1(r)=\sum_{j=1}^{r-2}9j(r-1-j)\chi_j\chi_{r-1-j}.
	\]
	Setting $r=m+1$, this becomes
	\[
	\mathbf C_1(m+1)=\sum_{j=1}^{m-1}9j(m-j)\chi_j\chi_{m-j}
	=9\sum_{j=1}^{m-1}j(m-j)\chi_j\chi_{m-j}.
	\]
	Dividing both sides by $9$, we obtain a simplified expression for the inner sum:
	\[
	\sum_{j=1}^{m-1}j(m-j)\chi_j\chi_{m-j}
	=\frac{1}{9}\mathbf C_1(m+1).
	\]
	Since $m=k-2-k_1$, we have $m+1=k-1-k_1$. Substituting this back into the expression for $S_3$, we get
	\[
	S_3
	=3\sum_{k_1=1}^{k-4}k_1\chi_{k_1}\,\mathbf C_1(k-1-k_1).
	\]
	To eliminate $\bC_1(k-1-k_1)$, we apply the recursive relation
	\eqref{eq:recursive relation one} to $k-1-k_1$ (instead of $k$), which gives
	\[
	\mathbf C_1(k-1-k_1)=3(k-1-k_1)\bigl(2\chi_{k-1-k_1}-3(k-2-k_1)\chi_{k-2-k_1}\bigr),
	\]
	Substituting this into the expression for $S_3$ and expanding the sum, we obtain
	\be\label{eq:S_3-A-B}
	\begin{aligned}
		S_3
		&=9\sum_{k_1=1}^{k-4}k_1(k-1-k_1)\chi_{k_1}
		\Bigl(2\chi_{k-1-k_1}-3(k-2-k_1)\chi_{k-2-k_1}\Bigr)\\
		&=18\sum_{k_1=1}^{k-4}k_1(k-1-k_1)\chi_{k_1}\chi_{k-1-k_1}
		-27\sum_{k_1=1}^{k-4}k_1(k-1-k_1)(k-2-k_1)\chi_{k_1}\chi_{k-2-k_1}\\
		&=18A-27B\,,
	\end{aligned}
	\ee
	where we define $A$ and $B$ for brevity as
	\[
	A:=\sum_{k_1=1}^{k-4}k_1(k-1-k_1)\chi_{k_1}\chi_{k-1-k_1},\qquad
	B:=\sum_{k_1=1}^{k-4}k_1(k-1-k_1)(k-2-k_1)\chi_{k_1}\chi_{k-2-k_1}.
	\]
	
	We first express $A$ in terms of $\bC_1=\bC_1(k)$.
	From  the definition of $\bC_1=\bC_1(k)$
	\[
	\bC_1(k)=\sum_{j=1}^{k-2}9j(k-1-j)\chi_j\chi_{k-1-j}\,,
	\]
	the sum $A$ is the same as in the sum in $\bC_1(k)$ but missing the terms where $k_1=k-3$ and $k_1=k-2$ (since $k_1\leq k-4$). Correcting for these boundary terms, we have 
	\be\label{eq:A in S_3}
	A=\frac{\mathbf C_1}{9}-2(k-3)\chi_{k-3}\chi_2-(k-2)\chi_{k-2}\chi_1.
	\ee
	Next we express $B$ in terms of $\bC_1(k-1)$. 
	Let $m:=k-2$ and define the full sum (without the upper limit $k-4$):
	\[
	B_{\rm full}:=\sum_{j=1}^{m-1}j(m+1-j)(m-j)\chi_j\chi_{m-j}.
	\]
	Then $B$ is $B_{\rm full}$ minus the boundary term where $k_1=k-3$ (i.e., $j=m-1$):
	\[
	B=B_{\rm full}-2(k-3)\chi_{k-3}\chi_1.
	\]
	To simplify $B_{\rm full}$, we pair the term  $j$ and  $m-j$:
	\begin{align*}
	B_{\rm full}
	&=\frac12\sum_{j=1}^{m-1}\Big(j(m+1-j)(m-j)+(m-j)(j+1)j\Big)\chi_j\chi_{m-j}\\
	&=\frac{m+2}{2}\sum_{j=1}^{m-1}j(m-j)\chi_j\chi_{m-j}=\frac{k}{2}\sum_{j=1}^{k-3}j(k-2-j)\chi_j\chi_{k-2-j}\,.
	\end{align*}
	Using  the definition of $\bC_1(k-1)$ from \eqref{eq:def-of-C-1},  we have
	\[
	\sum_{j=1}^{k-3}j(k-2-j)\chi_j\chi_{k-2-j}=\frac{1}{9}\bC_1(k-1)\,,
	\]
	so substituting this into the expression for $B_{\rm full}$ gives
	\[
	B_{\rm full}=\frac{k}{18}\mathbf C_1(k-1)\,.
	\]
	Thus
	\be\label{eq:B in S_3}
	B=\frac{k}{18}\mathbf C_1(k-1)-2(k-3)\chi_{k-3}\chi_1.
	\ee
	Plugging \eqref{eq:A in S_3} and \eqref{eq:B in S_3} into \eqref{eq:S_3-A-B}, we expand and simplify:
	\[
	\begin{aligned}
	S_3
	&=2\mathbf C_1-\frac{3k}{2}\mathbf C_1(k-1)
	-36(k-3)\chi_{k-3}\chi_2-18(k-2)\chi_{k-2}\chi_1+54(k-3)\chi_{k-3}\chi_1.
	\end{aligned}
	\]
	From the explicit expressions of $W_1$ and $W_2$ given in Lemma~\ref{lem: W_k in terms of theta}, we have $\chi_1=\frac{5}{24}$ and $\chi_2=\frac{5}{16}$. Substituting these values, the  terms containing $\chi_{k-3}$ cancel out, yielding the desired conclusion \eqref{eq:S-R-3}:
	\[
	S_3=2\mathbf C_1-\frac{3k}{2}\mathbf C_1(k-1)-\frac{15}{4}(k-2)\chi_{k-2}. \qedhere
	\]
\end{proof}

The key estimates for \eqref{eq: finer cnnkef asym} are given in Lemmas~\ref{lem: more refined R_3 estimate} and \ref{lem: more refined R_2 estimate}, which are refinements of Lemmas~\ref{lem: R_3} and \ref{lem: R_2}, respectively.
\begin{lemma}\label{lem: more refined R_3 estimate}
	For $R_3$, we have the following asymptotic behaviour
	\begin{align}\label{eq:R_3refined}
	R_3
	&=\sqrt{2\pi}\,n^{\,n+\frac{3k-5}{2}}
	\Bigg\{
	\frac{9k\,\kcc}{2^{\frac{3k+2}{2}}\Gamma\!\left(\frac{3k+2}{2}\right)}
	\,+\frac{3(3k-1)\big(\kticc+k\kcc-3(k-1)\kacc\big)}{2^{\frac{3k+1}{2}}\Gamma\!\left(\frac{3k+1}{2}\right)}\,n^{-\frac12}\nonumber
	\\
	&\qquad+\frac{1}{2^{\frac{3k}{2}}\Gamma\!\left(\frac{3k}{2}\right)}
	\Bigg(
	3(3k-2)\khcc+(3k-1)(3k-2)\kticc-3(3k-4)(3k-2)\katicc\nonumber\\\nonumber
	&\qquad\qquad
	-3(k-1)(3k-4)(3k-2)\kacc
	+\frac{3}{2}(k-2)(3k-4)(3k-2)\chi_{k-2}
	\\&\qquad\qquad+\frac14(18k^3-9k^2+34k+37)\kcc
	\Bigg)\,n^{-1}
	\,+O\!\left(n^{-\frac32}\right)
	\Bigg\}.
	\end{align}
	
\end{lemma}
The proof of Lemma~\ref{lem: more refined R_3 estimate} relies on the following two claims.
\begin{claim}\label{claim: coef of R_3}
	Set 
	\be\label{eq:f-3k-2}
	f_{-3k-2}\coloneq 9k\kcc\,,
	\ee
	\be\label{eq:f-3k-1}
	f_{-3k-1}\coloneq3(3k-1)\kticc - 27k\kcc + 9(k-1)\kacc + 27 \sum_{j=1}^{k-2} j(k-1-j)\kjcc \chi_{k-1-j},
	\ee
	and 
	\begin{align}\label{eq:f-3k}
	f_{-3k} &\coloneq (9k-6)\khcc - (27k-9)\kticc + 27k\kcc + (9k-12)\katicc - 45(k-1)\kacc + \frac{9}{4}(k-2)\chi_{k-2} \nonumber\\
	&+ \sum_{j=1}^{k-2} \bigg[ 9j\kjcc(3k-3j-4)\tilde{\chi}_{k-1-j} + 9(k-1-j)\chi_{k-1-j}(3j-1)\kjticc - 81j(k-1-j)\kjcc \chi_{k-1-j} \bigg]\nonumber \\
	&+ \sum_{j=1}^{k-3} \frac{27}{2} j(k-2-j)\kjcc \chi_{k-2-j} + \sum_{\substack{k_1,k_2,k_3\geq  1\\ k_1+k_2+k_3=k-2}} 27 k_1 k_2 k_3 \chi_{k_1}\chi_{k_2}\chi_{k_3}.
	\end{align}
	With these definitions, the following expansion holds:
	\[
	\sum_{\substack{k_1,k_2,k_3\geq - 1\\k_1+k_2+k_3=k-2}}z^3 W'_{k_1}W'_{k_2}W'_{k_3}  = f_{-3k-2}\theta^{-3k-2} +f_{-3k-1}\theta^{-3k-1}+f_{-3k}\theta^{-3k}+\hO(\theta^{-3k+1})\,.
	\]
\end{claim}
\begin{claim}\label{claim: coef of R_3 simplification}
	The coefficients $f_{-3k-1}$ and $f_{-3k}$ defined in Claim~\ref{claim: coef of R_3} admit the following simplifications:
	\be\label{eq: f-3k-1 simple}
	f_{-3k-1}=3(3k-1)\kticc-9k\kcc-9(k-1)(3k-1)\kacc
	\ee
	and
	\be\label{eq: f-3k simple}
	\begin{aligned}
		f_{-3k}
		&=3(3k-2)\khcc
		-3(3k-1)\kticc
		+9k\kcc
		-3(3k-4)(3k-2)\katicc\\
		&\quad+27(k-1)(2k-1)\kacc
		+\frac{3}{2}(k-2)(3k-4)(3k-2)\chi_{k-2}.
	\end{aligned}
	\ee	
\end{claim}

\begin{proof}[Proof of Claim~\ref{claim: coef of R_3}]
	We  first decompose the summation into  distinct cases based on the values of $k_1,k_2,k_3$ (each $\ge-1$) that satisfy $k_1+k_2+k_3=k-2$:
	\begin{align}
	\sum_{ \substack{k_1,k_2,k_3\ge-1\\ k_1+k_2+k_3=k-2} }W_{k_1}'(z)W_{k_2}'(z)W_{k_3}'(z)&=3W'_{-1}W'_{-1}W'_{k}+6W'_{-1}W'_{0}W'_{k-1}+3W'_{0}W'_{0}W'_{k-2}\nonumber
	\\&\quad+\sum_{\substack{k_1,k_2,k_3\geq  1\\ k_1+k_2+k_3=k-2}}W'_{k_1}W'_{k_2}W'_{k_3}\nonumber
	\\&\quad +3\sum_{\substack{k_1,k_2\geq  1\\k_1+k_2=k-1}}W'_{k_1}W'_{k_2}W'_{-1} \nonumber
	\\&\quad +3\sum_{\substack{k_1,k_2\geq  1\\k_1+k_2=k-2}}W'_{k_1}W'_{k_2}W'_{0} .
	\end{align}
	
	Following a similar approach to the proof of Lemma~\ref{lem: R_3}, our goal is to determine the coefficients of $\theta^{-3k-2},\theta^{-3k-1}$ and $\theta^{-3k}$ in the expression of $z^3\sum_{ \substack{k_1,k_2,k_3\ge-1\\ k_1+k_2+k_3=k-2} }W_{k_1}'(z)W_{k_2}'(z)W_{k_3}'(z)$. We analyze each case of the decomposed summation separately to compute the coefficients.
	
	First, we recall the key expressions for $zW_m'$ from previous results. 
	By \eqref{eq:W0'''} we have 
	\be\label{eq:zw-0'}
	zW_0'=\frac{(1-\theta)^3}{2\theta^2}\,.
	\ee 
	By \eqref{eq:W-1'''} we have 
	\be\label{eq:zw-1'}
	zW_{-1}'=1-\theta\,.
	\ee
	Using the expansion of $W_k$ given in \eqref{eq: three terms Wk}, we rewrite \eqref{eq:Wm'}  for $k\ge1$ as 
	\be\label{eq:Wm'new}
	zW_k'=(1-\theta)\big[ 3k\kcc\theta^{-3k-2}+(3k-1)\kticc\theta^{-3k-1}+(3k-2)\khcc\theta^{-3k}+\hO(\theta^{-3k+1})\big]\,.
	\ee 
	We begin with the first case: $z^3W'_{-1}W'_{-1}W'_{k}$. 
	Using \eqref{eq:zw-1'} and \eqref{eq:Wm'new}, we expand and collect like terms in $\theta$: 
	\begin{align}
	z^3W'_{-1}W'_{-1}W'_{k}&=(1-\theta)^{3}\cdot\big[ 3k\kcc\theta^{-3k-2}+(3k-1)\kticc\theta^{-3k-1}+(3k-2)\khcc\theta^{-3k}+\hO(\theta^{-3k+1})\big]\nonumber
	\\&=3k\kcc \theta^{-3k-2} + \left[ (3k-1)\kticc - 9k\kcc \right] \theta^{-3k-1} \nonumber\\
	&\quad + \left[ (3k-2)\khcc - 3(3k-1)\kticc + 9k\kcc \right] \theta^{-3k}+\hO(\theta^{-3k+1})
	\end{align}
	
	Next, we consider the case $3\sum_{\substack{k_1,k_2\geq  1\\k_1+k_2=k-1}}z^3W'_{k_1}W'_{k_2}W'_{-1}$. Expanding this sum using \eqref{eq:zw-1'} and \eqref{eq:Wm'new}, we focus on the terms involving $\theta^{-3k-1}$ and $\theta^{-3k}$ (higher-order terms are absorbed into $\hO(\theta^{-3k+1})$):
	\[
	\sum_{\substack{k_1,k_2\geq  1\\k_1+k_2=k-1}}z^3W'_{k_1}W'_{k_2}W'_{-1}=\sum_{j=1}^{k-2} \Bigg\{ 9j(k-1-j)\kicc \chi_{k-1-j} \theta^{-3k-1} + \Rtwotwosum_j \theta^{-3k}\Bigg\}+ \hO(\theta^{-3k+1}).
	\]
	where $\Rtwotwosum_j= 3j\kicc(3k-3j-4)\tilde{\chi}_{k-1-j} + 3(k-1-j)\chi_{k-1-j}(3j-1)\kiticc - 27j(k-1-j)\kicc \chi_{k-1-j}$.
	
	For the case $6z^3W'_{-1}W'_{0}W'_{k-1}$, we use
	\eqref{eq:zw-0'}, \eqref{eq:zw-1'} and \eqref{eq:Wm'new} (applied to $k-1$) to expand and simplify:
	\[
	z^3W'_{-1}W'_{0}W'_{k-1}=\frac{3}{2}(k-1)\kacc \theta^{-3k-1} + \frac{1}{2}\left[ (3k-4)\katicc - 15(k-1)\kacc \right] \theta^{-3k} + \hO(\theta^{-3k+1})\,.
	\]
	
	We now analyze the remaining two cases involving $W_0'$.
	By \eqref{eq:zw-0'} and \eqref{eq:Wm'new} (applied to $k-2$) we have 
	\[
	z^3W'_{0}W'_{0}W'_{k-2}=\frac{3}{4}(k-2)\chi_{k-2} \theta^{-3k}+ \hO(\theta^{-3k+1}),
	\]
	and
	\[
	\sum_{\substack{k_1,k_2\geq  1\\k_1+k_2=k-2}}z^3W'_{k_1}W'_{k_2}W'_{0}=\sum_{j=1}^{k-3} \frac{9}{2} j (k-2-j) \chi_j \chi_{k-2-j} \theta^{-3k}+ \hO(\theta^{-3k+1})\,.
	\]

	Finally, we consider the triple sum over $k_1,k_2,k_3\ge1$. Using \eqref{eq:Wm'new} for each $k_i$ and expanding the product, the leading term (in $\theta^{-3k}$) is 
	\[
	\sum_{\substack{k_1,k_2,k_3\geq  1\\ k_1+k_2+k_3=k-2}}z^3W'_{k_1}W'_{k_2}W'_{k_3}=\sum_{\substack{k_1,k_2,k_3\geq  1\\ k_1+k_2+k_3=k-2}} 27 k_1 k_2 k_3 \chi_{k_1}\chi_{k_2}\chi_{k_3}\theta^{-3k}+ \hO(\theta^{-3k+1}).
	\]	
	
	Combining the results from all the cases, we collect the coefficients of $\theta^{-3k-2}$, $\theta^{-3k-1}$, and $\theta^{-3k}$ (absorbing higher-order terms into $\hO(\theta^{-3k+1})$). This gives
	\[
	\sum_{\substack{k_1,k_2,k_3\geq - 1\\k_1+k_2+k_3=k-2}}z^3 W'_{k_1}W'_{k_2}W'_{k_3}  = f_{-3k-2}\theta^{-3k-2} +f_{-3k-1}\theta^{-3k-1}+f_{-3k}\theta^{-3k}+\hO(\theta^{-3k+1})\,.
	\]
	where $f_{-3k-2},f_{-3k-1}$, and $f_{-3k}$ are precisely the coefficients defined in \eqref{eq:f-3k-2}, \eqref{eq:f-3k-1}, and \eqref{eq:f-3k}, respectively. \qedhere
\end{proof}

\begin{proof}[Proof of Claim~\ref{claim: coef of R_3 simplification}]
	First, we prove \eqref{eq: f-3k-1 simple}. Recall  the definition of $\bC_1$ from Lemma~\ref{lem:Recursive derivation of the W_k coefficient}  and the definition of $f_{-3k-1}$ from \eqref{eq:f-3k-1}. Substituting the summation in $f_{-3k-1}$ with $\bC_1$, we rewrite $f_{-3k-1}$ as 
	\[
	f_{-3k-1}=3(3k-1)\kticc - 27k\kcc + 9(k-1)\kacc +3\bC_1\,.
	\]
	Next, we use the recursive relation  \eqref{eq:recursive relation one} (for $k$), which gives
	\[
	\bC_1=3k\big[2\kcc-3(k-1)\kacc\big]
	\]
	Plugging this expression for $\bC_1$ into the rewritten $f_{-3k-1}$ and simplifying the resulting terms, we obtain \eqref{eq: f-3k-1 simple}:
	\[
	f_{-3k-1}=3(3k-1)\kticc-9k\kcc-9(k-1)(3k-1)\kacc\,.
	\]
	
	Next, we prove \eqref{eq: f-3k simple}. We start by simplifying the summation terms in the definition of $f_{-3k}$ (from \eqref{eq:f-3k}) using the 
	definitions of $\bC_1$ and $\bC_2$ given in \eqref{eq:def-of-C-1} and \eqref{eq:def-of-C-2}. First, simplify the double summation involving $\kjcc$, $\kjticc$ and $\chi_{k-1-j}$:  
	we have that 
	\[
	\sum_{j=1}^{k-2} \bigg[ 9j\kjcc(3k-3j-4)\tilde{\chi}_{k-1-j} + 9(k-1-j)\chi_{k-1-j}(3j-1)\kjticc - 81j(k-1-j)\kjcc \chi_{k-1-j} \bigg]=3\bC_2-9\bC_1
	\]
	We also simplify the single summation involving $\chi_j$ an $\chi_{k-2-j}$ by relating it to $\bC_1(k-1)$:
	\[
	\sum_{j=1}^{k-3}\frac{27}{2}j(k-2-j)\chi_j\chi_{k-2-j}=\frac{3}{2}\mathbf C_1(k-1).
	\]
	Substituting these two simplified summations, along with the identity  \eqref{eq:S-R-3} (for $S_3$), into the definition of $f_{-3k}$  (from  \eqref{eq:f-3k}),  we obtain
	\[
	\begin{aligned}
	f_{-3k}
	&=(9k-6)\khcc-(27k-9)\kticc+27k\kcc+(9k-12)\katicc-45(k-1)\kacc\\
	&\quad+3\mathbf C_2-7\mathbf C_1-\frac{3}{2}(k-1)\mathbf C_1(k-1)
	-\frac{3}{2}(k-2)\chi_{k-2}.
	\end{aligned}
	\]
	Finally, we apply the recursive relation \eqref{eq:recursive relation two} and \eqref{eq:recursive relation one} (both for $k$ and $k-1$) to eliminate $\bC_1$, $\bC_1(k-1)$ and $\bC_2$, simplifying the expression to obtain  the desired conclusion \eqref{eq: f-3k simple}. 
\end{proof}

\begin{proof}[Proof of Lemma~\ref{lem: more refined R_3 estimate}]
	Recall \eqref{eq: def of R_3}:
	\begin{align*}
	R_3&\coloneq (n-3)!\big[z^{n-3}\big]\sum_{ \substack{k_1,k_2,k_3\ge-1\\ k_1+k_2+k_3=k-2} }W_{k_1}'(z)W_{k_2}'(z)W_{k_3}'(z)\\
	&=(n-3)!\big[z^{n}\big]\sum_{ \substack{k_1,k_2,k_3\ge-1\\ k_1+k_2+k_3=k-2} }z^3W_{k_1}'(z)W_{k_2}'(z)W_{k_3}'(z)\,.
	\end{align*}
	By combining Claim~\ref{claim: coef of R_3} (which provide the $\theta$-expansion of the sum) and  \cref{lem: coef using theta expansion} (which connects the $\theta$-expansion coefficients to the $z^n$ coefficient), we obtain
	\begin{align}
	&\big[z^n\big]\sum_{ \substack{k_1,k_2,k_3\ge-1\\ k_1+k_2+k_3=k-2} }z^3 W'_{k_1}W'_{k_2}W'_{k_3} \nonumber\\
	=&\, \,\frac{f_{-3k-2}}{2^{\frac{3k+2}{2}}\Gamma\left(\frac{3k+2}{2}\right)} e^n n^{\frac{3k}{2}} \nonumber + \frac{1}{2^{\frac{3k+1}{2}}\Gamma\left(\frac{3k+1}{2}\right)} \left[ \frac{3k+2}{3} f_{-3k-2} + f_{-3k-1} \right] e^n n^{\frac{3k}{2}-1} \nonumber \\
	&\quad + \frac{1}{2^{\frac{3k+2}{2}}\Gamma\left(\frac{3k}{2}\right)} \left[ \frac{(3k+2)(6k+5)}{18} f_{-3k-2} + \frac{6k+2}{3} f_{-3k-1} + 2f_{-3k} \right] e^n n^{\frac{3k-2}{2}} \nonumber \\
	&\quad + O\left( e^n n^{\frac{3k-3}{2}} \right).
	\end{align}
	We then apply   Stirling's formula \eqref{eq: n minus 3 factorial} to approximate $(n-3)!$, multiplying it with the above coefficient expression to refine the asymptotic expansion of $R_3$:
	\begin{align}\label{eq:B.32}
	R_3&=\frac{\sqrt{2\pi}f_{-3k-2}}{2^{\frac{3k+2}{2}}\Gamma\left(\frac{3k+2}{2}\right)} n^{n + \frac{3k-5}{2}} \nonumber + \frac{\sqrt{2\pi}}{2^{\frac{3k+1}{2}}\Gamma\left(\frac{3k+1}{2}\right)} \left[ \frac{3k+2}{3} f_{-3k-2} + f_{-3k-1} \right] n^{n + \frac{3k-6}{2}} \nonumber \\
	&+ \frac{\sqrt{2\pi}}{2^{\frac{3k+2}{2}}\Gamma\left(\frac{3k}{2}\right)} \left[ \frac{(3k+2)(6k+5)}{18} f_{-3k-2} + \frac{6k+2}{3} f_{-3k-1} + 2f_{-3k} +\frac{37}{18k}f_{-3k-2}\right] n^{n + \frac{3k-7}{2}} \nonumber \\
	&+ O\left(n^{n + \frac{3k-8}{2}} \right).
	\end{align}
	Finally, we substitute  $f_{-3k-2}=9k\kcc$, along with the simplified expression of $f_{-3k-1}$ and $f_{-3k}$ given in \eqref{eq: f-3k-1 simple} and \eqref{eq: f-3k simple}, into \eqref{eq:B.32}. Simplifying the resulting terms yields the desired conclusion of Lemma~\ref{lem: more refined R_3 estimate}:
	\begin{align*}
	R_3
	&=\sqrt{2\pi}\,n^{\,n+\frac{3k-5}{2}}
	\Bigg\{
	\frac{9k\,\kcc}{2^{\frac{3k+2}{2}}\Gamma\!\left(\frac{3k+2}{2}\right)}
	\,+\frac{3(3k-1)\big(\kticc+k\kcc-3(k-1)\kacc\big)}{2^{\frac{3k+1}{2}}\Gamma\!\left(\frac{3k+1}{2}\right)}\,n^{-\frac12}\nonumber
	\\
	&\qquad+\frac{1}{2^{\frac{3k}{2}}\Gamma\!\left(\frac{3k}{2}\right)}
	\Bigg(
	3(3k-2)\khcc+(3k-1)(3k-2)\kticc-3(3k-4)(3k-2)\katicc\nonumber\\\nonumber
	&\qquad\qquad
	-3(k-1)(3k-4)(3k-2)\kacc
	+\frac{3}{2}(k-2)(3k-4)(3k-2)\chi_{k-2}
	\\&\qquad\qquad+\frac14(18k^3-9k^2+34k+37)\kcc
	\Bigg)\,n^{-1}
	\,+O\!\left(n^{-\frac32}\right)
	\Bigg\}. \qedhere
	\end{align*}

\end{proof}

\begin{lemma}\label{lem: more refined R_2 estimate}
	For $R_2$, 	we have the following asymptotic behavior:
	\begin{align}\label{eq:R_2refined}
	R_2
	&=3\sqrt{2\pi}\,n^{\,n+\frac{3k}{2}-3}
	\Bigg\{
	\frac{3(k-1)(3k-1)\kacc}{2^{\frac{3k+1}{2}}\Gamma\!\left(\frac{3k+1}{2}\right)}\nonumber
	\\
	&\qquad\qquad
	+\frac{(3k-4)(3k-2)}{2^{\frac{3k}{2}}\Gamma\!\left(\frac{3k}{2}\right)}
	\Bigg((k-1)\kacc+\katicc-\frac{3}{2}(k-2)\chi_{k-2}\Bigg)\,n^{-\frac12}
	+O(n^{-1})
	\Bigg\}.
	\end{align}.
\end{lemma}
\begin{proof}
	Recall the definition of $R_2$ from \eqref{eq: def of R_2}:
	\begin{align*}
	R_2&=3(n-3)!\big[z^{n-3}\big]\sum_{j=-1}^{k-1}W_j'W_{k-2-j}''\\
	&=3(n-3)!\big[z^n\big]\sum_{j=-1}^{k-1}z^3W_j'W_{k-2-j}''\,.
	\end{align*}
	Due to the special forms of $W_0$ and $W_{-1}$ (which differ from the general form of $W_m$ for $m\ge1$ as in \eqref{eq: three terms Wk}), we decompose the summation into distinct cases to do the computation:
	\begin{align}\label{eq:R-2-sum-decomp}
	\sum_{j=-1}^{k-1}z^3W_j'W_{k-2-j}''
	&=\sum_{j=1}^{k-3}z^3W_j'W_{k-2-j}''+z^3W_{-1}'W_{k-1}''+z^3W_0'W_{k-2}''+z^3W_{k-2}'W_0''+z^3W_{k-1}'W_{-1}''\,.
	\end{align}
	We first	recall the explicit expressions for $W_{-1}'$, $W_{-1}''$, $W_0'$ and $W_0''$ from  \eqref{eq:W-1'''} and  \eqref{eq:W0'''}:
	\[
	W_{-1}'(z)=-\theta\cdot \theta'(z)=\frac{1-\theta}{z}\,,\quad \text{and}\quad W_{-1}''(z)=\frac{(1-\theta)^2}{z^2\theta}\,,
	\]
	and 
	\[
	W_0'(z)=\frac{1}{2}\cdot \frac{(1-\theta)^3}{z\theta^2}\,,
	\quad \text{and} \quad W_0''(z)=\frac{1}{2}\cdot \frac{(1-\theta)^3(1+\theta)(2-\theta)}{z^2\theta^4}\,.
	\]
	For $m\ge1$, we use the expansion of $W_m$ given in \eqref{eq: three terms Wk} and the expression for $W_m''$ from \eqref{eq:W''}  to derive a $\theta$-expansion form of $z^2W_m''$:
	\begin{align}\label{eq:z2W''}
	z^2W''_{m}&=(1-\theta)^2\Bigg[ \chi_m\big(9m^2+6m+3m\theta\big)\theta^{-3m-4}+\widetilde{\chi}_m\big((3m-1)^2+2(3m-1)+(3m-1)\theta\big)\theta^{-3m-3}\nonumber
	\\ &+\widehat{\chi}_m\big((3m-2)^2+2(3m-2)+(3m-2)\theta\big)\theta^{-3m-2}+\hO(\theta^{-3m-1})\Bigg]\nonumber\\
	&=(1-\theta)^2\Bigg[\chi_m(9m^2+6m)\theta^{-3m-4}+\Big((3m-1)(3m+1)\widetilde{\chi}_m+3m\chi_m\Big)\theta^{-3m-3}+\hO(\theta^{-3m-2})\Bigg].
	\end{align}
	We now analyze each term of the decomposed summation \eqref{eq:R-2-sum-decomp} separately, starting with $z^3W'_{-1}W''_{k-1}$.
	Applying \eqref{eq:z2W''} to $m=k-1$ and multiplying by $zW_{-1}'=1-\theta$, we expand and simplify to obtain
	\begin{align}
	z^3W'_{-1}W''_{k-1} &= (9k^2-12k+3)\kacc \theta^{-3k-1} \nonumber \\
	& + \Big[ (9k^2-18k+8)\katicc - (27k^2-39k+12)\kacc \Big] \theta^{-3k}  + \hO(\theta^{-3k+1}) .
	\end{align}
	Next, we compute $z^3W'_{0}W''_{k-2}$ by
	applying \eqref{eq:z2W''} to $m=k-2$ and  multiplying by $zW_{0}'=\frac{(1-\theta)^3}{2\theta^2}$:
	\begin{align}
	z^3W'_{0}W''_{k-2} = \frac{1}{2}(9k^2-30k+24)\chi_{k-2} \theta^{-3k}  + \hO(\theta^{-3k+1})\,.
	\end{align}
	For $	z^3W'_{k-2}W''_{0}$, we use 
	\eqref{eq:Wm'new} (for $m=k-2$) and the expression $z^2W_0''=\frac{1}{2}\cdot \frac{(1-\theta)^3(1+\theta)(2-\theta)}{\theta^4}$ to obtain  
	\begin{align}
	z^3W'_{k-2}W''_{0} = (3k-6)\chi_{k-2}\theta^{-3k}+\hO(\theta^{-3k+1})\,.
	\end{align}
	We then compute $z^3W'_{k-1}W''_{-1}$ using 
	\eqref{eq:Wm'new} (for $m=k-1$) and $z^2W_{-1}''=\frac{(1-\theta)^2}{\theta}$:
	\begin{align}
	z^3W'_{k-1}W''_{-1} = 3(k-1)\kacc \theta^{-3k} 
	+ \hO(\theta^{-3k+1})\,.
	\end{align}
	Finally, we handle the summation over $j=1$ to $k-3$ by 
	applying \eqref{eq:Wm'new} (for $zW_j'$) and \eqref{eq:z2W''} (for $z^2W_{k-2-j}''$). focusing on the leading term in $\theta^{-3k}$:
	\begin{align}
	\sum_{j=1}^{k-3}z^3W_j'W_{k-2-j}''&=\sum_{j=1}^{k-3} 3j\big[9(k-2-j)^2+6(k-2-j)\big]\chi_j \chi_{k-2-j} \theta^{-3k}+ \hO(\theta^{-3k+1})\,.
	\end{align}
	Substituting all these computed terms into the decomposed summation  \eqref{eq:R-2-sum-decomp}, we collect the coefficients of $\theta^{-3k-1}$ and $\theta^{-3k}$ (absorbing higher-order terms into $\hO(\theta^{-3k+1})$):
	\begin{align}
	\sum_{j=-1}^{k-1}z^3W_j'W_{k-2-j}'' &= (9k^2-12k+3)\kacc \theta^{-3k-1} \nonumber \\
	& \quad + \Bigg\{ (9k^2-18k+8)\katicc - (27k^2-42k+15)\kacc + \frac{3}{2}(3k^2-8k+4)\chi_{k-2} \nonumber \\
	& \qquad \quad + \sum_{j=1}^{k-3} 3j\big[9(k-2-j)^2+6(k-2-j)\big]\kjcc \chi_{k-2-j} \Bigg\} \theta^{-3k} \nonumber + \hO(\theta^{-3k+1}).
	\end{align}
	
	To find the coefficient of $z^n$ of the above sum, we apply
	Lemma~\ref{lem: coef using theta expansion} to obtain 	
	\begin{align}
	\left[z^n\right] \sum_{j=-1}^{k-1}z^3W_j'W_{k-2-j}'' &= \frac{(9k^2-12k+3)\kacc}{2^{\frac{3k+1}{2}}\Gamma\left(\frac{3k+1}{2}\right)} e^n n^{\frac{3k-1}{2}} \nonumber \\
	& + \frac{1}{2^{\frac{3k}{2}}\Gamma\left(\frac{3k}{2}\right)} \Bigg\{ (9k^3 - 36k^2 + 41k - 14)\kacc \nonumber \\
	& \quad + (9k^2-18k+8)\katicc + \frac{3}{2}(3k^2-8k+4)\chi_{k-2} \nonumber \\
	& \quad + \sum_{j=1}^{k-3} 3j\big[9(k-2-j)^2+6(k-2-j)\big]\kjcc \chi_{k-2-j} \Bigg\} e^n n^{\frac{3k-2}{2}} \nonumber  + O\left(e^n n^{\frac{3k-3}{2}}\right)\,.
	\end{align}
	We then multiply this coefficient with $3(n-3)!$, using  Stirling's formula \eqref{eq: n minus 3 factorial} to approximate $(n-3)!$, resulting in the refined asymptotic expansion of $R_2$:  
	\begin{align}
	R_2&= 3\sqrt{2\pi} n^{n + \frac{3k}{2} - 3} \Bigg\{ \nonumber  \quad \frac{(9k^2-12k+3)\kacc}{2^{\frac{3k+1}{2}}\Gamma\left(\frac{3k+1}{2}\right)} \nonumber \\
	& + \frac{1}{2^{\frac{3k}{2}}\Gamma\left(\frac{3k}{2}\right)} \Bigg[ (9k^3 - 36k^2 + 41k - 14)\kacc \nonumber \\
	& \quad + (9k^2-18k+8)\katicc + \frac{3}{2}(3k^2-8k+4)\chi_{k-2} \nonumber \\
	& \quad + \sum_{j=1}^{k-3} 3j\big[9(k-2-j)^2+6(k-2-j)\big]\kjcc \chi_{k-2-j}\Bigg] n^{-\frac{1}{2}} \nonumber  + O(n^{-1}) \Bigg\}.
	\end{align}
	Recall that we already established the identity  \eqref{eq:S-R-2},	 which simplifying the remaining summation:
	\[
	\sum_{j=1}^{k-3} 3j\big[9(k-2-j)^2+6(k-2-j)\big]\kjcc \chi_{k-2-j}=\frac{3}{2}(k-1)(3k-2)\bigl(2\chi_{k-1}-3(k-2)\chi_{k-2}\bigr)\,.
	\]
	Plugging this simplified summation into the previous expression and simplifying the resulting terms, we obtain the desired conclusion \eqref{eq:R_2refined}:
	\[
	\begin{aligned}
	R_2
	&=3\sqrt{2\pi}\,n^{\,n+\frac{3k}{2}-3}
	\Bigg\{
	\frac{3(k-1)(3k-1)\kacc}{2^{\frac{3k+1}{2}}\Gamma\!\left(\frac{3k+1}{2}\right)}
	\\
	&\qquad\qquad
	+\frac{(3k-4)(3k-2)}{2^{\frac{3k}{2}}\Gamma\!\left(\frac{3k}{2}\right)}
	\Bigg((k-1)\kacc+\katicc-\frac{3}{2}(k-2)\chi_{k-2}\Bigg)\,n^{-\frac12}
	+O(n^{-1})
	\Bigg\}.
	\end{aligned} \qedhere
	\]
\end{proof}

\begin{proof}[Proof of Proposition~\ref{prop: finer asym of cnnk and cnnkef}]
	We have already proved \eqref{eq: finer cnnk asym};  it now remains to establish the refined asymptotic expansion \eqref{eq: finer cnnkef asym}. 
	
	Recall from \eqref{eq: orders of cnnkef} that $C_{n,n+k}^{e,f}\asymp n^{\,n+\frac{3k-5}{2}}$, and  from \eqref{eq: 4.37new} we have the estimate
	\[
	C_{n,n+k}^{e,f}=R_1+R_2+R_3+O\big(n^{\,n+\frac{3k-8}{2}}\big).
	\]
	We first compute the asymptotic expansion of $R_1$ using 
	\cref{lem: coef using theta expansion} and the expression of $W_{k-2}$ in the form of \eqref{eq: three terms Wk}:
	\begin{align}
	R_1&=(1-\frac{4(n+k-2)}{n(n-1)})n![z^n]W_{k-2}(z)\nonumber
	\\&=(1-\frac{4(n+k-2)}{n(n-1)})\frac{\kbcc n!e^nn^{\frac{3k-8}{2}}}{\Gamma(\frac{3k-6}{2})2^{\frac{3k-6}{2}}}\left[1+O(n^{-\frac{1}{2}})\right]\nonumber
	\\&=(3k-6)(3k-4)(3k-2)\frac{\sqrt{2\pi}\kbcc n^{n+\frac{3k-7}{2}}}{\Gamma(\frac{3k}{2})2^{\frac{3k}{2}}}+O(n^{n+\frac{3k-8}{2}})\,.
	\end{align}
	Next, we sum the asymptotic expansions of $R_1$, $R_2$ (from \eqref{eq:R_2refined}), and $R_3$ (from \eqref{eq:R_3refined}). Notably, the coefficients of the terms involving  $\kacc$, $\kbcc$ and $\katicc$ cancel out completely---an unexpected but key simplification.  This summation yields 
	the desired refined asymptotic expansion \eqref{eq: finer cnnkef asym}:
	\[
	\begin{aligned}
	C_{n,n+k}^{e,f}&
	=\sqrt{2\pi}\,n^{\,n+\frac{3k-5}{2}}
	\Bigg\{
	\frac{9k\,\kcc}{2^{\frac{3k+2}{2}}\Gamma\!\left(\frac{3k+2}{2}\right)}
	+\frac{3(3k-1)\bigl(\kticc+k\kcc\bigr)}{2^{\frac{3k+1}{2}}\Gamma\!\left(\frac{3k+1}{2}\right)}\,n^{-\frac12}
	\\&
	\quad +\frac{(9k-6)\khcc+\bigl(9k^2-9k+2\bigr)\kticc+\frac14\bigl(18k^3-9k^2+34k+37\bigr)\kcc}
	{2^{\frac{3k}{2}}\Gamma\!\left(\frac{3k}{2}\right)}\,n^{-1}
	+O\!\left(n^{-\frac32}\right)
	\Bigg\}\,. \qedhere
	\end{aligned}
	\]
\end{proof}

%\bibliography{kUCsup_ref}
%\bibliographystyle{plain}
%
%\end{document}

%% If the supplement has its own bibliography and it is not
%% already included in kUC_supplement.tex:
%%
%\bibliographystyle{plain}
%\bibliography{kUCsup_ref}

\begin{thebibliography}{10}

\bibitem{Bollobas2001_2nd_edition}
B\'ela Bollob\'as.
\newblock {\em Random graphs}, volume~73 of {\em Cambridge Studies in Advanced
  Mathematics}.
\newblock Cambridge University Press, Cambridge, second edition, 2001.

\bibitem{Borcea_Branden_Liggett2009}
Julius Borcea, Petter Br\"and\'en, and Thomas~M. Liggett.
\newblock Negative dependence and the geometry of polynomials.
\newblock {\em J. Amer. Math. Soc.}, 22(2):521--567, 2009.

\bibitem{Branden_Huh2020}
Petter Br\"and\'en and June Huh.
\newblock Lorentzian polynomials.
\newblock {\em Ann. of Math. (2)}, 192(3):821--891, 2020.

\bibitem{ER1959}
P.~Erd\H{o}s and A.~R\'enyi.
\newblock On random graphs. {I}.
\newblock {\em Publ. Math. Debrecen}, 6:290--297, 1959.

\bibitem{ES1979}
Paul Erd{\H{o}}s and Joel Spencer.
\newblock Evolution of the {{\(n\)}}-cube.
\newblock {\em Comput. Math. Appl.}, 5:33--39, 1979.

\bibitem{Fortuin_Kasteleyn_Ginibre1971}
C.~M. Fortuin, P.~W. Kasteleyn, and J.~Ginibre.
\newblock Correlation inequalities on some partially ordered sets.
\newblock {\em Comm. Math. Phys.}, 22:89--103, 1971.

\bibitem{Grimmett_Winkler2004}
G.~R. Grimmett and S.~N. Winkler.
\newblock Negative association in uniform forests and connected graphs.
\newblock {\em Random Structures Algorithms}, 24(4):444--460, 2004.

\bibitem{Grimmett2006}
Geoffrey Grimmett.
\newblock {\em The random-cluster model}, volume 333 of {\em Grundlehren der
  mathematischen Wissenschaften [Fundamental Principles of Mathematical
  Sciences]}.
\newblock Springer-Verlag, Berlin, 2006.

\bibitem{Kahn2000}
Jeff Kahn.
\newblock A normal law for matchings.
\newblock {\em Combinatorica}, 20(3):339--391, 2000.

\bibitem{LP2016}
Russell Lyons and Yuval Peres.
\newblock {\em Probability on Trees and Networks}, volume~42 of {\em Cambridge
  Series in Statistical and Probabilistic Mathematics}.
\newblock Cambridge University Press, New York, 2016.
\newblock Available at \url{http://rdlyons.pages.iu.edu/}.

\bibitem{Pemantle2000}
Robin Pemantle.
\newblock Towards a theory of negative dependence.
\newblock volume~41, pages 1371--1390. 2000.
\newblock Probabilistic techniques in equilibrium and nonequilibrium
  statistical physics.

\bibitem{Stark2011}
Dudley Stark.
\newblock The edge correlation of random forests.
\newblock {\em Ann. Comb.}, 15(3):529--539, 2011.

\end{thebibliography}


\begin{thebibliography}{10}

\bibitem{Borcea_Branden_Liggett2009}
Julius Borcea, Petter Br\"and\'en, and Thomas~M. Liggett.
\newblock Negative dependence and the geometry of polynomials.
\newblock {\em J. Amer. Math. Soc.}, 22(2):521--567, 2009.

\bibitem{FM1992}
Tom{\'a}s Feder and Milena Mihail.
\newblock Balanced matroids.
\newblock In {\em Proceedings of the twenty-fourth annual ACM symposium on
  Theory of computing}, pages 26--38, 1992.

\bibitem{Grimmett_Winkler2004}
G.~R. Grimmett and S.~N. Winkler.
\newblock Negative association in uniform forests and connected graphs.
\newblock {\em Random Structures Algorithms}, 24(4):444--460, 2004.

\bibitem{Grimmett2006}
Geoffrey Grimmett.
\newblock {\em The random-cluster model}, volume 333 of {\em Grundlehren der
  mathematischen Wissenschaften [Fundamental Principles of Mathematical
  Sciences]}.
\newblock Springer-Verlag, Berlin, 2006.

\bibitem{Gupta2026UF}
Anish Gupta.
\newblock Edges of the uniform random forest of $k_n$ are pairwise negatively
  correlated for every $n$.
\newblock {\em arXiv preprint arXiv:2609.26134}, 2026.

\bibitem{HSW2022}
June Huh, Benjamin Schr\"oter, and Botong Wang.
\newblock Correlation bounds for fields and matroids.
\newblock {\em J. Eur. Math. Soc. (JEMS)}, 24(4):1335--1351, 2022.

\bibitem{Dev_Proschan1983}
Kumar Joag-Dev and Frank Proschan.
\newblock Negative association of random variables, with applications.
\newblock {\em Ann. Statist.}, 11(1):286--295, 1983.

\bibitem{Kirchhoff1847}
Gustav Kirchhoff.
\newblock Ueber die aufl{\"o}sung der gleichungen, auf welche man bei der
  untersuchung der linearen vertheilung galvanischer str{\"o}me gef{\"u}hrt
  wird.
\newblock {\em Annalen der Physik}, 148(12):497--508, 1847.

\bibitem{Liu_Chow1981}
C.~J. Liu and Yutze Chow.
\newblock Enumeration of forests in a graph.
\newblock {\em Proc. Amer. Math. Soc.}, 83(3):659--662, 1981.

\bibitem{LP2016}
Russell Lyons and Yuval Peres.
\newblock {\em Probability on Trees and Networks}, volume~42 of {\em Cambridge
  Series in Statistical and Probabilistic Mathematics}.
\newblock Cambridge University Press, New York, 2016.
\newblock Available at \url{http://rdlyons.pages.iu.edu/}.

\bibitem{Moon1970}
J.~W. Moon.
\newblock {\em Counting labelled trees}, volume No. 1 of {\em Canadian
  Mathematical Monographs}.
\newblock Canadian Mathematical Congress, Montreal, QC, 1970.
\newblock From lectures delivered to the Twelfth Biennial Seminar of the
  Canadian Mathematical Congress (Vancouver, 1969).

\bibitem{Myrvold1992}
Wendy Myrvold.
\newblock Counting {$k$}-component forests of a graph.
\newblock {\em Networks}, 22(7):647--652, 1992.

\bibitem{Pemantle2000}
Robin Pemantle.
\newblock Towards a theory of negative dependence.
\newblock volume~41, pages 1371--1390. 2000.
\newblock Probabilistic techniques in equilibrium and nonequilibrium
  statistical physics.

\bibitem{TZ2026}
Pengfei Tang and Zibo Zhang.
\newblock Pairwise negative correlation for uniform spanning subgraphs of the
  complete graph.
\newblock {\em arXiv preprint arXiv:2603.10738}, 2026.

\end{thebibliography}


\begin{thebibliography}{10}

\bibitem{FM1992}
Tom{\'a}s Feder and Milena Mihail.
\newblock Balanced matroids.
\newblock In {\em Proceedings of the twenty-fourth annual ACM symposium on
  Theory of computing}, pages 26--38, 1992.

\bibitem{FGKP1995}
Philippe Flajolet, Peter~J. Grabner, Peter Kirschenhofer, and Helmut Prodinger.
\newblock On {Ramanujan}'s {{\(Q\)}}-function.
\newblock {\em J. Comput. Appl. Math.}, 58(1):103--116, 1995.

\bibitem{FO1990}
Philippe Flajolet and Andrew Odlyzko.
\newblock Singularity analysis of generating functions.
\newblock {\em SIAM J. Discrete Math.}, 3(2):216--240, 1990.

\bibitem{FSS2004}
Philippe Flajolet, Bruno Salvy, and Gilles Schaeffer.
\newblock Airy phenomena and analytic combinatorics of connected graphs.
\newblock {\em Electron. J. Combin.}, 11(1):Research Paper 34, 30, 2004.

\bibitem{FS2009}
Philippe Flajolet and Robert Sedgewick.
\newblock {\em Analytic combinatorics}.
\newblock Cambridge: Cambridge University Press, 2009.

\bibitem{Grimmett_Winkler2004}
G.~R. Grimmett and S.~N. Winkler.
\newblock Negative association in uniform forests and connected graphs.
\newblock {\em Random Structures Algorithms}, 24(4):444--460, 2004.

\bibitem{Grimmett2006}
Geoffrey Grimmett.
\newblock {\em The random-cluster model}, volume 333 of {\em Grundlehren der
  mathematischen Wissenschaften [Fundamental Principles of Mathematical
  Sciences]}.
\newblock Springer-Verlag, Berlin, 2006.

\bibitem{HSW2022}
June Huh, Benjamin Schr\"oter, and Botong Wang.
\newblock Correlation bounds for fields and matroids.
\newblock {\em J. Eur. Math. Soc. (JEMS)}, 24(4):1335--1351, 2022.

\bibitem{LP2016}
Russell Lyons and Yuval Peres.
\newblock {\em Probability on Trees and Networks}, volume~42 of {\em Cambridge
  Series in Statistical and Probabilistic Mathematics}.
\newblock Cambridge University Press, New York, 2016.
\newblock Available at \url{http://rdlyons.pages.iu.edu/}.

\bibitem{Pemantle2000}
Robin Pemantle.
\newblock Towards a theory of negative dependence.
\newblock volume~41, pages 1371--1390. 2000.
\newblock Probabilistic techniques in equilibrium and nonequilibrium
  statistical physics.

\bibitem{Renyi1959b}
Alfr{\'e}d R{\'e}nyi.
\newblock On connected graphs. {I}.
\newblock {\em Publ. Math. Inst. Hung. Acad. Sci.}, 4:385--388, 1959.

\bibitem{Stark2011}
Dudley Stark.
\newblock The edge correlation of random forests.
\newblock {\em Ann. Comb.}, 15(3):529--539, 2011.

\bibitem{TZ2026}
Pengfei Tang and Zibo Zhang.
\newblock Pairwise negative correlation for uniform spanning subgraphs of the
  complete graph.
\newblock {\em arXiv preprint arXiv:2603.10738}, 2026.

\bibitem{Wright1977one}
E.~M. Wright.
\newblock The number of connected sparsely edged graphs.
\newblock {\em J. Graph Theory}, 1(4):317--330, 1977.

\bibitem{Wright1978two}
E.~M. Wright.
\newblock The number of connected sparsely edged graphs. {II}. {S}mooth graphs
  and blocks.
\newblock {\em J. Graph Theory}, 2(4):299--305, 1978.

\bibitem{Wright1980three}
E.~M. Wright.
\newblock The number of connected sparsely edged graphs. {III}. {A}symptotic
  results.
\newblock {\em J. Graph Theory}, 4(4):393--407, 1980.

\end{thebibliography}


\begin{thebibliography}{1}

\bibitem{FGKP1995}
Philippe Flajolet, Peter~J. Grabner, Peter Kirschenhofer, and Helmut Prodinger.
\newblock On {Ramanujan}'s {{\(Q\)}}-function.
\newblock {\em J. Comput. Appl. Math.}, 58(1):103--116, 1995.

\bibitem{FS2009}
Philippe Flajolet and Robert Sedgewick.
\newblock {\em Analytic combinatorics}.
\newblock Cambridge: Cambridge University Press, 2009.

\bibitem{Stark2011}
Dudley Stark.
\newblock The edge correlation of random forests.
\newblock {\em Ann. Comb.}, 15(3):529--539, 2011.

\bibitem{TZ2026}
Pengfei Tang and Zibo Zhang.
\newblock Pairwise negative correlation for uniform spanning subgraphs of the
  complete graph.
\newblock {\em arXiv preprint arXiv:2603.10738}, 2026.

\bibitem{Wright1977one}
E.~M. Wright.
\newblock The number of connected sparsely edged graphs.
\newblock {\em J. Graph Theory}, 1(4):317--330, 1977.

\end{thebibliography}

\putbib[NCUS_ref]

\end{bibunit}

\end{document}